\documentclass[11pt]{article}
 
\usepackage[sort,numbers]{natbib}
\usepackage[dvipdfmx]{graphicx}
\usepackage{amsmath,amssymb,amsfonts,amsthm}
\usepackage{lscape}
\usepackage{arydshln}
\usepackage[pagewise,mathlines]{lineno}

\usepackage[varg]{txfonts}
\renewcommand*{\baselinestretch}{1.25}

\usepackage[bookmarks=true,bookmarksnumbered=true,colorlinks=true,allcolors=blue]{hyperref}

\usepackage{geometry}
\usepackage{indentfirst}

\usepackage{enumitem,color}

\newtheorem{theorem}{Theorem}[section]
\newtheorem{lemma}{Lemma}[section]
\newtheorem{proposition}{Proposition}[section]
\newtheorem{corollary}{Corollary}[section]
\theoremstyle{definition}

\newtheorem*{rmk*}{Remark}
\newtheorem{rmk}{Remark}[section]

\renewcommand*\proofname{\upshape{\bfseries{Proof}}}

\allowdisplaybreaks

\numberwithin{equation}{section}

\makeatletter
    \renewcommand*{\section}{\@startsection{section}{1}{\z@}%
    {6pt}{3pt}{\reset@font\normalsize\bfseries}}
\makeatother
    
\makeatletter
    \renewcommand*{\subsection}{\@startsection{subsection}{2}{\z@}%
    {6pt}{3pt}{\reset@font\normalsize\mdseries\itshape}}
\makeatother

\makeatletter
    \renewcommand*{\subsubsection}{\@startsection{subsubsection}{3}{\z@}%
    {3pt}{3pt}{\reset@font\normalsize\mdseries\itshape}}
\makeatother

\makeatletter
\def\@listi{\leftmargin\leftmargini
  \topsep=.5\baselineskip 
  \partopsep=0pt \parsep=0pt \itemsep=0pt}
\let\@listI\@listi
\@listi
\def\@listii{\leftmargin\leftmarginii
  \labelwidth\leftmarginii \advance\labelwidth-\labelsep
  \topsep=0pt \partopsep=0pt \parsep=0pt \itemsep=0pt}
\def\@listiii{\leftmargin\leftmarginiii
  \labelwidth\leftmarginiii \advance\labelwidth-\labelsep
  \topsep=0pt \partopsep=0pt \parsep=0pt \itemsep=0pt}
\def\@listiv{\leftmargin\leftmarginiv
  \labelwidth\leftmarginiv \advance\labelwidth-\labelsep
  \topsep=0pt \partopsep=0pt \parsep=0pt \itemsep=0pt}
\makeatother

\makeatletter
\newcommand{\opnorm}{\@ifstar\@opnorms\@opnorm}
\newcommand{\@opnorms}[1]{%
  \left|\mkern-1.5mu\left|\mkern-1.5mu\left|
   #1
  \right|\mkern-1.5mu\right|\mkern-1.5mu\right|
}
\newcommand{\@opnorm}[2][]{%
  \mathopen{#1|\mkern-1.5mu#1|\mkern-1.5mu#1|}
  #2
  \mathclose{#1|\mkern-1.5mu#1|\mkern-1.5mu#1|}
}
\makeatother

\makeatletter
\renewenvironment{proof}[1][\proofname]{\par
  \pushQED{\qed}%
  \normalfont \topsep6\p@\@plus6\p@\relax
  \trivlist
  \item[\hskip\labelsep
        \bfseries
    #1\@addpunct{.}]\ignorespaces
}{%
  \popQED\endtrivlist\@endpefalse
}
\makeatother

\DeclareMathOperator{\vectorize}{vec}

\DeclareMathOperator{\diag}{diag}
\DeclareMathOperator{\domain}{Dom}
\DeclareMathOperator{\symm}{Sym}
\DeclareMathOperator{\FWER}{\mathbf{FWER}}
\DeclareMathOperator{\leb}{\mathsf{Leb}}

\newcommand{\tcr}[1]{\textcolor{black}{#1}}
\newcommand{\bs}[1]{\boldsymbol{#1}}
\newcommand{\ol}[1]{\overline{#1}}
\newcommand{\ul}[1]{\underline{#1}}
\newcommand{\wh}[1]{\widehat{#1}}

\newcommand*\patchAmsMathEnvironmentForLineno[1]{%
  \expandafter\let\csname old#1\expandafter\endcsname\csname #1\endcsname
  \expandafter\let\csname oldend#1\expandafter\endcsname\csname end#1\endcsname
  \renewenvironment{#1}%
     {\linenomath\csname old#1\endcsname}%
     {\csname oldend#1\endcsname\endlinenomath}}%
\newcommand*\patchBothAmsMathEnvironmentsForLineno[1]{%
  \patchAmsMathEnvironmentForLineno{#1}%
  \patchAmsMathEnvironmentForLineno{#1*}}%
\AtBeginDocument{%
\patchBothAmsMathEnvironmentsForLineno{equation}%
\patchBothAmsMathEnvironmentsForLineno{align}%
\patchBothAmsMathEnvironmentsForLineno{flalign}%
\patchBothAmsMathEnvironmentsForLineno{alignat}%
\patchBothAmsMathEnvironmentsForLineno{gather}%
\patchBothAmsMathEnvironmentsForLineno{multline}%
}

\title{Mixed-normal limit theorems for multiple Skorohod integrals in high-dimensions, with application to realized covariance}
\author{Yuta Koike\thanks{Mathematics and Informatics Center and Graduate School of Mathematical Sciences, The University of Tokyo, 3-8-1 Komaba, Meguro-ku, Tokyo 153-8914 Japan}
\thanks{Department of Business Administration, Graduate School of Social Sciences, Tokyo Metropolitan University, Marunouchi Eiraku Bldg. 18F, 1-4-1 Marunouchi, Chiyoda-ku, Tokyo 100-0005 Japan}
\thanks{The Institute of Statistical Mathematics, 10-3 Midori-cho, Tachikawa, Tokyo 190-8562, Japan}
\thanks{CREST, Japan Science and Technology Agency}
}

\begin{document}


\maketitle

\begin{abstract}

This paper develops mixed-normal approximations for probabilities that vectors of multiple Skorohod integrals belong to random convex polytopes when the dimensions of the vectors possibly diverge to infinity. 
We apply the developed theory to establish the asymptotic mixed normality of the realized covariance matrix of a high-dimensional continuous semimartingale observed at a high-frequency, where the dimension can be much larger than the sample size. 
We also present an application of this result to testing the residual sparsity of a high-dimensional continuous-time factor model. 

\vspace{3mm}



\end{abstract}

\section{Introduction}

Covariance matrix estimation of multiple assets is one of the most active research areas in high-frequency financial econometrics. 
Recently, many authors have been attacking the high-dimensionality in covariance matrix estimation from high-frequency data. 
A pioneering work on this topic is the paper by \citet{WZ2010}, where the regularization methods (banding and thresholding) proposed in \citet{BL2008band,BL2008th} have been applied to estimating high-dimensional quadratic covariation matrices from noisy and non-synchronous high-frequency data. 
Subsequently, their approach has been enhanced by several papers such as \cite{TWZ2013,KWZ2016,KKLW2018}. 
Meanwhile, such methods require a kind of sparsity of the target quadratic covariation matrix itself, which seems unrealistic in financial data in view of the celebrated factor structure such as the Fama-French three-factor model of \cite{3factor}. To overcome this issue, \citet{FFX2016} have proposed a covariance estimation method based on a continuous-time (approximate) factor model with observable factors, which can be seen as a continuous-time counterpart of the method introduced in \citet{FLM2011}. 
The method has been further extended in various directions such as situations with unobservable factor, noisy and non-synchronous observations, heavy-tail errors and so on; see \cite{AX2017,DLX2017,KLW2017,FK2017,Pelger2019} for details. 
As an alternative approach to avoid assuming the sparsity of the target matrix itself, \citet{BNS2018} have proposed applying the graphical Lasso, which imposes the sparsity on the inverse of the target matrix rather than the target matrix itself. 
On the empirical side, high-dimensional covariance matrix estimation from high-frequency financial data is particularly interesting in portfolio allocation. 
We refer to \cite{Ubukata2010,FLY2012,LSS2016} for illustrations of relevant empirical work on this topic, in addition to the empirical results reported in the papers cited above.  

To the best of the author's knowledge, however, there is no work to establish a statistical inference theory validating simultaneous hypothesis testing and construction of uniformly valid confidence regions for high-dimensional quadratic covariation estimation from high-frequency data. 
Such a theory is important in statistical applications as illustrated by the following example:  
Let $Y=(Y_t)_{t\in[0,1]}$ be a $d$-dimensional continuous semimartingale. We denote by $Y^i$ the $i$-th component of $Y$ for every $i=1,\dots,d$. 
If one attempts to apply a regularization procedure to estimating the quadratic covariation matrix $[Y,Y]_1=([Y^i,Y^j]_1)_{1\leq i,j\leq d}$ of $Y$, it is important to understand whether the target matrix is really sparse or not, and if so, how sparse it is. This amounts to evaluating the following series of the statistical hypotheses \textit{simultaneously}:
\begin{equation}\label{test:sparse}
H_{(i,j)}: [Y^i,Y^j]_1\equiv0,\qquad i,j=1,\dots,d\text{ such that }i<j.
\end{equation}
A natural way to construct test statistics for this problem is to estimate $[Y,Y]_1$ and test whether each of the entries is significantly away from 0 or not. Now suppose that $Y$ is observed at the equidistant times $t_h=h/n$, $h=0,1,\dots,n$. Then the most canonical estimator for $[Y,Y]_1$ would be the so-called \textit{realized covariance matrix}:
\begin{equation}\label{def:rc}
\widehat{[Y,Y]}^n_1:=\sum_{h=1}^n(Y_{t_h}-Y_{t_{h-1}})(Y_{t_h}-Y_{t_{h-1}})^\top.
\end{equation}
If one wants to test the null hypothesis such that all the hypotheses in \eqref{test:sparse} is true, it is natural to consider the maximum type statistic
\[
\max_{(i,j)\in\Lambda}\left|\widehat{[Y,Y]}^n_1\right|,
\]
where $\Lambda:=\{(i,j)\in\{1,\dots,d\}^2:i<j\}$. More generally, if one wants to control the family-wise error rate in multiple testing for the hypotheses \eqref{test:sparse}, it is enough to approximate the distribution of $\max_{(i,j)\in\mathcal{L}}|\widehat{[Y,Y]}^n_1|$ for any $\mathcal{L}\subset\Lambda$, with the help of the stepdown procedure illustrated in \citet{RW2005}. Hence the problem amounts to approximating the distributions of such maximum type statistics in an appropriate sense. 
Using the test statistics considered in \citet{BM2016}, this type of testing problem can be extended to the sparsity test for the residual processes of a continuous-time factor model with an observable factor \tcr{and} thus promising in applications to 
high-frequency financial data. 
In addition, such a problem will also be useful for covariance matrix modeling in a low-frequency setting because it often suffers from the curse of dimensionality due to the increase of the number of unknown parameters to be estimated, and thus it is a common practice to impose a certain structure on covariance matrices for reducing the number of unknown parameters in models. For example, \citet{TWYZ2011} have proposed fitting a matrix factor model to daily covariance matrices which are estimated from high-frequency data using the methodology of \cite{WZ2010}, while \citet{KO2016,KO2018} have introduced a dynamic (multiple-block) equicorrelation structure to multivariate stochastic volatility models. 
The afore-mentioned testing will be useful for examining the validity of such specification. 
If the dimension $d$ is fixed, the desired approximation can be obtained as a simple consequence of a multivariate mixed-normal limit theorem for $\sqrt{n}(\widehat{[Y,Y]}^n_1-[Y,Y]_1)$, which is well-studied in the literature and holds true under quite mild assumptions; see e.g.~Theorem 5.4.2 of \cite{JP2012}. 
The problem here is how to establish an analogous result when the dimension $d$ possibly diverges as $n$ tends to infinity. 

Indeed, even for the sum of independent random vectors, it is far from trivial to establish such a result in a situation where the dimension is possibly (much) larger than the sample size. 
This is not surprising because objective random vectors are typically not tight in the usual sense in such a high-dimensional setting, so any standard method to establish central limit theorems no longer works. 
A significant breakthrough in this subject was achieved by the seminal work of \citet*{CCK2013}, where a Gaussian approximation of the maxima of the sum of independent random vectors in terms of the Kolmogorov distance has been established under quite mild assumptions which allow the dimension is (possibly exponentially) larger than the sample size. With the help of the Gaussian comparison theorem by \citet{CCK2015}, it enables us to construct feasible statistical inference procedures based on the maximum type statistics. Their theory, which we call the \textit{Chernozhukov-Chetverikov-Kato theory}, or the \textit{CCK theory} for short, has been developed in the subsequent work by \citet{CCK2014,CCK2016} and \citet{CCK2017}: the first two papers have developed Gaussian approximation of the suprema of empirical processes, while the latter has extended the results of \cite{CCK2013} to a central limit theorem for hyperrectangles, or sparsely convex sets in more general. 
Extension of the CCK theory to statistics other than the sum of independent random vectors has also been studied in many articles: Weakening the independence assumption has been studied in e.g.~\cite{ZW2015,ZC2017,CCK2014c,CQYZ2018}; \citet{Chen2017} and \citet{CK2017,CK2017rand} have developed theories for $U$-statistics. 
Moreover, some authors have applied the CCK theory to statistical problems regarding high-frequency data; see \citet{KK2017} and \citet{Koike2017stein}.   
Nevertheless, none of the above studies is applicable to our problem due to its non-ergodic nature.  That is, the asymptotic covariance matrix is random and depends on the $\sigma$-filed of the original probability space, \tcr{so} the asymptotic distribution is essentially non-Gaussian.   


Meanwhile, inspection of the proofs of the CCK theory reveals that most the parts do not rely on \textit{any} structure of the underlying statistics. To be precise, let $S_n$ be the random vector corresponding to the objective statistic and suppose that we aim at approximating the distribution of $S_n$ by its Gaussian analog $S_n^\dagger$ which has the same mean and covariance matrix as those of $S_n$. In the proofs of the CCK theory, the fact that $S_n$ is the sum of independent random vectors is crucial only to obtain a good quantitative estimate for the quantities $|E[f(S_n)]-E[f(S_n^\dagger)]|$ for sufficiently smooth functions $f$. In the original CCK theory \cite{CCK2013,CCK2017}, such an estimate has been established by the so-called \textit{Stein's method}, especially \textit{Slepian's interpolation} (also known as the \textit{smart path method}) and \textit{Stein's leave-one-out method}. 
Although their approach is not directly applicable to our problem, it suggests that we might alternatively use \textit{Malliavin's integration by parts formula} because it can be viewed as an infinite-dimensional version of Stein's identity (cf.~\citet{SY2004}). In fact, the recent active research in probabilistic literature shows a beautiful harmony between Malliavin calculus and Stein's method, which is nowadays called the \textit{Malliavin-Stein method}; we refer to the monograph \cite{NP2012} for an introduction of this subject. Indeed, this idea has already been applied in \cite{Koike2017stein} to a situation where $S_n$ is a vector of smooth Wiener functionals (especially multiple Wiener-It\^o integrals) and $S_n^\dagger$ is Gaussian, which has produced several impressive results. 
Our plan here is to apply this idea to a situation where $S_n$ is a vector of multiple Skorohod integrals and $S_n^\dagger$ is conditionally Gaussian. 
In this regard, a relevant result has been given in Theorem 5.1 of \citet{NNP2016}. However, this result is not directly applicable to the current situation because it assumes that the components of $S_n^\dagger$ are conditionally independent, which is less interesting to statistical applications (and especially not the case in the problem illustrated above). To remove such a restriction from the result of \cite{NNP2016}, we employ the novel interpolation method introduced in \citet{NY2017}, instead of Slepian's interpolation used in \cite{NNP2016} and the original CCK theory. 

Another problem in the present context is validation of standardizing statistics by random variables. In a low-dimensional setting, this is typically achieved by proving the so-called \textit{stable convergence in law} (see e.g.~\cite{PV2010} for details). However, in a high-dimensional setting, the meaning of stable convergence is unclear and its na\"ive extension is not useful because of the lack of the continuous mapping theorem and the delta method (see Section \ref{sec:main} for a relevant discussion). So we also aim at developing a formulation appropriate to validating such an operation.  

The remainder of the paper is organized as follows. 
Section \ref{sec:preliminaries} is devoted to some preliminaries on notation and concepts used in the paper. 
Section \ref{sec:main} presents the main results obtained in this paper. 
In Section \ref{sec:rc} we apply the developed theory to establish the asymptotic mixed normality of realized covariance matrices in a high-dimensional setting and illustrate its application to testing the residual sparsity of a continuous-time factor model. 
Section \ref{sec:simulation} provides a small simulation study as well as an empirical illustration using real data. 
All the proofs are collected in the Appendix.  

\section{Preliminaries}\label{sec:preliminaries}

In this section we present some notation and concepts used throughout the paper. 

\subsection{Basic notation}\label{sec:basic}

We begin by introducing some basic notation which \tcr{is} more or less common in the literature. 
For a vector $x\in\mathbb{R}^d$, we write the $i$-th component of $x$ as $x^i$ for $i=1,\dots,d$. 
Also, we set $\min x:=\min_{1\leq i\leq d}x^i$. 
For two vectors $x,y\in\mathbb{R}^d$, the statement $x\leq y$ means $x^i\leq y^i$ for all $i=1,\dots,d$. 
For a vector $x\in\mathbb{R}^d$ and a scalar $a\in\mathbb{R}$, we set
\[
x\pm a:=(x^1\pm a,\dots,x^d\pm a)^\top.
\]
Here, $\top$ stands for the transpose of a matrix. 

For a matrix $A$, we write its $(i,j)$-th entry as $A^{ij}$. Also, $A^{i\cdot}$ and $A^{\cdot j}$ denote the $i$-th row vector and the $j$-th column vector, respectively. Here, we regard both the vectors $A^{i\cdot}$ and $A^{\cdot j}$ as column vectors. 
If $A$ is an $m\times d$ matrix, we denote by $\opnorm{A}_\infty$ the $\ell_\infty$-operator norm of $A$:
\[
\opnorm{A}_\infty=\max_{1\leq i\leq m}\sum_{j=1}^d|A^{ij}|.
\]
If $B$ is another $m\times d$ matrix, we denote by $A\cdot B$ the Frobenius inner product of $A$ and $B$. That is,
\[
A\cdot B:=\sum_{i=1}^m\sum_{j=1}^dA^{ij}B^{ij}.
\]
For a $d\times d$ matrix $A$, we denote by $\diag(A)$ the $d$-dimensional vector consisting of the diagonal entries of $A$, i.e.~$\diag(A)=(A^{11},\dots,A^{dd})^\top$.

For a random variable $\xi$ and a number $p>0$, we write $\|\xi\|_p=(E[|\xi|^p])^{1/p}$. 
We also use the notation $\|\xi\|_\infty$ to denote the essential supremum of $\xi$. 
We will denote by $L^{\infty-}$ the space of all random variables $\xi$ such that $\|\xi\|_p<\infty$ for every $p\in[1,\infty)$. 
The notation $\to^p$ stands for convergence in probability. 



If $V$ is a real Hilbert space, we denote by $\langle\cdot,\cdot\rangle_V$ and $\|\cdot\|_V$ the inner product and norm of $V$, respectively. 
Also, we denote by $L^p(\Omega;V)$ the set of all $V$-valued random variables $\xi$ such that $E[\|\xi\|_V^2]<\infty$. 

Given real Hilbert spaces $V_1,\dots,V_k$, we write their Hilbert space tensor product as $V_1\otimes\cdots\otimes V_k$. 
For a real Hilbert space $V$, we write the $k$th tensor power of $V$ as $V^{\otimes k}$, i.e.
\[
V^{\otimes k}:=\underbrace{V\otimes\cdots\otimes V}_{k}.
\]
\tcr{
Note that the Hilbert space tensor product is uniquely determined up to isomorphism, and we often select a convenient realization case by case. 
For example, we identify the tensor product $V\otimes\mathbb{R}^d$ with the Hilbert space $V^d$ equipped with the inner product $\langle(f_1,\dots,f_d),(g_1,\dots,g_d)\rangle_{V^d}=\sum_{i=1}^d\langle f_i,g_i\rangle_V$ for $f_1,\dots,f_d,g_1,\dots,g_d\in V$. 
This is possible because the latter is the Hilbert space tensor product of $V$ and $\mathbb{R}^d$ in the sense of Definition E.8 in \cite{Janson1997}. Namely, there is a bilinear map $T:V\times\mathbb{R}^d\to V^d$ such that the range of $T$ is total in $V^d$ and
\[
\langle T(f,a),T(g,b)\rangle_{V^d}=\langle f,g\rangle_{V}\langle a,b\rangle_{\mathbb{R}^d}
\]
for all $f,g\in V$ and $a,b\in\mathbb{R}^d$. In fact, we may define $T$ by 
\[
T(f,a)=(a_1f,\dots,a_df)\qquad(f\in V,~a=(a_1,\dots,a_d)^\top\in\mathbb{R}^d).
\]
Evidently, $T$ is bilinear and its range is total in $(H^{\otimes k})^d$. Moreover, for any $f,g\in V$ and $a=(a_1,\dots,a_d)^\top\in\mathbb{R}^d$, $b=(b_1,\dots,b_d)^\top\in\mathbb{R}^d$,
\begin{align*}
\langle T(f,a),T(g,b)\rangle_{V^d}
=\sum_{i=1}^d\langle a_if,b_ig\rangle_{V}
=\sum_{i=1}^da_ib_i\langle f,g\rangle_{V}
=\langle f,g\rangle_{V}\langle a,b\rangle_{\mathbb{R}^d}.
\end{align*}
}%
For an element $f\in V^{\otimes k}$, we write the (canonical) symmetrization of $f$ as $\symm(f)$. Namely, the map $V^{\otimes k}\ni f\mapsto \symm(f)\in V^{\otimes k}$ is characterized as the unique continuous linear operator on $V^{\otimes k}$ such that
\[
\symm(f_1\otimes\cdots\otimes f_k)=\frac{1}{k!}\sum_{\tau\in\mathcal{S}_k}f_{\tau(1)}\otimes\cdots\otimes f_{\tau(k)}
\]
for all $f_1,\dots,f_k\in V$, where $\mathcal{S}_k$ denotes the set of all permutations of $\{1,\dots,k\}$, i.e.~the symmetric group of degree $k$. 
An element $f\in V^{\otimes k}$ is said to be \textit{symmetric} if $\symm(f)=f$. 
We refer to Appendix E of \cite{Janson1997} for details on Hilbert space tensor products.  

\subsection{Multi-way arrays}\label{sec:array}

In this subsection we introduce some notation related to multi-way arrays (or tensors) which are necessary to state our main results. 

Given a positive integer $N$, we set $[N]:=\{1,\dots,N\}$ for short. 
We denote by $\mathbb{K}$ the real field $\mathbb{R}$ or the complex field $\mathbb{C}$ and consider a vector space $V$ over $\mathbb{K}$. 
Given $q$ positive integers $N_1,\dots,N_q$, we denote by $V^{N_1\times\cdots\times N_q}$ the set of all $V$-valued $N_1\times\cdots\times N_q$ arrays, i.e.~$V$-valued functions on $[N_1]\times\cdots\times[N_q]$. 
Note that $V^{N_1\times N_2}$ corresponds to the set of all $V$-valued $N_1\times N_2$ matrices. 
When $N_1=\cdots=N_q=N$, we call an element of $V^{N_1\times\cdots\times N_q}$ a $V$-valued $N$-dimensional $q$-way array. 
For an array $T\in V^{N_1\times\cdots\times N_q}$ and indices $i_k\in [N_k]$ ($k=1,\dots,q$), we write $T(i_1,\dots,i_q)$ as $T^{i_1,\dots,i_q}$ and $T$ itself as $T=(T^{i_1,\dots,i_q})_{(i_1,\dots,i_q)\in\prod_{k=1}^q[N_k]}$. 
When $V=\mathbb{K}$, $V^{N_1\times\cdots\times N_q}$ is naturally identified with the Hilbert space tensor product $\mathbb{K}^{N_1}\otimes\cdots\otimes\mathbb{K}^{N_q}$ by the unique linear isomorphism $\iota:\mathbb{K}^{N_1}\otimes\cdots\otimes\mathbb{K}^{N_q}\to\mathbb{K}^{N_1\times\cdots\times N_q}$ such that $\iota(x_1\otimes\cdots\otimes x_q)=(x_1^{i_1}\cdots x_q^{i_q})_{(i_1,\dots,i_q)\in\prod_{k=1}^q[N_k]}$ for $x_k\in\mathbb{K}^{N_k}$, $k=1,\dots,q$ (cf.~Example E.10 of \cite{Janson1997}). 

For two $\mathbb{K}$-valued arrays $S,T\in\mathbb{K}^{N_1\times\cdots\times N_q}$, we define their Hadamard-type product (i.e.~entry-wise product) by
\[
S\circ T:=(S^{i_1,\dots,i_q}T^{i_1,\dots,i_q})_{(i_1,\dots,i_q)\in\prod_{k=1}^q[N_k]}\in\mathbb{K}^{N_1\times\cdots\times N_q}.
\]
Also, we set
\[
\|T\|_{\ell_p}:=
\left\{\begin{array}{ll}
\displaystyle\left\{\sum_{(i_1,\dots,i_q)\in\prod_{k=1}^q[N_k]}|T^{i_1,\dots,i_q}|^p\right\}^{1/p} & \text{if }p\in(0,\infty),\\
\displaystyle\max_{(i_1,\dots,i_q)\in\prod_{k=1}^q[N_k]}|T^{i_1,\dots,i_q}| & \text{if }p=\infty.
\end{array}\right.
\]

Now suppose that $V$ is a real Hilbert space. 
For $T\in V^{N_1\times\cdots\times N_q}$ and $x\in V$, we define
\tcr{
\begin{equation}\label{array:inner-prod}
\langle T,x\rangle_V:=(\langle T^{i_1,\dots,i_q},x\rangle_V)_{(i_1,\dots,i_q)\in\prod_{k=1}^q[N_k]}\in\mathbb{R}^{N_1\times\cdots\times N_q}.
\end{equation}
}%
Let $m$ be a positive integer. For each $j=1,\dots,m$, let $V_j$ be a real Hilbert space, $p_j\in\mathbb{N}$, $N_1^{(j)},\dots,N_{p_j}^{(j)}\in\mathbb{N}$ and $T_j\in V_j^{N^{(j)}_1\times\cdots\times N^{(j)}_{p_j}}$. Then we define
\tcr{
\begin{align}
T_1\otimes\cdots\otimes T_m
&:=(T_1^{i_1,\dots,i_{p_1}}\otimes\cdots\otimes T^{i_{p_1+\cdots+p_{m-1}+1},\dots,i_{p_1+\cdots+p_m}})_{(i_1,\dots,i_{p_1+\cdots+p_m})\in\prod_{k=1}^{p_1+\cdots+p_m}[N_k]}\nonumber\\
&\in(V_1\otimes\cdots\otimes V_m)^{N_1^{(1)}\times\cdots\times N_{p_1}^{(1)}\times\cdots\times N_1^{(m)}\times\cdots\times N_{p_m}^{(m)}}.\label{array:tensor-prod}
\end{align}
}%
In particular, we write
\[
T^{\otimes m}:=\underbrace{T\otimes\cdots\otimes T}_{m}.
\]

\subsection{Malliavin calculus}\label{sec:malliavin}

This subsection introduces some notation and concepts from Malliavin calculus used throughout the paper. We refer to \citet{Nualart2006}, Chapter 2 of \citet{NP2012} and Chapter 15 of \citet{Janson1997} for further details on this subject. 

Given a probability space $(\Omega,\mathcal{F},P)$, let $\mathbb{W}=(\mathbb{W}(h))_{h\in H}$ be an isonormal Gaussian process over a real separable Hilbert space $H$. 

Let $V$ be another real separable Hilbert space. 
For any real number $p\geq1$ and any integer $k\geq1$, $\mathbb{D}_{k,p}(V)$ denotes the stochastic Sobolev space of $V$-valued random variables which are $k$ times differentiable in the Malliavin sense and the derivatives up to order $k$ have finite moments of order $p$. 
If $F\in\mathbb{D}_{k,p}(V)$, we denote by $D^kF$ the $k$th Malliavin derivative of $F$, which is a random variable taking its values in the space $L^p(\Omega;H^{\otimes k}\otimes V)$. We write $DF$ instead of $D^1F$ for short. 
We set $\mathbb{D}_{k,\infty}(V)=\bigcap_{p=1}^\infty\mathbb{D}_{k,p}(V)$. 
If $V=\mathbb{R}$, we simply write $\mathbb{D}_{k,p}(V)$ as $\mathbb{D}_{k,p}$. 

For a $d$-dimensional random vector $F\in\mathbb{D}_{k,p}(\mathbb{R}^d)$, we identify the $k$th Malliavin derivative $D^kF$ of $F$ as the $(H^{\otimes k})^d$-valued random variable $(D^kF^1,\dots,D^kF^d)$ \tcr{by identifying $H^{\otimes k}\otimes\mathbb{R}^d$ with $(H^{\otimes k})^d$ as in Section \ref{sec:basic}}. Similarly, for a $d\times d'$ matrix valued random variable $F\in\mathbb{D}_{k,p}(\mathbb{R}^{d\times d'})$, we identify $D^kF$ as the $(H^{\otimes k})^{d\times d'}$-valued random variable $(D^kF^{ij})_{(i,j)\in[d]\times[d']}$.


For a positive integer $q$, we denote by $\delta^q$ the $q$-th multiple Skorohod integral, which is the adjoint operator of the densely defined operator $L^2(\Omega)\supset\mathbb{D}_{q,2}\ni F\mapsto D^qF\in L^2(\Omega;H^{\otimes q})$. That is, the domain $\domain(\delta^q)$ of $\delta^q$ is defined as the set of all $H^{\otimes q}$-valued random variables $u$ such that there is a constant $C>0$ satisfying $|E[\langle u,D^qF\rangle_{H^{\otimes q}}]|\leq C\|F\|_2$ for all $F\in\mathbb{D}_{q,2}$, and the following duality formula holds for any $u\in\domain(\delta^q)$ and $F\in\mathbb{D}_{q,2}$:
\[
E[F\delta^q(u)]=E[\langle u,D^qF\rangle_{H^{\otimes q}}].
\]

\subsection{Multi-indices}

This subsection collects some notation related to multi-indices. 

Let $q$ be a positive integer. We denote by $\mathbb{Z}_+$ the set of all non-negative integers. 
We define
\[
\mathcal{A}(q):=\{\alpha\in\mathbb{Z}_+^q:\alpha_1+2\alpha_2+\cdots+q\alpha_q=q\}.
\]
For a multi-index $\alpha=(\alpha_1,\dots,\alpha_q)\in\mathbb{Z}_+^q$, we set $|\alpha|=\alpha_1+\cdots+\alpha_q$ as usual. 
Given another positive integer $r$, we define
\[
\mathcal{N}_r(\alpha):=\left\{\nu=(\nu_{ij})_{(i,j)\in[q]\times[r]}:\nu_{ij}\in\mathbb{Z},\sum_{j=1}^r\nu_{ij}=\alpha_i \right\}
\]
and
\[
\mathcal{N}^*_r(\alpha)=\{\nu=(\nu_{ij})\in\mathcal{N}_r(\alpha):\nu_{q1}=0\}.
\]
Moreover, we define
\[
\tcr{\ol{\mathcal{A}}(q):=\bigcup_{p=1}^q\mathcal{A}(p)\quad\text{and}\quad}
\ol{\mathcal{N}}^*_r(q):=\bigcup_{\alpha\in\tcr{\ol{\mathcal{A}}}(q)}\mathcal{N}^*_r(\alpha).
\]
Finally, for an element $\nu=(\nu_{ij})\in\mathcal{N}_4(\alpha)$, we set $|\nu|_*:=|\nu_{\cdot 1}|+2|\nu_{\cdot 2}|+|\nu_{\cdot 3}|$ and $|\nu|_{**}:=|\nu|_*+|\nu_{\cdot4}|$. 

\section{Main results}\label{sec:main}

Throughout the paper, we consider an asymptotic theory such that the parameter $n\in\mathbb{N}$ tends to infinity. For each $n\in\mathbb{N}$, we consider a probability space $(\Omega^{n},\mathcal{F}^{n},P^{n})$, and we suppose that all the random variables at stage $n$ are defined on $(\Omega^{n},\mathcal{F}^{n},P^{n})$. We also suppose that an isonormal Gaussian process $\mathbb{W}_n=(\mathbb{W}_n(h))_{h\in H_n}$ over a real separable Hilbert space $H_n$ is defined on $(\Omega^{n},\mathcal{F}^{n},P^{n})$. 
To keep the notation simple, we subtract the indices $n$ from $(\Omega^{n},\mathcal{F}^{n},P^{n})$, $\mathbb{W}_n$ and $H_n$, respectively. So we will write them simply as $(\Omega,\mathcal{F},P)$,  $\mathbb{W}$ and $H$, respectively. 
In particular, note that the spaces and the operators associated with $\mathbb{W}$ (which are introduced in Section \ref{sec:malliavin}) implicitly depend on $n$, although we do not attach the index $n$ to them. 

For each $n\in\mathbb{N}$, let $M_n$ be a $d$-dimensional random vector consisting of multiple Skorohod integrals:
\[
M_n^j=\delta^{q_j}(u_n^j),\qquad j=1,\dots,d,
\]
where $q_j$ is a positive integer and $u_n^j\in\domain(\delta^{q_j})$ for every $j$. 
Here, we assume that the dimension $d$ possibly depends on $n$ as $d=d_n$, while $q_j$'s do not depend on $n$. We also assume $d_n\geq\tcr{3}$ for every $n$ and $\ol{q}:=\sup_{j}q_j<\infty$. 
Our aim is to study mixed-normal limit theorems for the following functionals:
\[
Z_n=M_n+W_n,\qquad n=1,2,\dots,
\]
where $W_n$'s are $d$-dimensional random vectors which represent the uncentered part of the functionals. 
%

Let us introduce mixed-normal random vectors approximating the functionals $Z_n$ in law as follows:
\[
\mathfrak{Z}_n=\mathfrak{C}_n^{1/2}\zeta_n+W_n,\qquad n=1,2,\dots.
\]
Here, $\mathfrak{C}_n$ is a $d\times d$ symmetric positive semidefinite random matrix and $\zeta_n$ is a $d$-dimensional standard Gaussian vector independent of $\mathcal{F}$, which is defined on an extension of the probability space $(\Omega,\mathcal{F},P)$ if necessary. 
%

The main aim of this paper is to investigate reasonable regularity conditions under which the distribution of $Z_n$ is well-approximated by that of $\mathfrak{Z}_n$. 
To be precise, we are interested in the following type of result:
\[
\sup_{z\in\mathbb{R}^d}\left|P(Z_n\leq z)-P(\mathfrak{Z}_n\leq z)\right|\to0\qquad\text{as }n\to\infty.
\]
It is well-recognized in statistic literature, however, that this type of result is usually insufficient for statistical applications because it does not ensure standardization by a \textit{random} vector which is still random in the limit; such an operation is crucial for Studentization in the present context. In a low-dimensional setting, this issue is usually resolved by proving the stability of the convergence so that
\[
(Z_n,X)\to^{\mathcal{L}}(\mathfrak{Z}_n,X)\qquad\text{as }n\to\infty
\]
for any $m$-dimensional ($\mathcal{F}$-measurable) random variable $X$, where $\to^{\mathcal{L}}$ denotes the convergence in law. 
This statement is no longer meaningful in a high-dimensional setting such that $d\to\infty$ as $n\to\infty$, so we need to reformulate it appropriately. A na\"ive idea is to consider the following statement:
\begin{equation}\label{stable-kol}
\sup_{z\in\mathbb{R}^d,x\in\mathbb{R}^m}\left|P(Z_n\leq z,X\leq x)-P(\mathfrak{Z}_n\leq z,X\leq x)\right|\to0\qquad\text{as }n\to\infty.
\end{equation}
However, if $m$ depends also on $n$, this type of statement is not attractive neither theoretical nor practical points of view due to the following reasons: 
From a theoretical point of view, we need to assume a so-called \textit{anti-concentration inequality} for $X$ to prove this type of result by the CCK approach, but it is usually hard to check such an inequality for general random variables, especially when $m\to\infty$ as $n\to\infty$. 
Besides, from a practical point of view, it is still unclear whether the convergence \eqref{stable-kol} ensures the validity of standardization of $Z_n$ because no analog of the continuous mapping theorem has been established yet for high-dimensional central limit theorems of the form \eqref{stable-kol}. 
For these reasons we choose the way to directly prove convergence results for normalized statistics of $Z_n$. More formally, let $\Xi_n$ be an $m\times d$ random matrix, where $m=m_n\geq\tcr{3}$ possibly depends on $n$. Our aim is to establish
\begin{equation}\label{aim}
\sup_{y\in\mathbb{R}^m}|P(\Xi_nZ_n\leq y)-P(\Xi_n\mathfrak{Z}_n\leq y)|\to0
\end{equation}
as $n\to\infty$ under reasonable regularity conditions on $Z_n$ and $\Xi_n$. 
Mathematically speaking, given a vector $y\in\mathbb{R}^m$, the set $\{z\in\mathbb{R}^d:\Xi_nz\leq y\}$ is a finite intersection of hyperplanes in $\mathbb{R}^d$, i.e.~convex polytopes in $\mathbb{R}^d$, \tcr{so} the convergence \eqref{aim} can be considered as a high-dimensional central limit theorem for random convex polytopes. 
If we take $\Xi_n$ as the $d\times d$ diagonal matrix whose diagonals are the inverses of the ``standard errors'' of $Z_n$, the convergence \eqref{aim} does ensures the validity of (marginal) standardization of $Z_n$. 

Now, our main theorem is stated as follows: 
\begin{theorem}\label{thm:main}
Suppose that $M_n,W_n\in\mathbb{D}_{\overline{q},\infty}(\mathbb{R}^d)$ and $\mathfrak{C}_n\in\mathbb{D}_{\overline{q},\infty}(\mathbb{R}^{d\times d})$ and that $u_n^j$ is symmetric for all $n$ and $j$.  
Suppose also that $\Xi_n$ can be written as $\Xi_n=\Upsilon_n\circ\bs{X}_n$ with $\Upsilon_n$ being an $m\times d$ (deterministic) matrix such that $\opnorm{\Upsilon_n}_\infty\geq1$ and $\bs{X}_{n}\in\mathbb{D}_{\overline{q},\infty}(\mathbb{R}^{m\times d})$.  
Assume that the following convergences hold true:
\begin{equation}\label{qtan-conv}
\opnorm{\Upsilon_n}_\infty^2E\left[\|\bs{X}_n\|_{\ell_\infty}^2\|\Delta_n\|_{\ell_\infty}\right](\log m)^2\to0
\end{equation}
and
\begin{equation}\label{delta-conv}
\opnorm{\Upsilon_n}_\infty^{|\nu|_{**}+1}E\left[\left(1+\|\bs{X}_n\|_{\ell_\infty}^{|\nu|_*+1}\right)\left(1+\|Z_n\|_{\ell_\infty}^{|\nu_{\cdot 4}|}+\|\mathfrak{Z}_n\|_{\ell_\infty}^{|\nu_{\cdot 4}|}\right)\max_{1\leq j\leq d}\|\Delta_{n,j}(\nu)\|_{\ell_\infty}\right](\log m)^{\frac{3}{2}|\nu|_{**}+\frac{1}{2}}\to0
\end{equation}
as $n\to\infty$ for every $\nu\in\ol{\mathcal{N}}_4^*(\ol{q})$, where 
\begin{equation}\label{qtan}
\Delta_n=\left(\langle D^{q_j}M^i_n,u_n^j\rangle_{H^{\otimes q_j}}-\mathfrak{C}_n^{ij}\right)_{1\leq i,j\leq d}
\end{equation}
and
\begin{equation}\label{eq:delta}
\Delta_{n,j}(\nu):=\left\langle \bigotimes_{k=1}^{q_j}(D^kM_n)^{\otimes \nu_{k1}}\otimes(D^k\mathfrak{C}_n)^{\otimes \nu_{k2}}\otimes(D^kW_n)^{\otimes \nu_{k3}}\otimes(D^k\bs{X}_{n})^{\otimes \nu_{k4}},u_n^j\right\rangle_{H^{\otimes q_j}}
\end{equation}
\tcr{if $\nu\in\bigcup_{\alpha\in\mathcal{A}(q_j)}\mathcal{N}_4^*(\alpha)$ and $\Delta_{n,j}(\nu)=0$ otherwise.} 
Assume also that the following condition is satisfied:
\begin{equation}\label{diag-tight}
\lim_{b\downarrow0}\limsup_{n\to\infty}P(\min\diag(\Xi_n\mathfrak{C}_n\Xi_n^\top)<b)=0.
\end{equation}
Then we have \eqref{aim} as $n\to\infty$. 
\end{theorem}

\begin{rmk}
The variable $\Delta_n$ defined in \eqref{qtan} is called the \textit{quasi-tangent} in \cite{NY2017}. 
\end{rmk}
\begin{rmk}
The variable $\Delta_{n,j}(\nu)$ defined in \eqref{eq:delta} takes values in 
\[
\mathbb{R}^{\underbrace{\scriptstyle{d\times\cdots\times d}}_{|\nu|_{*}}\times\underbrace{\scriptstyle{m\times\cdots\times m}}_{|\nu_{\cdot4}|}
\times\underbrace{\scriptstyle{d\times\cdots\times d}}_{|\nu_{\cdot4}|}}
\]
when $\nu=(\nu_{kl})\in\bigcup_{\alpha\in\mathcal{A}(q_j)}\mathcal{N}_4^*(\alpha)$. 
To see this, let us recall that $D^kM_n$, $D^k\mathfrak{C}_n$, $D^kW_n$ and $D^k\bs{X}_n$ take values in $(H^{\otimes k})^d$, $(H^{\otimes k})^{d\times d}$, $(H^{\otimes k})^d$ and $(H^{\otimes k})^{m\times d}$, respectively (cf.~Section \ref{sec:malliavin}). 
Therefore, according to the notation defined by \eqref{array:tensor-prod}, the variable
\[
\bigotimes_{k=1}^{q_j}(D^kM_n)^{\otimes \nu_{k1}}\otimes(D^k\mathfrak{C}_n)^{\otimes \nu_{k2}}\otimes(D^kW_n)^{\otimes \nu_{k3}}\otimes(D^k\bs{X}_{n})^{\otimes \nu_{k4}}
\]
takes values in
\[
\left(H^{\otimes \sum_{k=1}^{q_j}k(\nu_{k1}+\nu_{k2}+\nu_{k3}+\nu_{k4})}\right)^{\underbrace{\scriptstyle{d\times\cdots\times d}}_{\sum_{k=1}^{q_j}(\nu_{k1}+2\nu_{k2}+\nu_{k3})}
\times\underbrace{\scriptstyle{m\times\cdots\times m}}_{\sum_{k=1}^{q_j}\nu_{k4}}
\times\underbrace{\scriptstyle{d\times\cdots\times d}}_{\sum_{k=1}^{q_j}\nu_{k4}}}
=\left(H^{\otimes q_j}\right)^{\underbrace{\scriptstyle{d\times\cdots\times d}}_{|\nu|_{*}}\times\underbrace{\scriptstyle{m\times\cdots\times m}}_{|\nu_{\cdot4}|}
\times\underbrace{\scriptstyle{d\times\cdots\times d}}_{|\nu_{\cdot4}|}},
\]
where the last identity follows from the relation $\sum_{k=1}^{q_j}k(\nu_{k1}+\nu_{k2}+\nu_{k3}+\nu_{k4})=q_j$. 
Hence, according to the notation defined by \eqref{array:inner-prod}, we obtain
\begin{align*}
\Delta_{n,j}(\nu)&=\left\langle \bigotimes_{k=1}^{q_j}(D^kM_n)^{\otimes \nu_{k1}}\otimes(D^k\mathfrak{C}_n)^{\otimes \nu_{k2}}\otimes(D^kW_n)^{\otimes \nu_{k3}}\otimes(D^k\bs{X}_{n})^{\otimes \nu_{k4}},u_n^j\right\rangle_{H^{\otimes q_j}}\\
&\in\mathbb{R}^{\underbrace{\scriptstyle{d\times\cdots\times d}}_{|\nu|_{*}}\times\underbrace{\scriptstyle{m\times\cdots\times m}}_{|\nu_{\cdot4}|}
\times\underbrace{\scriptstyle{d\times\cdots\times d}}_{|\nu_{\cdot4}|}}.
\end{align*}
\end{rmk}

\begin{rmk}
In Theorem \ref{thm:main}, we require all the variables appearing there to have finite moments of all orders just for simplicity. It would be enough for them to have finite moments up to order $p$ only, where $p$ would be a function of $\ol{q}$.  
\end{rmk}

\begin{rmk}[Quantitative bound]
As in the original CCK theory, it is possible to give a quantitative version of the convergence \eqref{aim}, but we do not implement it here to make the statement of the theorem simpler. 
\end{rmk}

\tcr{
Let us write down conditions \eqref{qtan-conv}--\eqref{delta-conv} in the special case that $q_j\in\{1,2\}$ for all $j$. In this case, setting $\mathcal{J}_q=\{j\in\{1,\dots,d\}:q_j=q\}$ for $q=1,2$, we can rewrite these conditions as follows: 
\begin{align*}
&\opnorm{\Upsilon_n}_\infty^2\max_{q=1,2}E\left[\|\bs{X}_n\|_{\ell_\infty}^2\max_{1\leq i\leq d}\max_{j\in\mathcal{J}_q}\left|\langle D^qM^i_n,u_n^j\rangle_{H}-\mathfrak{C}_n^{ij}\right|\right](\log m)^2\to0,\\
&\opnorm{\Upsilon_n}_\infty^{3}\max_{q=1,2}E\left[(1+\|\bs{X}_n\|_{\ell_\infty}^{3})\max_{1\leq i,j\leq d}\max_{k\in\mathcal{J}_q}\left|\langle D^q\mathfrak{C}_n^{ij},u_n^k\rangle_{H^{\otimes q}}\right|\right](\log m)^{\frac{7}{2}}\to0,\\
&\opnorm{\Upsilon_n}_\infty^{2}\max_{q=1,2}E\left[(1+\|\bs{X}_n\|_{\ell_\infty}^{2})\max_{1\leq i\leq d}\max_{j\in\mathcal{J}_q}\left|\langle D^qW^{i}_n,u_n^j\rangle_{H^{\otimes q}}\right|\right](\log m)^{2}\to0,\\
&\opnorm{\Upsilon_n}_\infty^{2}\max_{q=1,2}E\left[\left(1+\|Z_n\|_{\ell_\infty}+\|\mathfrak{Z}_n\|_{\ell_\infty}\right)\max_{1\leq i\leq m}\max_{1\leq j\leq d}\max_{k\in\mathcal{J}_q}\left|\langle D^q\bs{X}^{ij}_n,u_n^k\rangle_{H^{\otimes q}}\right|\right](\log m)^2\to0,\\
&\opnorm{\Upsilon_n}_\infty^{3}E\left[(1+\|\bs{X}_n\|_{\ell_\infty}^{3})\max_{1\leq i,j\leq d}\max_{k\in\mathcal{J}_2}\left|\langle DF^{i}_n\otimes DG_n^j,u_n^k\rangle_{H^{\otimes 2}}\right|\right](\log m)^{\frac{7}{2}}\to0,\\
&\opnorm{\Upsilon_n}_\infty^{5}E\left[(1+\|\bs{X}_n\|_{\ell_\infty}^{5})\max_{1\leq i,j,k,l\leq d}\max_{h\in\mathcal{J}_2}\left|\langle D\mathfrak{C}_n^{ij}\otimes D\mathfrak{C}_n^{kl},u_n^h\rangle_{H^{\otimes 2}}\right|\right](\log m)^{\frac{13}{2}}\to0,\\
&\opnorm{\Upsilon_n}_\infty^{3}E\left[\left(1+\|Z_n\|_{\ell_\infty}^2+\|\mathfrak{Z}_n\|_{\ell_\infty}^2\right)\max_{1\leq i,k\leq m}\max_{1\leq j,l\leq d}\max_{h\in\mathcal{J}_2}\left|\langle D\bs{X}^{ij}_n\otimes D\bs{X}_n^{kl},u_n^h\rangle_{H^{\otimes 2}}\right|\right](\log m)^{\frac{7}{2}}\to0,\\
&\opnorm{\Upsilon_n}_\infty^{4}E\left[(1+\|\bs{X}_n\|_{\ell_\infty}^{4})\max_{1\leq i,j,k\leq d}\max_{l\in\mathcal{J}_2}\left|\langle D\mathfrak{C}_n^{ij}\otimes DF_n^{k},u_n^l\rangle_{H^{\otimes 2}}\right|\right](\log m)^{5}\to0,\\
&\opnorm{\Upsilon_n}_\infty^{3}E\left[(1+\|\bs{X}_n\|_{\ell_\infty}^{2})\left(1+\|Z_n\|_{\ell_\infty}+\|\mathfrak{Z}_n\|_{\ell_\infty}\right)\max_{1\leq j\leq m}\max_{1\leq i,k\leq d}\max_{l\in\mathcal{J}_2}\left|\langle DF^{i}_n\otimes D\bs{X}_n^{jk},u_n^l\rangle_{H^{\otimes 2}}\right|\right](\log m)^{\frac{7}{2}}\to0,\\
&\opnorm{\Upsilon_n}_\infty^{4}E\left[(1+\|\bs{X}_n\|_{\ell_\infty}^{3})\left(1+\|Z_n\|_{\ell_\infty}+\|\mathfrak{Z}_n\|_{\ell_\infty}\right)\max_{1\leq k\leq m}\max_{1\leq i,j,l\leq d}\max_{h\in\mathcal{J}_2}\left|\langle D\mathfrak{C}_n^{ij}\otimes D\bs{X}_n^{kl},u_n^h\rangle_{H^{\otimes 2}}\right|\right](\log m)^{5}\to0,
\end{align*}
where $F_n,G_n\in\{M_n,W_n\}$.
}%
%
In particular, when $q_j=1$ for all $j$, they consist of the following convergences:
\begin{align*}
&\opnorm{\Upsilon_n}_\infty^2E\left[\|\bs{X}_n\|_{\ell_\infty}^2\max_{1\leq i,j\leq d}\left|\langle DM^i_n,u_n^j\rangle_{H}-\mathfrak{C}_n^{ij}\right|\right](\log m)^2\to0,\\
&\opnorm{\Upsilon_n}_\infty^{3}E\left[(1+\|\bs{X}_n\|_{\ell_\infty}^{3})\max_{1\leq i,j,k\leq d}\left|\langle D\mathfrak{C}_n^{ij},u_n^k\rangle_{H}\right|\right](\log m)^{\frac{7}{2}}\to0,\\
&\opnorm{\Upsilon_n}_\infty^{2}E\left[(1+\|\bs{X}_n\|_{\ell_\infty}^{2})\max_{1\leq i,j\leq d}\left|\langle DW^{i}_n,u_n^j\rangle_{H}\right|\right](\log m)^{2}\to0,\\
&\opnorm{\Upsilon_n}_\infty^{2}E\left[\left(1+\|Z_n\|_{\ell_\infty}+\|\mathfrak{Z}_n\|_{\ell_\infty}\right)\max_{1\leq i\leq m}\max_{1\leq j,k\leq d}\left|\langle D\bs{X}^{ij}_n,u_n^k\rangle_{H}\right|\right](\log m)^2\to0.
\end{align*}
\tcr{W}hen $q_j=2$ for all $j$, they consist of the following convergences:
\tcr{
\begin{align}
&\opnorm{\Upsilon_n}_\infty^2E\left[\|\bs{X}_n\|_{\ell_\infty}^2\max_{1\leq i,j\leq d}\left|\langle D^2M^i_n,u_n^j\rangle_{H^{\otimes 2}}-\mathfrak{C}_n^{ij}\right|\right](\log m)^2\to0,\label{main2:eq1}\\
&\opnorm{\Upsilon_n}_\infty^{3}E\left[(1+\|\bs{X}_n\|_{\ell_\infty}^{3})\max_{1\leq i,j,k\leq d}\left|\langle D^{2}\mathfrak{C}_n^{ij},u_n^k\rangle_{H^{\otimes 2}}\right|\right](\log m)^{\frac{7}{2}}\to0,\label{main2:eq2}\\
&\opnorm{\Upsilon_n}_\infty^{2}E\left[(1+\|\bs{X}_n\|_{\ell_\infty}^{2})\max_{1\leq i,j\leq d}\left|\langle D^{2}W^{i}_n,u_n^j\rangle_{H^{\otimes 2}}\right|\right](\log m)^{2}\to0,\label{main2:eq3}\\
&\opnorm{\Upsilon_n}_\infty^{2}E\left[\left(1+\|Z_n\|_{\ell_\infty}+\|\mathfrak{Z}_n\|_{\ell_\infty}\right)\max_{1\leq i\leq m}\max_{1\leq j,k\leq d}\left|\langle D^{2}\bs{X}^{ij}_n,u_n^k\rangle_{H^{\otimes 2}}\right|\right](\log m)^2\to0,\label{main2:eq4}\\
&\opnorm{\Upsilon_n}_\infty^{3}E\left[(1+\|\bs{X}_n\|_{\ell_\infty}^{3})\max_{1\leq i,j,k\leq d}\left|\langle DF^{i}_n\otimes DG_n^j,u_n^k\rangle_{H^{\otimes 2}}\right|\right](\log m)^{\frac{7}{2}}\to0,\label{main2:eq5}\\
&\opnorm{\Upsilon_n}_\infty^{5}E\left[(1+\|\bs{X}_n\|_{\ell_\infty}^{5})\max_{1\leq i,j,k,l,h\leq d}\left|\langle D\mathfrak{C}_n^{ij}\otimes D\mathfrak{C}_n^{kl},u_n^h\rangle_{H^{\otimes 2}}\right|\right](\log m)^{\frac{13}{2}}\to0,\label{main2:eq6}\\
&\opnorm{\Upsilon_n}_\infty^{3}E\left[\left(1+\|Z_n\|_{\ell_\infty}^2+\|\mathfrak{Z}_n\|_{\ell_\infty}^2\right)\max_{1\leq i,k\leq m}\max_{1\leq j,l,h\leq d}\left|\langle D\bs{X}^{ij}_n\otimes D\bs{X}_n^{kl},u_n^h\rangle_{H^{\otimes 2}}\right|\right](\log m)^{\frac{7}{2}}\to0,\label{main2:eq7}\\
&\opnorm{\Upsilon_n}_\infty^{4}E\left[(1+\|\bs{X}_n\|_{\ell_\infty}^{4})\max_{1\leq i,j,k,l\leq d}\left|\langle D\mathfrak{C}_n^{ij}\otimes DF_n^{k},u_n^l\rangle_{H^{\otimes 2}}\right|\right](\log m)^{5}\to0,\label{main2:eq8}\\
&\opnorm{\Upsilon_n}_\infty^{3}E\left[(1+\|\bs{X}_n\|_{\ell_\infty}^{2})\left(1+\|Z_n\|_{\ell_\infty}+\|\mathfrak{Z}_n\|_{\ell_\infty}\right)\max_{1\leq j\leq m}\max_{1\leq i,k,l\leq d}\left|\langle DF^{i}_n\otimes D\bs{X}_n^{jk},u_n^l\rangle_{H^{\otimes 2}}\right|\right](\log m)^{\frac{7}{2}}\to0,\label{main2:eq9}\\
&\opnorm{\Upsilon_n}_\infty^{4}E\left[(1+\|\bs{X}_n\|_{\ell_\infty}^{3})\left(1+\|Z_n\|_{\ell_\infty}+\|\mathfrak{Z}_n\|_{\ell_\infty}\right)\max_{1\leq k\leq m}\max_{1\leq i,j,l,h\leq d}\left|\langle D\mathfrak{C}_n^{ij}\otimes D\bs{X}_n^{kl},u_n^h\rangle_{H^{\otimes 2}}\right|\right](\log m)^{5}\to0,\label{main2:eq10}
\end{align}
}%
where $F_n,G_n\in\{M_n,W_n\}$.

As a special case of Theorem \ref{thm:main}, we can deduce a high-dimensional central limit theorem for multiple Skorohod integrals in hyperrectangles as follows. Let $\mathcal{A}^\mathrm{re}(d)$ be the set of all hyperrectangles in $\mathbb{R}^d$, i.e.~$\mathcal{A}^\mathrm{re}(d)$ consists of all sets $A$ of the form
\[
A=\{z\in\mathbb{R}^d:a_j\leq z^j\leq b_j\text{ for all }j=1,\dots,d\}
\]
for some $-\infty\leq a_j\leq b_j\leq\infty$, $j=1,\dots,d$. Taking $\Xi_n$ as
\[
\Xi_n=\left(\begin{array}{c}
\mathsf{E}_d\\
-\mathsf{E}_d
\end{array}\right)
\]
in Theorem \ref{thm:main}, where $\mathsf{E}_d$ denotes the identity matrix of size $d$, we obtain the following result (note that \eqref{aim} continues to hold true while $\mathbb{R}^d$ is replaced by $(-\infty,\infty]^d$):
\begin{corollary}\label{coro:main}
Suppose that $M_n,W_n\in\mathbb{D}_{\overline{q},\infty}(\mathbb{R}^d)$ and $\mathfrak{C}_n\in\mathbb{D}_{\overline{q},\infty}(\mathbb{R}^{d\times d})$ and that $u_n^j$ is symmetric for all $n$ and $j$. 
Assume that the following convergences hold true:
\begin{equation*}
E\left[\|\bs{X}_n\|_{\ell_\infty}^2\|\Delta_n\|_{\ell_\infty}\right](\log d)^2\to0
\end{equation*}
and
\begin{equation*}
E\left[\left(1+\|\bs{X}_n\|_{\ell_\infty}^{|\nu|_*+1}\right)\left(1+\|Z_n\|_{\ell_\infty}^{|\nu_{\cdot 4}|}+\|\mathfrak{Z}_n\|_{\ell_\infty}^{|\nu_{\cdot 4}|}\right)\max_{1\leq j\leq d}\|\Delta_{n,j}(\nu)\|_{\ell_\infty}\right](\log d)^{\frac{3}{2}|\nu|_{**}+\frac{1}{2}}\to0
\end{equation*}
as $n\to\infty$ for every $\nu\in\ol{\mathcal{N}}_4^*(\ol{q})$. 
Assume also that the following condition is satisfied:
\begin{equation*}
\lim_{b\downarrow0}\limsup_{n\to\infty}P(\min\diag(\mathfrak{C}_n)<b)=0.
\end{equation*}
Then we have
\[
\sup_{A\in\mathcal{A}^\mathrm{re}(d)}\left|P(Z_n\in A)-P(\mathfrak{Z}_n\in A)\right|\to0
\]
as $n\to\infty$. 
\end{corollary}

\subsection*{Some related results for statistical applications}

In many applications, the objective variables are only approximately multiple Skorohod integrals. The following lemma is useful for such a situation. 
\begin{lemma}\label{lemma:approx}
For each $n\in\mathbb{N}$, let $Y_n,Y_n'$ be $m$-dimensional random vectors such that
\[
\sqrt{\log m}\|Y_n'-Y_n\|_{\ell_\infty}\to^p0
\]
and
\[
\sup_{y\in\mathbb{R}^m}|P(Y_n\leq y)-P(\Xi_n\mathfrak{Z}_n\leq y)|\to0
\]
as $n\to\infty$. Then we have
\[
\sup_{y\in\mathbb{R}^m}|P(Y'_n\leq y)-P(\Xi_n\mathfrak{Z}_n\leq y)|\to0
\]
as $n\to\infty$, provided that \eqref{diag-tight} holds true. 
\end{lemma}


In terms of statistical applications, the mixed-normal approximation given by Theorem \ref{thm:main} is often infeasible because the ``asymptotic'' covariance matrix $\mathfrak{C}_n$ usually contains unobservable quantities. In the following we give two auxiliary results bridging this gap. 
The first result ensures the validity of estimating the $\mathcal{F}$-conditional distribution of $\Xi_n\mathfrak{Z}_n$ while we replace $\mathfrak{C}_n,W_n$ and $\Xi_n$ by their estimators.  
\begin{proposition}\label{prop:comparison}
For each $n$, let $\wh{\mathfrak{C}}_n,\wh{W}_n$ and $\wh{\Xi}_n$ be a $d\times d$ symmetric positive semidefinite random matrix, a $d$-dimensional random vector and an $m\times d$ random matrix, respectively. Set $\widehat{\mathfrak{Z}}_n:=\widehat{\mathfrak{C}}_n^{1/2}\zeta_n+\wh{W}_n$. 
Suppose that
\begin{equation}\label{eq:consistent}
\sqrt{\log m}\|\wh{\Xi}_n\wh{W}_n-\Xi_nW_n\|_{\ell_\infty}\to^p0,\qquad
(\log m)^2\|\wh{\Xi}_n\widehat{\mathfrak{C}}_n\wh{\Xi}_n^\top-\Xi_n\mathfrak{C}_n\Xi_n^\top\|_{\ell_\infty}\to^p0
\end{equation}
as $n\to\infty$. Then we have
\[
\sup_{y\in\mathbb{R}^m}|P(\wh{\Xi}_n\widehat{\mathfrak{Z}}_n\leq y|\mathcal{F})-P(\Xi_n\mathfrak{Z}_n\leq y|\mathcal{F})|\to0
\]
as $n\to\infty$, provided that \eqref{diag-tight} holds true.
\end{proposition}
We remark that the above proposition only gives a way to estimate the $\mathcal{F}$-conditional distribution of $\Xi_n\mathfrak{Z}_n$ when we have appropriate estimators for relevant variables: It says nothing about how to estimate the \textit{unconditional} distribution of $\Xi_n\mathfrak{Z}_n$. 
Because of the non-ergodic nature of the problem, in general there seems no hope of consistently estimating the latter quantity even if we can consistently estimate unknown variables contained in $\Xi_n\mathfrak{Z}_n$. 
In a low-dimensional setting this issue is usually resolved by standardizing the objective statistic by a consistent estimator for its asymptotic covariance matrix, which is validated via the stability of convergence in law. 
In a high-dimensional setting, however, standardizing the (joint) distribution of the objective statistic is often difficult: Estimators for the conditional covariance matrix of the objective statistic are usually singular because the sample size is smaller than the dimension, and even if it is regular, computation of the inverse is typically time-consuming. 
Nevertheless, we can fortunately show that, in order to estimate quantiles of the unconditional distribution $\Xi_n\mathfrak{Z}_n$, it is sufficient to only estimate its $\mathcal{F}$-conditional distribution. 
We remark that this fact has already been known in high-frequency financial econometrics and typically been used to construct jump-related testing procedures; see \cite{JT2009,LTT2017} for example. 
Formally, we can prove the following result:
\begin{proposition}\label{prop:quantile}
For each $n\in\mathbb{N}$, let $T_n,T_n^\dagger,T_n^*$ be random variables defined on an extension of the probability space $(\Omega,\mathcal{F},P)$. Suppose that
\begin{align*}
&\sup_{x\in\mathbb{R}}\left|P(T_n\leq x)-P(T_n^\dagger\leq x)\right|\to0,
&\sup_{x\in\mathbb{R}}\left|P(T_n^*\leq x|\mathcal{F})-P(T_n^\dagger\leq x|\mathcal{F})\right|\to^p0
\end{align*}
as $n\to\infty$. 
Suppose also that there is a sequence $(E_n)$ of elements in $\mathcal{F}$ such that the $\mathcal{F}$-conditional distribution of $T_n^\dagger$ has the density on $E_n$ for every $n$ and $\lim_{n\to\infty}P(E_n)=1$. For each $n\in\mathbb{N}$, let $q_n^*$ be the $\mathcal{F}$-conditional quantile function of $T_n^*$:
\[
q_n^*(\alpha)=\inf\{x\in\mathbb{R}:P(T_n^*\leq x|\mathcal{F})\geq \alpha\},\qquad \alpha\in(0,1).
\]
Then we have
\[
P(T_n \leq q_n^*(\alpha))\to\alpha
\]
as $n\to\infty$ for all $\alpha\in(0,1)$. 
\end{proposition}

\section{Application to realized covariance}\label{sec:rc}

In this section we assume that the probability space $(\Omega,\mathcal{F},P)$ admits the structure such that $\Omega=\Omega'\times\mathbf{W}$, $\mathcal{F}=\mathcal{F}'\otimes\mathbf{B}$ and $P=P'\times\mathbf{P}$ for some probability space $(\Omega',\mathcal{F}',P')$ and the $r$-dimensional Wiener space $(\mathbf{W},\mathbf{B},\mathbf{P})$ over time interval $[0,1]$, and consider the partial Malliavin calculus with respect to the $r$-dimensional Brownian motion $B=(B_t)_{t\in[0,1]}$ defined by $B_t(\omega',w)=w(t)$ for $\omega'\in\Omega'$, $w\in\mathbf{W}$ and $t\in[0,1]$ (cf.~Section 6.1 of \cite{Yoshida1997}). 
In this setting the Hilbert space $H$ coincides with the space $L^2([0,1];\mathbb{R}^r)$. 
We here allow the dimension $r=r_n$ to possibly depend on $n\in\mathbb{N}$, so $(\Omega,\mathcal{F},P)$ and $B$ may depend on $n$, but we subtract the index $n$ from the notation. 
Let $(\mathcal{B}_t)_{t\in[0,1]}$ \tcr{denote} the filtration generated by the canonical process on $\mathbf{W}$, and define the filtration $(\mathcal{F}_t)_{t\in[0,1]}$ of $\mathcal{F}$ by $\mathcal{F}_t:=\mathcal{F}'\otimes\mathcal{B}_t$ for $t\in[0,1]$. 
On the stochastic basis $(\Omega,\mathcal{F},(\mathcal{F}_t),P)$, we consider the $d$-dimensional continuous It\^o semimartingale $Y=(Y_t)_{t\in[0,1]}$ given by the following:
\[
Y_t=Y_0+\int_0^t\mu_sds+\int_0^t\sigma_sdB_s,\qquad t\in[0,1].
\]
Here, $\mu=(\mu_s)_{s\in[0,1]}$ is a $d$-dimensional $(\mathcal{F}_t)$-progressively measurable process and $\sigma=(\sigma_s)_{s\in[0,1]}$ is an $\mathbb{R}^{d\times r}$-valued $(\mathcal{F}_t)$-progressively measurable process such that 
\[
\int_0^1\left(\|\mu_s\|_{\ell_1}+\|\sigma_s\|_{\ell_2}^2\right)ds<\infty\quad\text{a.s.}
\]
We remark that the processes $\mu$ and $\sigma$ generally depend on $n$ because $d$ and $r$ may depend on $n$. However, following the custom of high-dimensional statistics, we subtract the index $n$ from the notation as above. 

We observe the process $Y$ at the discrete time points $t_h=t_h^n=h/n$, $h=0,1,\dots,n$. In such a setting, the discretized quadratic covariation matrix
\[
\widehat{[Y,Y]}^n_1:=\sum_{h=1}^n(Y_{t_h}-Y_{t_{h-1}})(Y_{t_h}-Y_{t_{h-1}})^\top,
\]
which is known as the \textit{realized covariance matrix} in high-frequency financial econometrics, is a natural estimator for the quadratic covariance matrix of $Y$:
\[
[Y,Y]_1=\int_0^1\Sigma_tdt,\qquad\Sigma_t:=\sigma_t\sigma_t^\top.
\]
The aim of this section is to establish the asymptotic mixed normality of the estimator $\widehat{[Y,Y]}^n_1$ in a high-dimensional setting such that the dimension $d$ is possibly (much) larger than the sample size $n$. 

Before stating the result\tcr{s}, we introduce some notation. 
First, for a random variable $F$ taking values in $\mathbb{R}^{N_1\times\cdots\times N_q}$ for some $N_1,\dots,N_q\in\mathbb{N}$, we set $\|F\|_{p,\ell_2}:=\|\|F\|_{\ell_2}\|_p$ for every $p\in(0,\infty]$. 
Next, for a positive integer $k$, we identify the space $H^{\otimes k}$ with $L^2([0,1]^k;(\mathbb{R}^r)^{\otimes k})$ in the canonical way \tcr{(cf.~Example E.10 in \cite{Janson1997})}. Therefore, if a \tcr{univariate} random variable $F$ is $k$ times differentiable in the Malliavin sense, the $k$th Malliavin derivative $D^kF$ of $F$ \tcr{takes values in} $L^2([0,1]^k;(\mathbb{R}^r)^{\otimes k})$, so we can consider the value $D^kF(t_1,\dots,t_k)$ in $(\mathbb{R}^r)^{\otimes k}$ evaluated at $(t_1,\dots,t_k)\in[0,1]^k$. We denote this value by $D_{t_1,\dots,t_k}F$. Moreover, for an index $(a_1,\dots,a_k)\in\{1,\dots,r\}^k$, we write the $(a_1,\dots,a_k)$-th entry of $D_{t_1,\dots,t_k}F$ as $D^{(a_1,\dots,a_k)}_{t_1,\dots,t_k}F$ (note that we identify $(\mathbb{R}^r)^{\otimes k}$ with $\mathbb{R}^{r\times\cdots\times r}$). 
We remark that the variable $D_{t_1,\dots,t_k}F$ is defined only a.e.~on $[0,1]^k\times\Omega$ with respect to the measure $\leb_k\times P$, where $\leb_k$ denotes the Lebesgue measure on $[0,1]^k$. Therefore, if $D_{t_1,\dots,t_k}F$ satisfies some property a.e.~on $[0,1]^k\times\Omega$ with respect to the measure $\leb_k\times P$, by convention we will always take a version of $D_{t_1,\dots,t_k}F$ satisfying that property everywhere on $[0,1]^k\times\Omega$ if necessary. 
\tcr{
Also, note that if a $d$-dimensional random vector $F$ is $k$ times differentiable in the Malliavin sense, the $k$th Malliavin derivative $D^kF$ is first identified with the $(H^{\otimes k})^d$-valued random variable $(D^kF^1,\dots,D^kF^d)$ according to the identification of $H^{\otimes k}\otimes\mathbb{R}^d$ with $(H^{\otimes k})^d$ (cf.~Sections \ref{sec:basic} and \ref{sec:malliavin}). 
Then, each $D^kF^j$ is identified with the $L^2([0,1]^k;(\mathbb{R}^r)^{\otimes k})$-valued random variable as above. 
}

We define the $d^2\times d^2$ random matrix $\mathfrak{C}_n$ by
\begin{align*}
\mathfrak{C}_n^{(i-1)d+j,(k-1)d+l}:=
n\sum_{h=1}^n\left\{\left(\int_{t_{h-1}}^{t_h}\Sigma_s^{ik}ds\right)\left(\int_{t_{h-1}}^{t_h}\Sigma_s^{jl}ds\right)
+\left(\int_{t_{h-1}}^{t_h}\Sigma_s^{il}ds\right)\left(\int_{t_{h-1}}^{t_h}\Sigma_s^{jk}ds\right)\right\},\\
i,j,k,l=1,\dots,d,
\end{align*}
which plays the role of the conditional covariance matrix of the approximating mixed-normal distribution in our setting. 
\begin{rmk}
In the fixed dimensional setting, $\mathfrak{C}_n$ converges in probability as $n\to\infty$ to the random matrix $\bar{\mathfrak{C}}$ defined by
\begin{align*}
\bar{\mathfrak{C}}^{(i-1)d+j,(k-1)d+l}:=
\int_0^1\left(\Sigma_t^{ik}\Sigma_t^{jl}+\Sigma_t^{il}\Sigma_t^{jk}\right)dt,\qquad
i,j,k,l=1,\dots,d
\end{align*}
under mild regularity assumptions, so $\bar{\mathfrak{C}}$ plays the role of the asymptotic covariance matrix in such a setting. 
However, in the high-dimensional setting the convergence rate of $\mathfrak{C}_n$ to $\bar{\mathfrak{C}}$ does matter and we usually need an additional condition like \eqref{sigma-modulus} to derive it. 
To avoid such an extra assumption, we use the ``intermediate version'' $\mathfrak{C}_n$ of $\bar{\mathfrak{C}}$ to state Theorem \ref{thm:rc} below.  
\end{rmk}
%
\begin{theorem}\label{thm:rc}
Suppose that $\mu_t\in\mathbb{D}_{1,\infty}(\mathbb{R}^{d})$ and $\sigma_t\in\mathbb{D}_{2,\infty}(\mathbb{R}^{d\times r})$ for all $t\in[0,1]$. 
For every $n\in\mathbb{N}$, let $W_n\in\mathbb{D}_{2,\infty}(\mathbb{R}^{d^2})$, $\bs{X}_n\in\mathbb{D}_{2,\infty}(\mathbb{R}^{m\times d^2})$ and $\Upsilon_n$ be an $m\times d^2$ (deterministic) matrix such that $\opnorm{\Upsilon_n}_\infty\geq1$, where $m=m_n$ possibly depends on $n\in\mathbb{N}$.   
Define $\Xi_n:=\Upsilon_n\circ\bs{X}_n$ and assume
\begin{equation}\label{rc:diag}
\lim_{b\downarrow0}\limsup_{n\to\infty}P(\min\diag(\Xi_n\mathfrak{C}_n\Xi_n^\top)<b)=0.
\end{equation}
Then the following statements hold true:
\begin{enumerate}[label={\normalfont(\alph*)}]

\item Suppose that there is a constant $\varpi\in(0,\frac{1}{2})$ such that 
$\opnorm{\Upsilon_n}_\infty^5=O(n^\varpi)$ 
and
\begin{align}
&\sup_{n\in\mathbb{N}}
\max_{1\leq i\leq d^2}\left(
\|W^{i}_n\|_p
\tcr{+}\sup_{0\leq t\leq 1}\|D_tW_n^{i}\|_{p,\ell_2}
+\sup_{0\leq s,t\leq 1}\|D_{s,t}W_n^{i}\|_{p,\ell_2}
\right)
<\infty,\label{eq-W}\\
&\sup_{n\in\mathbb{N}}
\max_{1\leq i\leq m}\max_{1\leq j\leq d^2}\left(
\|\bs{X}^{ij}_n\|_p
\tcr{+}\sup_{0\leq t\leq 1}\|D_t\bs{X}_n^{ij}\|_{p,\ell_2}
+\sup_{0\leq s,t\leq 1}\|D_{s,t}\bs{X}_n^{ij}\|_{p,\ell_2}
\right)
<\infty,\label{eq-X}\\
&\sup_{n\in\mathbb{N}}
\max_{1\leq i\leq d}\sup_{0\leq t\leq 1}\left(
\|\mu_t^{i}\|_p
+\sup_{0\leq s,t\leq 1}\|D_s\mu_t^{i}\|_{p,\ell_2}
\right)
<\infty,\label{eq-mu}\\
&\sup_{n\in\mathbb{N}}
\max_{1\leq i\leq d}\sup_{0\leq t\leq 1}\left(
\|\Sigma_t^{ii}\|_p
+\sup_{0\leq s,t\leq 1}\|D_s\sigma_t^{i\cdot}\|_{p,\ell_2}
+\sup_{0\leq s,t,u\leq 1}\|D_{s,t}\sigma_u^{i\cdot}\|_{p,\ell_2}
\right)
<\infty\label{eq-sigma}
\end{align}
for all $p\in[1,\infty)$. 
Suppose also that $d=O(n^\mathfrak{c})$ and $m=O(n^\mathfrak{c})$ as $n\to\infty$ for some $\mathfrak{c}>0$. 
Then we have
\begin{equation}\label{rc:result}
\sup_{y\in\mathbb{R}^m}\left|P\left(\Xi_n\left(S_n+W_n\right)\leq y\right)-P(\Xi_n(\mathfrak{C}_n^{1/2}\zeta_n+W_n)\leq y)\right|\to0
\end{equation}
as $n\to\infty$, where
\[
S_n:=\vectorize\left[\sqrt{n}\left(\widehat{[Y,Y]}^n_1-[Y,Y]_1\right)\right]
\]
and $\zeta_n$ is a $d^2$-dimensional Gaussian vector independent of $\mathcal{F}$. 

\item Suppose that 
$\opnorm{\Upsilon_n}_\infty^5(\log dm)^{\frac{13}{2}}=o(\sqrt{n})$ 
as $n\to\infty$ and 
\eqref{eq-W}-\eqref{eq-sigma} are satisfied for $p=\infty$. 
Then we have \eqref{rc:result} as $n\to\infty$. 

\end{enumerate}
\end{theorem}

\begin{rmk}\label{rmk:rc}
We enumerate some remarks on the assumptions of Theorem \ref{thm:rc} in the following:
\begin{enumerate}[label=(\alph*),itemsep=5.5pt]

\item 
In typical applications of Theorem \ref{thm:rc}, we take $W_n\equiv0$ and $X_n$ a smooth functional of the volatility process $\sigma$. \tcr{H}ence only the assumptions on $\mu$ and $\sigma$ do matter (see also Section \ref{sec:factor}). 
The Malliavin differentiability conditions on $\mu$ and $\sigma$ are satisfied, for example, when $\mu$ and $\sigma$ are respectively solutions of stochastic differential equations (SDEs) with sufficiently regular coefficients; see e.g.~Section 2.2.2 of \cite{Nualart2006}. We remark that the (local) Malliavin differentiability has been known for solutions of some SDEs with irregular coefficients as well; see Section 4 of \citet{AE2008} and Lemma 5.9 of \citet{Naganuma2013} for example. 

\item 
A major restriction imposed by the assumptions of Theorem \ref{thm:rc} is that they require the arrays $(D_s^{(a)}\sigma_t^{ib})_{(a,b)\in[r]^2}$ and $(D_{s,t}^{(a,b)}\sigma_u^{ic})_{(a,b,c)\in[r]^2}$ are sufficiently ``sparse'' for all $s,t,u\in[0,1]$ so that 
\[
\sup_{0\leq s,t\leq 1}\|D_s\sigma_t^{i\cdot}\|_{p,\ell_2}\quad\text{and}\quad\sup_{0\leq s,t,u\leq 1}\|D_{s,t}\sigma_u^{i\cdot}\|_{p,\ell_2}
\] 
do not diverge as $n\to\infty$. This is a restriction because $r$ typically diverges as $n\to\infty$ in a high-dimensional setting. Such a condition is satisfied e.g.~when $Y^i$ and $(\sigma_t^{i\cdot})_{t\in[0,1]}$ depend on only finitely many components of $B$ for each $i$ (they may vary with $i$, though). Therefore, it is satisfied if the price and volatility processes have a certain factor structure, which seems realistic in financial applications.  

\item The Malliavin differentiability condition on $\mu$ in Theorem \ref{thm:rc} can be replaced by a continuity condition on $\mu$ analogous to \eqref{sigma-modulus}. In fact, it is used only to prove Lemma \ref{drift}, where it is only crucial that $\mu$ is well-approximated by a ``strongly predictable'' process.  

\item Assumptions on the second Malliavin derivatives of the volatility process $\sigma_t$ sometimes appear in high-frequency financial econometrics even for the fixed-dimensional case; see \cite{CG2011,CPTV2017} for example.  

\item The assumptions of Theorem \ref{thm:rc} do not rule out the possibility of the presence of jumps in the volatility process $\sigma$; see \citet{Fukasawa2011}. 

\item\label{rmk:moment} It would be enough in Theorem \ref{thm:rc}(a) to assume conditions \eqref{eq-W}--\eqref{eq-sigma} for some $p\in[1,\infty)$ only, where $p$ depends on the value of $\mathfrak{c}$, i.e.~the divergence rates of $d$ and $m$. 

\end{enumerate}
\end{rmk}

By an analogous discussion to the one before Corollary \ref{coro:main}, we can deduce a high-dimensional central limit theorem for realized covariance in hyperrectangles from Theorem \ref{thm:rc}:
\begin{corollary}\label{coro:rc}
Under the assumptions of Theorem \ref{thm:rc} with replacing \eqref{rc:diag} by
\[
\lim_{b\downarrow0}\limsup_{n\to\infty}P(\min\diag(\mathfrak{C}_n)<b)=0,
\]
we have
\[
\sup_{A\in\mathcal{A}^\mathrm{re}(d^2)}\left|P\left(S_n+W_n\in A\right)-P\left(\mathfrak{C}_n^{1/2}\zeta_n+W_n\in A\right)\right|\to0
\]
as $n\to\infty$. 
\end{corollary}

In some situations, it is more convenient to consider a localized version of the assumptions of Theorem \ref{thm:rc} as follows:
\begin{theorem}\label{thm:rc-local}
For every $n\in\mathbb{N}$, let $W_n$ be a $d^2$-dimensional random vector, $\bs{X}_n$ be an $m\times d^2$ random matrix and $\Upsilon_n$ be an $m\times d^2$ (deterministic) matrix such that $\opnorm{\Upsilon_n}_\infty\geq1$, where $m=m_n$ possibly depends on $n\in\mathbb{N}$.   
Moreover, for every $\nu\in\mathbb{N}$, let $\Omega_n(\nu)\in\mathcal{F}$, $\mu(\nu)=(\mu(\nu)_t)_{t\in[0,1]}$ be a $d$-dimensional $(\mathcal{F}_t)$-progressively measurable process, $\sigma(\nu)=(\sigma(\nu)_t)_{t\in[0,1]}$ be an $\mathbb{R}^{d\times r}$-valued $(\mathcal{F}_t)$-progressively measurable process, $W_n(\nu)\in\mathbb{D}_{2,\infty}(\mathbb{R}^{d^2})$ and $\bs{X}_n(\nu)\in\mathbb{D}_{2,\infty}(\mathbb{R}^{m\times d^2})$, and suppose that the following conditions are satisfied: 
\begin{enumerate}[label={\normalfont(\roman*)}]

\item $\lim_{\nu\to\infty}\limsup_{n\to\infty}P(\Omega_n(\nu)^c)=0$.

\item For all $\nu\in\mathbb{N}$ and $t\in[0,1]$, $\mu_t=\mu(\nu)_t$ and $\sigma_t=\sigma(\nu)_t$ on $\Omega_n(\nu)$ as well as $\mu(\nu)_t\in\mathbb{D}_{1,\infty}(\mathbb{R}^{d})$ and $\sigma(\nu)_t\in\mathbb{D}_{2,\infty}(\mathbb{R}^{d\times r})$.  

\item For all $\nu\in\mathbb{N}$, $W_n=W_n(\nu)$ and $\bs{X}_n=\bs{X}_n(\nu)$ on $\Omega_n(\nu)$.

\item For all $\nu\in\mathbb{N}$, \eqref{rc:diag} holds true with replacing $\bs{X}_n$ and $\sigma$ by $\bs{X}_n(\nu)$ and $\sigma(\nu)$ respectively. 

\end{enumerate}
Then the following statements hold true:
\begin{enumerate}[label={\normalfont(\alph*)}]

\item Suppose that there are constants $\varpi\in(0,\frac{1}{2})$ and $\mathfrak{c}>0$ such that $\opnorm{\Upsilon_n}_\infty^5=O(n^\varpi)$, $d=O(n^\mathfrak{c})$ and $m=O(n^\mathfrak{c})$ as $n\to\infty$. 
Suppose also that, for all $\nu\in\mathbb{N}$, \eqref{eq-W}-\eqref{eq-sigma} are satisfied for all $p\in[1,\infty)$ with replacing $W_n,\bs{X}_n,\mu,\sigma$ by $W_n(\nu),\bs{X}_n(\nu),\mu(\nu),\sigma(\nu)$ respectively. 
Then we have \eqref{rc:result} as $n\to\infty$. 

\item Suppose that 
$\opnorm{\Upsilon_n}_\infty^5(\log dm)^{\frac{13}{2}}=o(\sqrt{n})$ 
as $n\to\infty$ and, for all $\nu\in\mathbb{N}$, 
\eqref{eq-W}-\eqref{eq-sigma} are satisfied for $p=\infty$ with replacing $W_n,\bs{X}_n,\mu,\sigma$ by $W_n(\nu),\bs{X}_n(\nu),\mu(\nu),\sigma(\nu)$ respectively. 
Then we have \eqref{rc:result} as $n\to\infty$. 

\end{enumerate}
\end{theorem}

To make Theorems \ref{thm:rc}--\ref{thm:rc-local} statistically feasible, we need to estimate the ``asymptotic'' covariance matrix $\mathfrak{C}_n$. 
We can construct a ``consistent'' estimator for $\mathfrak{C}_n$ in the same way as in the low-dimensional setting of \citet{BNS2004rc}: Define the $d^2$-dimensional random vectors $\chi_h$ by
\[
\chi_{h}:=\vectorize\left[(Y_{t_h}-Y_{t_{h-1}})(Y_{t_h}-Y_{t_{h-1}})^\top\right],\qquad
h=1,\dots,n.
\]
Then we set
\[
\widehat{\mathfrak{C}}_n:=n\sum_{h=1}^n\chi_h\chi_h^\top-\frac{n}{2}\sum_{h=1}^{n-1}\left(\chi_h\chi_{h+1}^\top+\chi_{h+1}\chi_{h}^\top\right).
\]

\begin{proposition}\label{prop:acov}
For all $n\in\mathbb{N}$ and $\nu\in\mathbb{N}$, let $\Omega_n(\nu)\in\mathcal{F}$, $\mu(\nu)=(\mu(\nu)_t)_{t\in[0,1]}$ be a $d$-dimensional $(\mathcal{F}_t)$-progressively measurable process and $\sigma(\nu)=(\sigma(\nu)_t)_{t\in[0,1]}$ be an $\mathbb{R}^{d\times r}$-valued $(\mathcal{F}_t)$-progressively measurable process, and suppose that the following conditions are satisfied: 
\begin{enumerate}[label={\normalfont(\roman*)}]

\item $\lim_{\nu\to\infty}\limsup_{n\to\infty}P(\Omega_n(\nu)^c)=0$.

\item For all $\nu\in\mathbb{N}$ and $t\in[0,1]$, $\mu_t=\mu(\nu)_t$ and $\sigma_t=\sigma(\nu)_t$ on $\Omega_n(\nu)$ as well as $\sigma(\nu)_t\in\mathbb{D}_{1,\infty}(\mathbb{R}^{d\times r})$.  

\item There is a constant $\gamma\in(0,\frac{1}{2}]$ such that
\begin{equation}\label{sigma-modulus}
\sup_{0< t\leq1-\frac{1}{n}}\left\|\max_{1\leq k,l\leq d}\left|\Sigma(\nu)^{kl}_{t+\frac{1}{n}}-\Sigma(\nu)^{kl}_t\right|\right\|_2=O(n^{-\gamma})
\end{equation}
as $n\to\infty$, where $\Sigma(\nu)_t:=\sigma(\nu)_t\sigma(\nu)_t^\top$. 

\end{enumerate}
Then the following statements hold true:
\begin{enumerate}[label={\normalfont(\alph*)}]

\item Suppose that
\begin{equation}\label{acov-moment}
\sup_{n\in\mathbb{N}}
\max_{1\leq i\leq d}\sup_{0\leq t\leq 1}\left(
\|\mu(\nu)_t^i\|_p
+\|\Sigma(\nu)_t^{ii}\|_p
+\sup_{0\leq s,t\leq 1}\|D_s\sigma(\nu)_t^{i\cdot}\|_{p,\ell_2}
\right)
<\infty
\end{equation}
for all $p\in[1,\infty)$ and $\nu\in\mathbb{N}$. Suppose also that $d=O(n^{\mathfrak{c}})$ as $n\to\infty$ for some $\mathfrak{c}>0$. Then we have 
$
\|\wh{\mathfrak{C}}_n-\mathfrak{C}_n\|_{\ell_\infty}=O_p(n^{-\varpi})
$ 
as $n\to\infty$ for any $\varpi\in(0,\gamma)$.  

\item Suppose that \eqref{acov-moment} is satisfied for $p=\infty$. Then we have 
$
\|\wh{\mathfrak{C}}_n-\mathfrak{C}_n\|_{\ell_\infty}=O_p(n^{-1/2}\log^2 d+n^{-\gamma})
$ 
as $n\to\infty$.

\end{enumerate}
\end{proposition}

\begin{rmk}
It is presumably possible to remove the (local) Malliavin differentiability assumption on $\sigma_t$ from Proposition \ref{prop:acov} if we impose an additional condition on $d$ and $n^{-\gamma}$ (such an additional assumption will be even unnecessary to prove the part (a) only, but we keep that condition to prove two claims in a unified way). 
\end{rmk}


When the dimension $d$ is very large, computation of $\wh{\mathfrak{C}}_n^{1/2}$ is practically challenging, \tcr{so} it is better to employ a (wild) bootstrap to generate random vectors having the same distributions as that of $\wh{\mathfrak{C}}_n^{1/2}\zeta_n$ as follows.  
Let $(e_h)_{h=1}^\infty$ be a centered Gaussian process independent of $\mathcal{F}$, which is defined on an extension of $(\Omega,\mathcal{F},P)$ if necessary. Then we define
\[
S_n^*:=\sqrt{n}\sum_{h=1}^ne_h\chi_h.
\]
The Gaussian process $(e_h)_{h=1}^\infty$ must have an appropriate covariance matrix so that the $\mathcal{F}$-conditional \tcr{covariance} matrix of $S_n^*$ mimics $\wh{\mathfrak{C}}_n$. As is well-known in the literature (see e.g.~\cite{Hounyo2014}), the standard i.i.d.~wild bootstrap fails to approximate the joint distributions of statistics in the present context.\footnote{It is also known that empirical bootstrap fails in the present context as well; see e.g.~\cite{DGM2013} for a discussion.} 
Alternatively, we assume that $(e_h)_{h=1}^\infty$ is stationary with auto-covariance function
\[
E[e_he_{h+\ell}]
=\left\{
\begin{array}{cl}
1  & \text{if }\ell=0,   \\
-\frac{1}{2}  & \text{if }\ell=1, \\
0  & \text{otherwise}. 
\end{array}
\right.
\]
Then we can easily check that the $\mathcal{F}$-conditional covariance matrix of $S_n^*$ is equal to $\wh{\mathfrak{C}}_n$, \tcr{so} $S_n^*$ has the same distribution as that of $\wh{\mathfrak{C}}_n^{1/2}\zeta_n$.  
We remark that such a sequence $(e_h)_{h=1}^\infty$ considered above can be generated by the following Gaussian MA(1) process:
\[
e_h=\eta^*_h-\eta^*_{h-1},\qquad h=1,\dots,n,
\]
where $(\eta^*_h)_{h=0}^n$ is a sequence of i.i.d.~centered Gaussian variables with variance $\frac{1}{2}$. Therefore, we can rewrite $S_n^*$ as 
\[
S_n^*=\sqrt{n}\sum_{h=1}^{n-1}\eta^*_h(\chi_h-\chi_{h+1})+\sqrt{n}(\eta^*_n\chi_n-\eta^*_0\chi_1).
\]
The second term on the right side of the above equation is usually asymptotically negligible, so the bootstrap procedure considered here is essentially the same as the \textit{wild blocks of blocks bootstrap} proposed in \citet{Hounyo2014}. 


\subsection{Testing the residual sparsity of a continuous-time factor model}\label{sec:factor}

As an application of the theory developed above, we consider the problem of testing the correlation structure of the residual process of a continuous-time factor model. This problem was investigated in Section 4 of \citet{BM2016} for the case of two assets, and we are aim at extending their analysis to a multiple assets situation. 
Specifically, we suppose that the $d$-th asset $Y^d$ is regarded as an observable factor and consider the following continuous-time factor model:
\begin{equation}\label{factor-model}
Y^j=\beta^jY^d+R^j,\qquad j=1,\dots,\underline{d}:=d-1.
\end{equation}
Here, $\beta^j$ is a constant and $R^j$ is a semimartingale such that $[R^j,Y^d]\equiv0$. 
Let us set $\Lambda_n:=\{(i,j):1\leq i<j\leq \underline{d}\}$. 
For each $(i,j)\in\Lambda_n$, we consider the following hypothesis testing problem:
\begin{equation}\label{rs-test}
H_0^{(i,j)}:[R^i,R^j]_1=0\quad\text{a.s.}\qquad
\text{vs}\qquad
H_1^{(i,j)}:[R^i,R^j]_1\neq0\quad\text{a.s.}
\end{equation}
Our aim is to test the hypothesis \eqref{rs-test} simultaneously for $(i,j)\in\Lambda_n$, but we start with constructing a test statistic for a fixed $(i,j)\in\Lambda_n$. 
For notational convenience, we construct the test statistic for every pair $(i,j)$ in $\{1,\dots,\ul{d}\}^2$. 

\begin{rmk}[Sparsity test of the quadratic covariation matrix itself]
Considering the case $Y^d\equiv0$, we have $R^i=Y^i$ for all $i=1,\dots,\ul{d}$. \tcr{H}ence the problem turns to multiple testing for the hypotheses \eqref{test:sparse}. 
\end{rmk}
 
We follow \cite{BM2016} and consider the following statistic
\[
\mathfrak{T}^{ij}:=[Y^i,Y^d]_1[Y^j,Y^d]_1-[Y^i,Y^j]_1[Y^d,Y^d]_1,
\]
which is zero under $H_0^{(i,j)}$.  
Therefore, it is natural to consider the estimated version of $\mathfrak{T}^{ij}$ as follows:
\[
\hat{\mathfrak{T}}^{ij}_n:=\widehat{[Y^i,Y^d]}^n_1\widehat{[Y^j,Y^d]}^n_1-\widehat{[Y^i,Y^j]}^n_1\widehat{[Y^d,Y^d]}^n_1.
\]
In order to make the test statistic scale invariant, we consider the Studentized version of $\hat{\mathfrak{T}}^{ij}_n$. According to \cite{BM2016}, the ``asymptotic variance'' of $\hat{\mathfrak{T}}^{ij}_n$ is given by the following statistic: 
\begin{align*}
\mathfrak{V}^{ij}_n&:=[Y^j,Y^d]_1^2\mathfrak{C}_n^{id,id}+[Y^i,Y^d]_1^2\mathfrak{C}_n^{jd,jd}
+[Y^i,Y^j]_1^2\mathfrak{C}_n^{d^2,d^2}+[Y^d,Y^d]_1^2\mathfrak{C}_n^{(i-1)d+j,(i-1)d+j}\\
&\quad+2[Y^d,Y^d]_1[Y^i,Y^j]_1\mathfrak{C}_n^{(i-1)d+j,d^2}
+2[Y^i,Y^d]_1[Y^j,Y^d]_1\mathfrak{C}_n^{id,jd}\\
&\quad-2[Y^i,Y^d]_1[Y^d,Y^d]_1\mathfrak{C}_n^{(i-1)d+j,jd}
-2[Y^j,Y^d]_1[Y^d,Y^d]_1\mathfrak{C}_n^{(i-1)d+j,id}\\
&\quad-2[Y^i,Y^j]_1[Y^i,Y^d]_1\mathfrak{C}_n^{jd,d^2}
-2[Y^i,Y^j]_1[Y^j,Y^d]_1\mathfrak{C}_n^{id,d^2}.
\end{align*}
Let us denote by $\hat{\mathfrak{V}}^{ij}_n$ the estimated version of $\mathfrak{V}^{ij}_n$, i.e.~we define $\hat{\mathfrak{V}}^{ij}_n$ by the right side of the above equation with replacing $[Y,Y]_1$ and $\mathfrak{C}_n$ by $\widehat{[Y,Y]}_1$ and $\wh{\mathfrak{C}}_n$, respectively. Then we define the test statistic by
\[
T_n^{(i,j)}:=\frac{\sqrt{n}\hat{\mathfrak{T}}^{ij}_n}{\sqrt{\hat{\mathfrak{V}}^{ij}_n}}.
\]
The statistic $T_n^{(i,j)}$ is generally uncentered unless the null hypothesis $H_0^{(i,j)}$ is true, and it is convenient to consider the centered version of $T_n^{(i,j)}$ in the general situation as follows:
\[
\tilde{T}_n^{(i,j)}:=\frac{\sqrt{n}\left(\hat{\mathfrak{T}}^{ij}_n-\mathfrak{T}^{ij}_n\right)}{\sqrt{\hat{\mathfrak{V}}^{ij}_n}}.
\]
Note that we can rewrite $\hat{\mathfrak{T}}^{ij}_n-\mathfrak{T}^{ij}_n$ as
\begin{multline*}
\hat{\mathfrak{T}}^{ij}_n-\mathfrak{T}^{ij}_n
=\left(\widehat{[Y^i,Y^d]}^n_1-[Y^i,Y^d]_1\right)\widehat{[Y^j,Y^d]}^n_1
+[Y^i,Y^d]_1\left(\widehat{[Y^j,Y^d]}^n_1-[Y^j,Y^d]_1\right)\\
-\left(\widehat{[Y^i,Y^j]}^n_1-[Y^i,Y^j]_1\right)\widehat{[Y^d,Y^d]}^n_1
-[Y^i,Y^j]_1\left(\widehat{[Y^d,Y^d]}^n_1-[Y^d,Y^d]_1\right).
\end{multline*}
Therefore, a bootstrapped version of $\tilde{T}_n^{(i,j)}$ is defined as
\[
T_{n,*}^{(i,j)}:=\frac{\sqrt{n}\hat{\mathfrak{T}}^{ij}_{n,*}}{\sqrt{\hat{\mathfrak{V}}^{ij}_n}},
\]
where
\[
\hat{\mathfrak{T}}^{ij}_{n,*}:=\widehat{[Y^i,Y^d]}^{n,*}_1\widehat{[Y^j,Y^d]}^n_1
+\widehat{[Y^i,Y^d]}^{n}_1\widehat{[Y^j,Y^d]}^{n,*}_1
-\widehat{[Y^i,Y^j]}^{n,*}_1\widehat{[Y^d,Y^d]}^{n}_1
-\widehat{[Y^i,Y^j]}^{n}_1\widehat{[Y^d,Y^d]}^{n,*}_1
\]
and
\[
\widehat{[Y^i,Y^j]}^{n,*}_1:=\sqrt{n}\sum_{h=1}^ne_h(Y^i_{t_h}-Y^i_{t_{h-1}})(Y^j_{t_h}-Y^j_{t_{h-1}}),\qquad
i,j=1,\dots,d.
\]

Set $\tilde{T}_n=(\tilde{T}_n^{(i,j)})_{1\leq i,j\leq\ul{d}}$ and $T_{n,*}=(T_{n,*}^{(i,j)})_{1\leq i,j\leq\ul{d}}$. 
We derive mixed-normal approximations for $\vectorize(\tilde{T}_n)$ and $\vectorize(T_{n,*})$ by applying the theory developed above. For this purpose we define the $\ul{d}^2\times d^2$ random matrix $\bs{X}_n$ by
\[
\bs{X}_n^{(i-1)\ul{d}+j,(k-1)d+l}=
\left\{
\begin{array}{cl}
~[Y^j,Y^d]_1/\sqrt{\mathfrak{V}_n^{ij}}  & \text{if }k=i,~l=d,  \\
~[Y^i,Y^d]_1/\sqrt{\mathfrak{V}_n^{ij}}  & \text{if }k=j,~l=d, \\
~-[Y^d,Y^d]_1/\sqrt{\mathfrak{V}_n^{ij}}  & \text{if }k=l=d,  \\
~-[Y^i,Y^j]_1/\sqrt{\mathfrak{V}_n^{ij}} & \text{if }k=i, l=j,\\
0 & \text{otherwise}   
\end{array}
\right.
\]
for $i,j=1,\dots,\ul{d}$ and $k,l=1,\dots,d$. Note that the statistics $\vectorize(\tilde{T}_n)$ and $\vectorize(T_{n,*})$ can be approximated by $\bs{X}_nS_n$ and $\bs{X}_nS_n^*$, respectively. 
We then obtain the following result. 
\begin{proposition}\label{prop:factor-test}
%
For all $n\in\mathbb{N}$ and $\nu\in\mathbb{N}$, let $\Omega_n(\nu)\in\mathcal{F}$, $\mu(\nu)=(\mu(\nu)_t)_{t\in[0,1]}$ be a $d$-dimensional $(\mathcal{F}_t)$-progressively measurable process and $\sigma(\nu)=(\sigma(\nu)_t)_{t\in[0,1]}$ be an $\mathbb{R}^{d\times r}$-valued $(\mathcal{F}_t)$-progressively measurable process, and suppose that the following conditions are satisfied: 
\begin{enumerate}[label={\normalfont(\roman*)}]

\item $\lim_{\nu\to\infty}\limsup_{n\to\infty}P(\Omega_n(\nu)^c)=0$.

\item For all $\nu\in\mathbb{N}$ and $t\in[0,1]$, $\mu_t=\mu(\nu)_t$ and $\sigma_t=\sigma(\nu)_t$ on $\Omega_n(\nu)$ as well as $\sigma(\nu)_t\in\mathbb{D}_{1,\infty}(\mathbb{R}^{d\times r})$.  

\item For all $p\in[1,\infty)$, it holds that
\begin{align*}
&\sup_{n\in\mathbb{N}}
\max_{1\leq i\leq d}\sup_{0\leq t\leq 1}\left(
\|\mu_t^{i}\|_p
+\sup_{0\leq s,t\leq 1}\|D_s\mu(\nu)_t^{i}\|_{p,\ell_2}
\right)
<\infty,\\
&\sup_{n\in\mathbb{N}}
\max_{1\leq i\leq d}\sup_{0\leq t\leq 1}\left(
\|\Sigma(\nu)_t^{ii}\|_p
+\sup_{0\leq s,t\leq 1}\|D_s\sigma(\nu)_t^{i\cdot}\|_{p,\ell_2}
+\sup_{0\leq s,t,u\leq 1}\|D_{s,t}\sigma(\nu)_u^{i\cdot}\|_{p,\ell_2}
\right)
<\infty,
\end{align*}
where $\Sigma(\nu)_t:=\sigma(\nu)_t\sigma(\nu)_t^\top$.

\item There is a constant $\gamma\in(0,\frac{1}{2}]$ such that \eqref{sigma-modulus} holds true as $n\to\infty$. 

\item  For all $p\in[1,\infty)$, it holds that
\begin{equation}\label{v-inv-moment}
\sup_{n\in\mathbb{N}}\max_{1\leq i,j\leq \ul{d}}E\left[\left(\mathfrak{V}_n(\nu)^{ij}\right)^{-p}\right]<\infty,
\end{equation}
where $\mathfrak{V}_n(\nu)$ is defined analogously to $\mathfrak{V}_n$ with replacing $\Sigma$ by $\Sigma(\nu)$. 

\end{enumerate}
Then we have
\[
\sup_{A\in\mathcal{A}^\mathrm{re}(\underline{d}^2)}\left|P\left(\vectorize\left(\tilde{T}_n\right)\in A\right)-P\left(\bs{X}_n\mathfrak{C}_n^{1/2}\zeta_n\in A\right)\right|\to0
\]
and
\[
\sup_{A\in\mathcal{A}^\mathrm{re}(\underline{d}^2)}\left|P\left(\vectorize\left(T_{n,*}\right)\in A|\mathcal{F}\right)-P\left(\bs{X}_n\mathfrak{C}_n^{1/2}\zeta_n\in A|\mathcal{F}\right)\right|\to^p0
\]
as $n\to\infty$, provided that $d=O(n^{\mathfrak{c}})$ as $n\to\infty$ for some $\mathfrak{c}>0$. 
\end{proposition}

Now we return to the problem of testing \eqref{rs-test} simultaneously for $(i,j)\in\Lambda_n$. 
Here, we consider a more general setting described in the following for the purposes of application (cf.~Section \ref{sec:empirical}). We suppose that the set $\Lambda_n$ is decomposed into non-empty disjoint sets $\Lambda_n^{1},\dots,\Lambda_n^{\mathsf{L}}$ as $\Lambda_n=\bigcup_{\ell=1}^{\mathsf{L}}\Lambda_n^{\ell}$. We consider the problem of testing 
\begin{equation}\label{test:group}
\bigwedge_{\lambda\in\Lambda_n^{\ell}}H_0^\lambda\qquad\text{vs}\qquad
\bigvee_{\lambda\in\Lambda_n^{\ell}}H_1^\lambda 
\end{equation}
simultaneously for $\ell=1,\dots,\mathsf{L}$. Here, for a subset $\mathcal{L}$ of $\Lambda_n$, $\bigwedge_{\lambda\in\mathcal{L}}H_0^\lambda$ (resp.~$\bigvee_{\lambda\in\mathcal{L}}H_1^\lambda$) denotes the hypothesis that $H_0^\lambda$ is true for all $\lambda\in\mathcal{L}$ (resp.~$H_1^\lambda$ is true for some $\lambda\in\mathcal{L}$). 
For simplicity of notation, we set $\mathsf{H}_0^\ell:=\bigwedge_{\lambda\in\Lambda_n^{\ell}}H_0^\lambda$ and $\mathsf{H}_1^\ell:=\bigvee_{\lambda\in\Lambda_n^{\ell}}H_1^\lambda$.  
If we let $\mathsf{L}$ be the number of elements in $\Lambda_n$ and write $\Lambda_n=\{\lambda_1,\dots,\lambda_{\mathsf{L}}\}$ and set $\Lambda_n^{\ell}=\{\lambda_\ell\}$ for $\ell=1,\dots,\mathsf{L}$, we recover the original problem of testing \eqref{rs-test} simultaneously for $(i,j)\in\Lambda_n$. 

Our aim is the strong control of the family-wise error rate (FWER) in this problem. 
More formally, let $\Theta_n$ be a set of pairs $(\mu,\sigma)$ of coefficient processes, which is considered as the set of all data generating processes we are interested in (note that the data generating process may vary with $n$ mainly because the dimensions $d$ and $r$ may depend on $n$). 
For each $\theta\in\Theta_n$, we denote by $\mathcal{L}_n(\theta)$ the set of all indices $\ell\in\{1,\dots,\mathsf{L}\}$ for which the hypothesis $\mathsf{H}_0^\ell$ holds true when $\theta$ is the true data generating process. 
Then, the FWER for $\theta\in\Theta_n$, which is denoted by $\FWER(\theta)$, is defined as the probability that $\mathsf{H}_0^\ell$ for some $\ell\in\mathcal{L}_n(\theta)$ is rejected when $\theta$ is the true data generating process. 
Given the significance level $\alpha\in(0,1)$, we aim at constructing multiple testing procedures such that
\begin{equation}\label{eq:fwer}
\limsup_{n\to\infty}\FWER(\theta_n)\leq\alpha
\end{equation}
for any sequence $\theta_n\in\Theta_n$ ($n=1,2,\dots$) of data generating processes. 
To accomplish this, we employ the stepdown procedure of \citet{RW2005} which we describe in the following. 
First, given a fixed index $\ell$, we shall use the test statistic $\mathsf{T}_n^{\ell}:=\max_{\lambda\in\Lambda_n^{\ell}}|T_n^\lambda|$ for the problem \eqref{test:group}. 
Next, we sort the observed test statistics in descending order and denote them by
\[
\mathsf{T}_n^{\ell_1}\geq\cdots \geq \mathsf{T}_n^{\ell_{\mathsf{L}}}.
\]
Also, for every subset $\mathcal{L}\subset\{1,\dots,\mathsf{L}\}$, suppose that we have a critical value $c_n^{\mathcal{L}}(1-\alpha)$ to test the null $\bigwedge_{\lambda\in\mathcal{L}}H_0^\lambda$ against the alternative $\bigvee_{\lambda\in\mathcal{L}}H_1^\lambda$. Those critical values can be random variables and will be specified later. 
Then the stepdown procedure reads as follows:
\begin{enumerate}

\item Let $\mathcal{L}_1:=\{1,\dots,\mathsf{L}\}$. If $\mathsf{T}_n^{\ell_1}\leq c_n^{\mathcal{L}_1}(1-\alpha)$, then accept all the hypotheses and stop; otherwise, reject $\mathsf{H}_0^{\ell_1}$ and continue. 

\item Let $\mathcal{L}_2:=\mathcal{L}_1\setminus\{\ell_1\}$. If $\mathsf{T}_n^{\ell_2}\leq c_n^{\mathcal{L}_2}(1-\alpha)$, then accept all the hypotheses $\mathsf{H}_0^{\ell}$ for $\ell\in\mathcal{L}_2$ and stop; otherwise, reject $\mathsf{H}_0^{\ell_2}$ and continue.

$\vdots$ 

\item[$k$.] Let $\mathcal{L}_k:=\mathcal{L}_{k-1}\setminus\{\ell_{k-1}\}$. If $\mathsf{T}_n^{\ell_k}\leq c_n^{\mathcal{L}_k}(1-\alpha)$, then accept all the hypotheses $\mathsf{H}_0^{\ell}$ for $\ell\in\mathcal{L}_k$ and stop; otherwise, reject $\mathsf{H}_0^{\ell_k}$ and continue.

$\vdots$ 

\item[$\mathsf{L}$.] If $\mathsf{T}_n^{\lambda_{\mathsf{L}}}\leq c_n^{\{\ell_{\mathsf{L}}\}}(1-\alpha)$, then accept $\mathsf{H}_0^{\ell_\mathsf{L}}$; otherwise, reject $\mathsf{H}_0^{\ell_\mathsf{L}}$.

\end{enumerate}
According to Theorem 3 of \cite{RW2005}, the above stepdown procedure satisfies \eqref{eq:fwer} if the critical values $c_n^{\mathcal{L}}(1-\alpha)$, $\mathcal{L}\subset\{1,\dots,\mathsf{L}\}$, satisfy the following conditions:
\begin{enumerate}[label=(\roman*)]

\item\label{monotone} $c_n^{\mathcal{L}}(1-\alpha)\leq c_n^{\mathcal{L}'}(1-\alpha)$ whenever $\mathcal{L}\subset\mathcal{L}'\subset\{1,\dots,\mathsf{L}\}$.

\item\label{max-quantile} For any sequence $\theta_n\in\Theta_n$ ($n=1,2,\dots$), it holds that
\[
\limsup_{n\to\infty}P\left(\max_{\ell\in\mathcal{L}_n(\theta_n)}\mathsf{T}_n^\ell >c_n^{\mathcal{L}_n(\theta_n)}(1-\alpha)\right)\leq\alpha
\]
whenever $\theta_n$ is the true data generating process for every $n$. 

\end{enumerate}
%
The first method to construct the desired critical values is the well-known \textit{Bonferroni-Holm method}. Namely, we set $c_n^{\mathcal{L}}(1-\alpha):=q_{N(0,1)}(1-\alpha/(2\#[\bigcup_{\ell\in\mathcal{L}}\Lambda_n^{\ell}]))$ for every $\mathcal{L}\subset\Lambda$, where $q_{N(0,1)}$ denotes the quantile function of the standard normal distribution and $\#[\bigcup_{\ell\in\mathcal{L}}\Lambda_n^{\ell}]$ is the number of elements in $\bigcup_{\ell\in\mathcal{L}}\Lambda_n^{\ell}$. 
The second method is to use the $(1-\alpha)$-quantile of $\max_{\ell\in\mathcal{L}}\mathsf{T}_n^\ell$. Of course, we cannot analytically compute the quantiles of $\max_{\ell\in\mathcal{L}}\mathsf{T}_n^\ell$ in general, so we approximate them by resampling as in \cite{RW2005,CCK2013}. Formally, setting $\mathsf{T}_{n,*}^{\ell}:=\max_{\lambda\in\Lambda_n^{\ell}}|T_{n,*}^{\lambda}|$, we use the $\mathcal{F}$-conditional $(1-\alpha)$-quantile of $\max_{\ell\in\mathcal{L}}\mathsf{T}_{n,*}^{\ell}$ as $c_n^{\mathcal{L}}(1-\alpha)$, which can be evaluated by simulation. We refer to this method as the \textit{Romano-Wolf method} in the following. 
\begin{corollary}\label{coro:testing}
Suppose that the assumptions of Proposition \ref{prop:factor-test} are satisfied for any sequence $(\mu,\sigma)=(\mu^{(n)},\sigma^{(n)})\in\Theta_n$ ($n=1,2,\dots$) of data generating processes whenever $(\mu^{(n)},\sigma^{(n)})$ is the true data generating process for every $n$. 
Then, both the Bonferroni-Holm and Romano-Wolf methods satisfy conditions \ref{monotone}--\ref{max-quantile}, \tcr{so} \eqref{eq:fwer} holds true.
\end{corollary}

\begin{rmk}
The Romano-Wolf method takes account of the dependence structure of the test statistics while the Bonferroni-Holm method ignores it, \tcr{so} the former is generally more powerful than the latter, especially when the test statistics are strongly dependent on each other. Meanwhile, we need no resampling to implement the Bonferroni-Holm method, \tcr{so} it is computationally more attractive than the Romano-Wolf method.  
\end{rmk}

\begin{rmk}[Application to threshold selection in covariance estimation]
Another possible application of Theorem \ref{thm:rc} would be selection of the thresholds in high-dimensional quadratic covariation estimation from high-frequency data (see e.g.~\citet{WZ2010} for such an estimation method): We refer to Section 4.1 of \citet{Chen2017} for details on such an application in the case of i.i.d.~observations.  
\end{rmk}

\section{Simulation study and an empirical illustration}\label{sec:simulation}

In this section we present a small Monte Carlo study to assess the finite sample performance of the multiple testing procedures proposed in Section \ref{sec:factor}. We also demonstrate how the proposed methodology works in a real world using high-frequency data from the components of the S\&P 100 index. 

\subsection{Simulations}

We focus on the problem of testing the hypotheses \eqref{rs-test} simultaneously for $(i,j)\in\Lambda_n$. 
The simulation design is basically adopted from \cite{FFX2016}, but we include only the first factor representing the market factor in our model. 
Specifically, we simulate model \eqref{factor-model} with the following specification\footnote{
One can show that the volatility process $\sigma$ generated by \eqref{heston} \textit{locally} satisfy the condition \eqref{eq-sigma} for any $p\in[1,\infty)$ as long as the Feller condition $2\kappa\theta>\eta^2$ is satisfied. 
In fact, one can show this by setting 
$\Omega_n(\nu):=\{\inf_{t\in[0,1]}\sigma_t\geq\nu^{-1}\}$ 
and taking smoothed versions of $\sigma_t$ analogous to the one considered in \cite{AE2008} as $\sigma(\nu)$'s for $\nu=1,2,\dots$. 
}:
\[
dY^d_t=\mu dt+\sqrt{v_t}dB^d_t,\qquad
dR^j_t=\gamma_j^\top d\ul{B}_t\quad(j=1,\dots,\ul{d})
\]
and
\begin{equation}\label{heston}
dv_t=\kappa(\theta-v_t)dt+\eta\sqrt{v_t}\left(\rho dB^d_t+\sqrt{1-\rho^2}dB^{d+1}_t\right).
\end{equation}
Here, $\mu,\kappa,\theta,\eta$ and $\rho$ are constants, $\ul{B}_t=(B^1_t,\dots,B^{\ul{d}}_t)$, and $\gamma_1,\dots,\gamma_{\ul{d}}$ are $\ul{d}$-dimensional random vectors independent of $B$. 
The values of $\beta^1,\dots,\beta^{\ul{d}}$ are independently drawn from the uniform distribution on $[0.25,2.25]$. 
We set $\mu=0.05$, $\kappa=3$, $\theta=0.09$, $\eta=0.3$ and $\rho=-0.6$. The initial value $v_0$ is drawn from the stationary distribution of the process $(v_t)_{t\in[0,1]}$, i.e.~the gamma distribution with shape $2\kappa\theta/\eta^2$ and rate $2\kappa/\eta^2$. 
We assume that $\Gamma:=(\gamma_i^\top\gamma_j)_{1\leq i,j\leq \ul{d}}$ is a block diagonal matrix with 10 blocks of size $(\ul{d}/10)\times(\ul{d}/10)$ whose diagonals are uniformly generated from $[0.2,0.5]$ and the corresponding correlation matrices have the constant correlation of $\rho_\gamma$. 
We set $\ul{d}=100$ and vary $\rho_\gamma$ as $\rho_\gamma\in\{0.25,0.5,0.75\}$. 

For each scenario, we compute the FWERs and the average powers (i.e.~the average probabilities of rejecting the false null hypotheses) of the Bonferroni-Holm and Romano-Wolf methods at the 5\% level based on 10,000 Monte Carlo iterations respectively. 
Here, we generate 999 bootstrap resamples for the Romano-Wolf method. 

Tables \ref{table:fwer} and \ref{table:power} report the results. 
\tcr{We see from Table \ref{table:fwer} that both the methods succeed in controlling the FWERs under the nominal level 5\%, although both are rather conservative.} 
Table \ref{table:power} shows that the average powers in both the methods tend to 1 as $n$ and $\rho_\gamma$ increase. 
The table also reveals that the Romano-Wolf method is more powerful than the Bonferroni-Holm method. As expected, the difference of the average powers between two methods becomes pronounced as the correlation $\rho_\gamma$ of the residual processes increases.

\begin{table}[ht]
\centering
\caption{Family-wise error rates at the 5\% level} 
\label{table:fwer}
\begin{tabular}{lrrrrrr}
  \hline
rn & $n=26$ & $n=39$ & $n=78$ & $n=130$ & $n=195$ & $n=390$ \\ 
  \hline
$\rho_\gamma=0.25$ &  &  &  &  &  &  \\ 
  Holm & 0.010 & 0.004 & 0.002 & 0.003 & 0.008 & 0.018 \\ 
  RW & 0.022 & 0.007 & 0.003 & 0.004 & 0.009 & 0.018 \\ 
  $\rho_\gamma=0.50$ &  &  &  &  &  &  \\ 
  Holm & 0.009 & 0.004 & 0.002 & 0.003 & 0.007 & 0.017 \\ 
  RW & 0.023 & 0.008 & 0.004 & 0.005 & 0.009 & 0.019 \\ 
  $\rho_\gamma=0.75$ &  &  &  &  &  &  \\ 
  Holm & 0.008 & 0.003 & 0.002 & 0.003 & 0.006 & 0.010 \\ 
  RW & 0.026 & 0.011 & 0.006 & 0.008 & 0.014 & 0.023 \\ 
   \hline
\end{tabular}\vspace{5mm}

\parbox{12cm}{\small 
\textit{Note}. This table reports the family-wise error rates at the 5\% level of multiple testing for the hypotheses \eqref{rs-test} by the Bonferroni-Holm (BH) and Romano-Wolf (RW) methods, respectively. The reported values are based on 10,000 Monte Carlo iterations. 999 bootstrap resamples are generated to implement the RW method.} 
\end{table}

\begin{table}[ht]
\centering
\caption{Average powers at the 5\% level} 
\label{table:power}
\begin{tabular}{lrrrrrr}
  \hline
rn & $n=26$ & $n=39$ & $n=78$ & $n=130$ & $n=195$ & $n=390$ \\ 
  \hline
$\rho_\gamma=0.25$ &  &  &  &  &  &  \\ 
  Holm & 0.000 & 0.000 & 0.000 & 0.005 & 0.046 & 0.563 \\ 
  RW & 0.000 & 0.000 & 0.000 & 0.006 & 0.048 & 0.567 \\ 
  $\rho_\gamma=0.50$ &  &  &  &  &  &  \\ 
  Holm & 0.000 & 0.001 & 0.028 & 0.421 & 0.950 & 1.000 \\ 
  RW & 0.000 & 0.001 & 0.037 & 0.458 & 0.956 & 1.000 \\ 
  $\rho_\gamma=0.75$ &  &  &  &  &  &  \\ 
  Holm & 0.001 & 0.007 & 0.262 & 0.953 & 1.000 & 1.000 \\ 
  RW & 0.004 & 0.017 & 0.393 & 0.977 & 1.000 & 1.000 \\ 
   \hline
\end{tabular}\vspace{5mm}

\parbox{12cm}{\small 
\textit{Note}. This table reports the average powers at the 5\% level of multiple testing for the hypotheses \eqref{rs-test} by the Bonferroni-Holm (BH) and Romano-Wolf (RW) methods, respectively. The reported values are based on 10,000 Monte Carlo iterations. 999 bootstrap resamples are generated to implement the RW method.} 
\end{table}

\subsection{Empirical illustration}\label{sec:empirical}

We apply our methodology to high-frequency returns of the components of the S\&P 100 index while taking the SPDR S\&P 500 ETF (SPY) as the observable factor process. 
The sample period is the one month, March 2018, and we regard this period as the interval $[0,1]$ (over-night returns are ignored). 
The data are provided by Bloomberg. 
Following \citet{FFX2016}, we use 15 minute returns to avoid notable market microstructure effects. 
To illuminate the block diagonal structure reported in \cite{FFX2016}, we sort the assets by their Global Industry Classification Standard (GICS) sectors while we construct the log-price processes $Y^j$, $j=1,\dots,\ul{d}$. 

We begin by examining the sparsity of the quadratic covariation matrix of the assets without taking account of the factor process. The top panel of Figure \ref{fig:sparse} shows the corresponding realized correlation matrix. Here, we perform multiple testing for the hypotheses \eqref{test:sparse} using the Romano-Wolf method with 999 bootstrap resamples and change the entries for which the null hypotheses are not rejected at the 5\% level to blanks. 
The violet squares indicate GICS sector classifications. Namely, all assets in the same square belong to the same sector.  
We clearly find that the raw realized correlation matrix is far from sparse, i.e.~most the entries are not blank. In fact, our test suggests that about \tcr{90.9}\% pairs would have significant correlations at the 5\% level. 
Meanwhile, the bottom panel of Figure \ref{fig:sparse} shows the realized correlation matrix of the residual processes of the assets regressed on SPY. 
Again, we perform multiple testing for the hypotheses \eqref{rs-test} as above to change the entries with insignificant correlations to blanks. 
The violet squares have the same meaning as above. 
In contrast to the first case, the realized correlation matrix exhibits the remarkable diagonal structure inherited from the assets' sectors. In this case only about \tcr{4.3}\% pairs are significantly correlated at the 5\% level. 

To investigate this \tcr{block} diagonal structure more deeply, we conduct another multiple testing for the absence of covariations within and between sectors after regressing assets on SPY. 
Formally, let $G_1,\dots,G_N$ be all the sectors, then we set $I_k:=\{i\in\{1,\dots,\ul{d}\}:\text{ the $i$-th asset $Y^i$ belongs to the sector $G_k$}\}$ for every $k=1,\dots,N$ and $\Lambda_n^{(k,l)}:=\Lambda_n\cap(I_k\times I_l)$ for all $k,l=1,\dots,N$. 
We test the null hypothesis $\bigwedge_{\lambda\in\Lambda_n^{(k,l)}}H_0^\lambda$ against the alternative $\bigvee_{\lambda\in\Lambda_n^{(k,l)}}H_1^\lambda$ simultaneously for all $1\leq k\leq l\leq N$ using the Romano-Wolf method with 999 bootstrap resamples. 
In our analysis there are totally $N=11$ sectors: Consumer Discretionary, Consumer Staples, Financials, Health Care, Industrials, Information Technology, Materials, Real Estate, Telecommunication Services, and Utilities. Since Materials and Real Estate contain only one asset respectively, we exclude the case $k=l$ from the above hypotheses when $G_k$ is Materials or Real Estate. 
The results are reported in Table \ref{table:sectors}. \tcr{As expected}, the $p$-values for the absence of within-sector covariations are very small across all the sectors, which suggests within-sector covariations should exist for all the sectors. 
In contrast, we find that between-sector covariations can be insignificant for several pairs. For example, assets belonging to Materials (M) are not significantly correlated with assets belonging to the other sectors at the 5\% level. 
The table also reveals a similar between-sector covariation pattern to the one observed in \cite{FFX2016}. Namely, they report that the correlation between Energy (E) and Financials (F) disappears but Consumer Staples (CS) and Utilities (U) remain strongly correlated after 2010, which is consistent with the $p$-values reported in Table \ref{table:sectors}.  

Overall, our methodology partially provides a statistically formal support of the findings by \cite{FFX2016}, although the scope of our analysis is quite limited and thus more comprehensive empirical studies will be necessary.  

\begin{figure}[hp]
\centering
\includegraphics[scale=0.7]{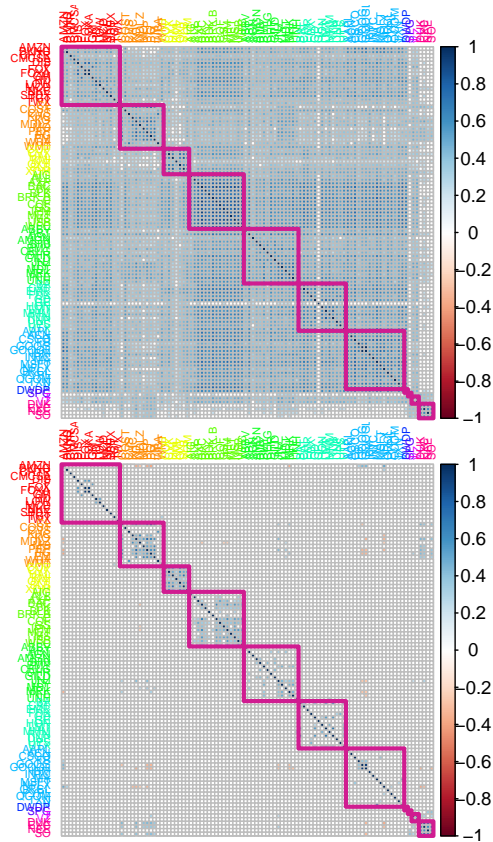}
\caption{
Realized correlation matrices of the S\&P 100 assets (top) and their residual processes regressed on SPY (bottom). They are computed from 15 minute returns in March 2018, where we ignore over-night returns. We perform multiple testing for whether each the entry is zero or not using the Romano-Wolf method with 999 bootstrap resamples, then the entries which are not significantly away from zero at the 5\% level are made blank. 
The violet squares indicate sector blocks. The figure was depicted using the \texttt{R} function \texttt{corrplot} from the \texttt{corrplot} package. 
}
\label{fig:sparse}
\end{figure}

\begin{table}[ht]
\caption{$p$-values of multiple testing for the absence of within- and between-secor covariations (the null hypotheses are the absence of covariations). 
}
\label{table:sectors}
\begin{center}
\begin{tabular}{c|ccccccccccc}
  \hline
 & CD & CS & E & F & HC & I & IT & M & RE & TS & U \\ 
  \hline
CD & 0.001 & 0.002 & 0.119 & 0.246 & 0.003 & 0.246 & 0.001 & 0.076 & 0.246 & 0.226 & 0.045 \\ 
  CS &  & 0.001 & 0.044 & 0.007 & 0.001 & 0.413 & 0.001 & 0.891 & 0.025 & 0.017 & 0.001 \\ 
  E &  &  & 0.001 & 0.502 & 0.446 & 0.211 & 0.076 & 0.932 & 0.098 & 0.662 & 0.076 \\ 
  F &  &  &  & 0.001 & 0.246 & 0.246 & 0.076 & 0.846 & 0.246 & 0.662 & 0.246 \\ 
  HC &  &  &  &  & 0.001 & 0.001 & 0.003 & 0.932 & 0.308 & 0.008 & 0.024 \\ 
  I &  &  &  &  &  & 0.001 & 0.246 & 0.502 & 0.846 & 0.662 & 0.224 \\ 
  IT &  &  &  &  &  &  & 0.001 & 0.446 & 0.246 & 0.072 & 0.004 \\ 
  M &  &  &  &  &  &  &  & -- & 0.932 & 0.909 & 0.932 \\ 
  RE &  &  &  &  &  &  &  &  & -- & 0.256 & 0.004 \\ 
  TS &  &  &  &  &  &  &  &  &  & 0.001 & 0.001 \\ 
  U &  &  &  &  &  &  &  &  &  &  & 0.001 \\ 
   \hline
\end{tabular}\vspace{5mm}

\parbox{15cm}{\small 
\textit{Note}. The $p$-values are computed using the Romano-Wolf method with 999 bootstrap resamples. 
The sector names are abbreviated as follows: CD: Consumer Discretionary; CS: Consumer Staples; F: Financials; HC: Health Care; I: Industrials; IT: Information Technology; M: Materials; RE: Real Estate; TS: Telecommunication Services; U: Utilities. 
}
\end{center}
\end{table}

\appendix

\addcontentsline{toc}{section}{Appendix}

\section{Proofs for Section \ref{sec:main}}

\subsection{Additional notation}

This subsection introduces some additional notation related to multi-way arrays and derivatives, which are necessary for the subsequent proofs. 

As in Section \ref{sec:array}, $\mathbb{K}$ denotes the real field $\mathbb{R}$ or the complex field $\mathbb{C}$. We consider a vector space $V$ over $\mathbb{K}$. 
Let $N_1,\dots,N_q$ be positive integers. For $T\in V^{N_1\times\cdots\times N_q}$ and $x\in \mathbb{K}^{N_1\times\cdots\times N_q}$, we set
\[
T[x]:=\sum_{(i_1,\dots,i_q)\in\prod_{k=1}^q[N_k]}T^{i_1,\dots,i_q}x^{i_1,\dots,i_q}\in V.
\]
In particular, for $x_j\in\mathbb{K}^{N_j}$ ($j=1,\dots,q$) we have
\[
T[x_1\otimes\cdots\otimes x_q]=\sum_{(i_1,\dots,i_q)\in\prod_{k=1}^q[N_k]}T^{i_1,\dots,i_q}x_1^{i_1}\cdots x_q^{i_q}.
\]
Here, note that we identify $\mathbb{K}^{N_1}\otimes\cdots\otimes\mathbb{K}^{N_q}$ with $\mathbb{K}^{N_1\times\cdots\times N_q}$ in the canonical way (see Section \ref{sec:array}). 
Moreover, we evidently have
\begin{equation}\label{tensor:holder}
|T[x]|\leq\|T\|_{\ell_\infty}\|x\|_{\ell_1}.
\end{equation}

Now suppose that $\mathbb{K}=\mathbb{R}$ and $V$ is a real Hilbert space. Then we have
\begin{equation}\label{tensor:inner}
\langle T[x],v\rangle_V
=\sum_{(i_1,\dots,i_q)\in\prod_{k=1}^q[N_k]}\langle T^{i_1,\dots,i_q},v\rangle_Vx^{i_1,\dots,i_q}
=\langle T,v\rangle_V[x]
\end{equation}
for any $v\in V$ \tcr{(recall \eqref{array:inner-prod})}. 
Let $V_0$ be another real Hilbert space and $N_1',\dots,N_p'\in\mathbb{N}$. Then, 
for any $S\in V_0^{N_1'\times\cdots\times N_p'}$ and $y\in\mathbb{R}^{N_1'\times\cdots\times N_p'}$, it holds that
\begin{equation}\label{tensor:commute}
T[x]\otimes S[y]=(T\otimes S)[x\otimes y].
\end{equation}
In fact, we have
\begin{align*}
T[x]\otimes S[y]
&=\sum_{(i_1,\dots,i_q)\in\prod_{k=1}^q[N_k]}\sum_{(j_1,\dots,j_p)\in\prod_{k=1}^p[N'_k]}(T^{i_1,\dots,i_q}x^{i_1,\dots,i_q})\otimes(\tcr{S^{j_1,\dots,j_p}}y^{j_1,\dots,j_p})\\
&=\sum_{(i_1,\dots,i_q)\in\prod_{k=1}^q[N_k]}\sum_{(j_1,\dots,j_p)\in\prod_{k=1}^p[N'_k]}(T^{i_1,\dots,i_q}\otimes S^{j_1,\dots,j_p})x^{i_1,\dots,i_q}y^{j_1,\dots,j_p}\\
&=(T\otimes S)[x\otimes y].
\end{align*}

Let $\phi=(\phi(y))_{y\in\mathbb{R}^N}$ be a real-valued function. If $\phi$ is a $C^\infty$ function, we define the $\mathbb{R}$-valued $N$-dimensional $q$-way array $\partial_y^{\otimes q}\phi(y)$ by 
\[
\partial_y^{\otimes q}\phi(y)=(\partial_{y^{i_1}\cdots y^{i_q}}\phi(y))_{1\leq i_1,\dots,i_q\leq N}\in\mathbb{R}^{N\times\cdots\times N}
\]
for any $y\in\mathbb{R}^N$ and $q\in\mathbb{N}$, where $\partial_{y^{i_1}\cdots y^{i_q}}:=\partial^q/\partial y^{i_1}\cdots \partial y^{i_q}$. We set $\partial_y^{\otimes 0}\phi(y):=\phi(y)$ by convention. 
In general, we say that $\phi$ is \textit{rapidly decreasing} if $\phi$ is a $C^\infty$ function and
\[
\sup_{y\in\mathbb{R}^N}(1+\|y\|_{\ell_2})^A\|\partial_y^{\otimes q}\phi(y)\|_{\ell_\infty}<\infty
\]
for any $A>0$ and $q\in\mathbb{Z}_+$. When $\phi$ is rapidly decreasing, we define its Fourier transform $\hat{\phi}:\mathbb{R}^N\to\mathbb{C}$ by
\[
\hat{\phi}(\mathsf{y})=\int_{\mathbb{R}^N}\phi(y)e^{-y[\mathsf{iy}]}dy,\qquad \mathsf{y}\in\mathbb{R}^N.
\]
Here, $\mathsf{i}$ denotes the imaginary unit. By Theorem 7.4(c) from \cite{Rudin1991}, one has
\begin{equation}\label{eq:rudin}
T[(\mathsf{iy})^{\otimes q}]\hat{\phi}(\mathsf{y})=T[\widehat{\partial_y^{\otimes q}\phi}(\mathsf{y})]
\end{equation}
for any $\mathsf{y}\in\mathbb{R}^N$, $q\in\mathbb{N}$ and $\mathbb{C}$-valued $N$-dimensional $q$-way array $T\in\mathbb{C}^{N\times\cdots\times N}$. 

If $i_1=\cdots=i_q=i$, we will write $\partial_{y^{i_1}\cdots y^{i_q}}$ as $\partial^q_{y^i}$. We set $\partial^0_{y^i}\varphi(y)=\varphi(y)$ by convention. 
For a multi-index $\alpha=(\alpha_1,\dots,\alpha_N)\in\mathbb{Z}_+^N$, we write $\partial^\alpha_y:=\partial^{\alpha_1}_{y^1}\cdots\partial^{\alpha_N}_{y^N}$ as usual. 
Given a subset $\mathcal{A}=\{a_1,\dots,a_k\}$ of $\{1,\dots,s\}$, we will write $\prod_{a\in\mathcal{A}}\partial_{y^{i_a}}:=\partial_{y^{i_{a_1}}\cdots y^{i_{a_k}}}$. We set $(\prod_{a\in\emptyset}\partial_{y^{i_a}})\phi(y):=\phi(y)$ by convention. 

\subsection{Proof of Theorem \ref{thm:main}}

We begin by noting that it is enough to prove the theorem for the special case that all the rows of the matrix $\bs{X}_n$ are identical:
\begin{lemma}\label{lemma:reduction}
Suppose that the claim of Theorem \ref{thm:main} holds true if $\bs{X}_{n}^{1\cdot}=\cdots=\bs{X}_{n}^{m\cdot}$ for every $n\in\mathbb{N}$. Then the claim of Theorem \ref{thm:main} holds true for the general case as well.
\end{lemma}

\begin{proof}
Define
the $m\times md$ matrix $\ol{\Upsilon}_n$ by
\[
\ol{\Upsilon}_n=
\left(
\begin{array}{ccccc}
(\Upsilon_n^{1\cdot})^\top  & 0  & \cdots & \cdots & 0  \\
0  & (\Upsilon_n^{2\cdot})^\top  & 0 & \cdots & 0  \\
\vdots  & \ddots  &  \ddots & \ddots & \vdots \\
0 & \ddots & \ddots & \ddots & 0 \\
0 & \cdots & \cdots & 0 & (\Upsilon_n^{m\cdot})^\top
\end{array}
\right).
\] 
We also define the $m\times md$ random matrix $\ol{\bs{X}}_n$ so that all the rows are identical to the $md$-dimensional random vector given by
\[
((\bs{X}_n^{1\cdot})^\top,\dots,(\bs{X}_{n}^{m\cdot})^\top).
\]
In addition, we define the $md$-dimensional random vector $\ol{Z}_n$ so that
\[
\ol{Z}_n^\top=(\underbrace{Z_n^\top,\dots,Z_n^\top}_{m})^\top.
\]
By assumption we can apply Theorem \ref{thm:main} with taking $\ol{\Upsilon}_n$, $\ol{\bs{X}}_n$ and $\ol{Z}_n$ as $\Upsilon_n$, $\bs{X}_n$ and $Z_n$ respectively, which yields the desired result.
\end{proof}
Taking account of Lemma \ref{lemma:reduction}, we focus only on the case that $\bs{X}_{n}^{1\cdot}=\cdots=\bs{X}_{n}^{m\cdot}=:X_n$ for every $n\in\mathbb{N}$. 

Next we recall the following anti-concentration inequality called Nazarov's inequality in \cite{CCK2017}: 
\begin{proposition}[Nazarov's inequality]
Let $\xi$ be an $m$-dimensional centered Gaussian vector such that $\|\xi^j\|_2\geq a$ for all $j=1,\dots,m$ and some constant $a>0$. Then for any $y\in\mathbb{R}^m$ and $\varepsilon>0$,
\[
P(\xi\leq y+\varepsilon)-P(\xi\leq y)\leq\frac{\varepsilon}{a}\left(\sqrt{2\log m}+2\right).
\]
\end{proposition}
The above form of Nazarov's inequality is found in \cite{CCK2017nazarov}. An application of the above result immediately yields the following anti-concentration inequality for a mixed-normal random vector:
\begin{lemma}\label{mixed-nazarov}
Let $\xi$ be an $m$-dimensional standard Gaussian vector. Also, let $\Gamma$ be an $m\times m$ symmetric positive-semidefinite random matrix independent of $\xi$. Then for any $y\in\mathbb{R}^m$ and $b,\varepsilon>0$,
\[
P(\Gamma^{1/2}\xi\leq y+\varepsilon)-P(\Gamma^{1/2}\xi\leq y)\leq\frac{\varepsilon}{\sqrt{b}}\left(\sqrt{2\log m}+2\right)+P\left(\min\diag(\Gamma)<b\right).
\]
\end{lemma}

Now we turn to the main body of the proof. 
As is pointed out in the Introduction, the key part of the proof is to derive reasonable estimates for the quantities
\begin{equation}\label{gap}
E[f(Z_n,X_n)]-E[f(\mathfrak{Z}_n,X_n)]
\end{equation}
for smooth functions $f:\mathbb{R}^{2d}\to\mathbb{R}$. In fact, the remaining part of the proof is essentially the same as the one for the high-dimensional central limit theorem of \cite{CCK2017}. 
To get a reasonable estimate for \eqref{gap}, we derive an interpolation formula for it, borrowing an idea from \cite{NY2017}. Namely, we use the duality between iterated Malliavin derivatives and multiple Skorohod integrals combined with the interpolation method in the frequency domain introduced in \cite{NY2017} (see also \cite{TY2017}). 

Following \cite{NY2017}, we set
\[
\lambda_n(\theta;\mathsf{z},\mathsf{x})=\theta M_n[\mathsf{i}\mathsf{z}]+2^{-1}(1-\theta^2) \mathfrak{C}_n[(\mathsf{i}\mathsf{z})^{\otimes2}]+W_n[\mathsf{i}\mathsf{z}]+ X_n[\mathsf{i}\mathsf{x}]
\]
and
\[
\varphi_n(\theta;\mathsf{z},\mathsf{x})=E[e^{\lambda_n(\theta;\mathsf{z},\mathsf{x})}]
\]
for $\theta\in[0,1]$ and $\mathsf{z},\mathsf{x}\in\mathbb{R}^d$. 
We first derive a representation formula for the derivative of $\varphi_n(\theta;\mathsf{z},\mathsf{x})$ with respect to $\theta$. For this purpose we need the following Malliavin derivative version of the (generalized) Fa\`a di Bruno formula for the iterated derivative of a composition of functions:
\begin{lemma}\label{faa}
Let $q,r$ be positive integers and $g=(g(x))_{x\in\mathbb{R}^r}$ be a real-valued $C^q$ function all of whose partial derivatives up to order $q$ are of polynomial growth. Then, for any $F\in\mathbb{D}_{q,\infty}(\mathbb{R}^r)$, we have $g(F)\in\mathbb{D}_{q,\infty}$ and
\begin{align*}
D^qg(F)=\sum_{\alpha\in\mathcal{A}(q)}\sum_{\nu\in\mathcal{N}_r(\alpha)}\mathsf{C}(\alpha,\nu)\partial_{x_1}^{|\nu_{\cdot1}|}\cdots \partial_{x_r}^{|\nu_{\cdot r}|}g(F)\symm\left(\bigotimes_{i=1}^{q}\bigotimes_{j=1}^{r}(D^iF^j)^{\otimes \nu_{ij}}\right),
\end{align*}
where
\[
\mathsf{C}(\alpha,\nu)=\frac{q!}{\prod_{i=1}^{q}(i!)^{\alpha_i}\prod_{j=1}^r\nu_{ij}!}.
\]
\end{lemma}
Noting that Malliavin derivatives can be characterized by directional derivatives along Cameron-Martin shifts (cf.~Chapter 15 of \cite{Janson1997}), we can derive Lemma \ref{faa} from the usual Fa\`a di Bruno formula (found in e.g.~\cite{Mishkov2000}). Alternatively, we can prove Lemma \ref{faa} in a parallel way to the usual Fa\`a di Bruno formula using the chain rule for Malliavin derivatives (see e.g.~Theorem 15.78 of \cite{Janson1997}) instead of that for standard ones.  

\begin{lemma}\label{lemma:cf-deriv}
Under the assumptions of Theorem \ref{thm:main}, the partial derivative $\partial_\theta\varphi_n(\theta;\mathsf{z},\mathsf{x})$ exists and it is given by
\begin{align*}
&\partial_\theta\varphi_n(\theta;\mathsf{z},\mathsf{x})\\
&=\theta \sum_{i,j=1}^dE\left[e^{\lambda(\theta;\mathsf{z},\mathsf{x})}\left(\left\langle D^{q_j}\tcr{M^i_n},u_n^j\right\rangle_{H^{\otimes q_j}}-\mathfrak{C}_n^{ij}\right)(\mathsf{i}\mathsf{z}^i)(\mathsf{i}\mathsf{z}^j)\right]\\
&+\sum_{j=1}^d\sum_{\alpha\in\mathcal{A}(q_j)}\sum_{\nu\in\mathcal{N}^*_4(\alpha)}\mathsf{C}(\alpha,\nu)\theta^{|\nu_{\cdot1}|}\{2^{-1}(1-\theta^2)\}^{|\nu_{\cdot2}|}(\mathsf{i}\mathsf{z}^j) 
E\left[e^{\lambda(\theta;\mathsf{z},\mathsf{x})}\Delta_{n,j}(\nu)[(\mathsf{i}\mathsf{z})^{\otimes (\nu_{i1}+2\nu_{i2}+\nu_{i3})}\otimes(\mathsf{i}\mathsf{x})^{\otimes \nu_{i4}}]\right].
\end{align*}
\end{lemma}

\begin{proof}
By assumption the function $\theta\mapsto\varphi_n(\theta;\mathsf{z},\mathsf{x})$ is evidently differentiable and we have
\begin{equation}\label{phi-deriv}
\partial_\theta\varphi_n(\theta;\mathsf{z},\mathsf{x})=E[e^{\lambda_n(\theta;\mathsf{z},\mathsf{x})}(M_n[\mathsf{i}\mathsf{z}]-\theta \mathfrak{C}_n[(\mathsf{i}\mathsf{z})^{\otimes2}])].
\end{equation}
By duality we obtain
\[
E[e^{\lambda_n(\theta;\mathsf{z},\mathsf{x})}M_n^j]
=E[\langle D^{q_j}\Re[e^{\lambda_n(\theta;\mathsf{z},\mathsf{x})}],u_n^j\rangle_{H^{\otimes q_j}}]
+\mathsf{i}E[\langle D^{q_j}\Im[e^{\lambda_n(\theta;\mathsf{z},\mathsf{x})}],u_n^j\rangle_{H^{\otimes q_j}}]
\]
for every $j$. Therefore, Lemma \ref{faa} yields
\begin{multline*}
E[e^{\lambda_n(\theta;\mathsf{z},\mathsf{x})}M_n^j]
=\sum_{\alpha\in\mathcal{A}(q_j)}\sum_{\nu\in\mathcal{N}_4(\alpha)}\mathsf{C}(\alpha,\nu)\theta^{|\nu_{\cdot1}|}\{2^{-1}(1-\theta^2)\}^{|\nu_{\cdot2}|}\cdot\mathsf{i}^{|\nu_{\cdot1}|+2|\nu_{\cdot2}|+|\nu_{\cdot3}|+|\nu_{\cdot4}|}\\
\times E\left[e^{\lambda(\theta;\mathsf{z},\mathsf{x})}\left\langle \bigotimes_{i=1}^{q_j}(D^iM_n[\mathsf{z}])^{\otimes \nu_{i1}}\otimes(D^i\mathfrak{C}_n[\mathsf{z}^{\otimes2}])^{\otimes \nu_{i2}}\otimes(D^iW_n[\mathsf{z}])^{\otimes \nu_{i3}}\otimes(D^iX_n[\mathsf{x}])^{\otimes \nu_{i4}},u_n^j\right\rangle_{H^{\otimes q_j}}\right],
\end{multline*} 
where we also use the identity
\begin{equation}\label{eq:symm}
\langle \symm(f),g\rangle_{H^{\otimes q}}=\langle f,g\rangle_{H^{\otimes q}}
\end{equation}
holding for any $f,g\in H^{\otimes q}$ such that $g$ is symmetric.  
Now, by \eqref{tensor:commute} we have
\begin{align*}
&\bigotimes_{i=1}^{q_j}(D^iM_n[\mathsf{z}])^{\otimes \nu_{i1}}\otimes(D^i\mathfrak{C}_n[\mathsf{z}^{\otimes2}])^{\otimes \nu_{i2}}\otimes(D^iW_n[\mathsf{z}])^{\otimes \nu_{i3}}\otimes(D^iX_n[\mathsf{x}])^{\otimes \nu_{i4}}\\
&=\bigotimes_{i=1}^{q_j}(D^iM_n)^{\otimes \nu_{i1}}[\mathsf{z}^{\otimes \nu_{i1}}]\otimes(D^i\mathfrak{C}_n)^{\otimes \nu_{i2}}[\mathsf{z}^{\otimes2\nu_{i2}}]\otimes(D^iW_n)^{\otimes \nu_{i3}}[\mathsf{z}^{\otimes \nu_{i3}}]\otimes(D^iX_n)^{\otimes \nu_{i4}}[\mathsf{x}^{\otimes \nu_{i4}}]\\
&=\left(\bigotimes_{i=1}^{q_j}(D^iM_n)^{\otimes \nu_{i1}}\otimes(D^i\mathfrak{C}_n)^{\otimes \nu_{i2}}\otimes(D^iW_n)^{\otimes \nu_{i3}}\otimes(D^iX_n)^{\otimes \nu_{i4}}\right)[\mathsf{z}^{\otimes (\nu_{i1}+2\nu_{i2}+\nu_{i3})}\otimes\mathsf{x}^{\otimes \nu_{i4}}],
\end{align*}
\tcr{so} using \eqref{tensor:inner} we obtain
\begin{align*}
&E[e^{\lambda_n(\theta;\mathsf{z},\mathsf{x})}M_n^j]\\
&=\sum_{\alpha\in\mathcal{A}(q_j)}\sum_{\nu\in\mathcal{N}_4(\alpha)}\mathsf{C}(\alpha,\nu)\theta^{|\nu_{\cdot1}|}\{2^{-1}(1-\theta^2)\}^{|\nu_{\cdot2}|}\cdot\mathsf{i}^{|\nu_{\cdot1}|+2|\nu_{\cdot2}|+|\nu_{\cdot3}|+|\nu_{\cdot4}|}
E\left[e^{\lambda(\theta;\mathsf{z},\mathsf{x})}\Delta_{n,j}(\nu)[\mathsf{z}^{\otimes (\nu_{i1}+2\nu_{i2}+\nu_{i3})}\otimes\mathsf{x}^{\otimes \nu_{i4}}]\right]\\
&=\theta E\left[e^{\lambda(\theta;\mathsf{z},\mathsf{x})}\left\langle D^{q_j}M_n,u_n^j\right\rangle_{H^{\otimes q_j}}[\mathsf{i}\mathsf{z}]\right]\\
&+\sum_{\alpha\in\mathcal{A}(q_j)}\sum_{\nu\in\mathcal{N}^*_4(\alpha)}\mathsf{C}(\alpha,\nu)\theta^{|\nu_{\cdot1}|}\{2^{-1}(1-\theta^2)\}^{|\nu_{\cdot2}|}
 E\left[e^{\lambda(\theta;\mathsf{z},\mathsf{x})}\Delta_{n,j}(\nu)[(\mathsf{i}\mathsf{z})^{\otimes (\nu_{i1}+2\nu_{i2}+\nu_{i3})}\otimes(\mathsf{i}\mathsf{x})^{\otimes \nu_{i4}}]\right].
\end{align*}
Combining this identity with \eqref{phi-deriv}, we obtain the desired result. 
\end{proof}

The following lemma is presumably a standard result. We prove it for the shake of completeness. 
\begin{lemma}\label{regularization}
Let $f=(f(y))_{y\in\mathbb{R}^N}$ be a real-valued $C^\infty$ function all of whose partial derivatives are of polynomial growth. Then there is a sequence $(f_j)_{j=1}^\infty$ of compactly supported real-valued $C^\infty$ functions on $\mathbb{R}^N$ such that 
\begin{equation}\label{eq:vitali}
E\left[\xi_0 \partial_y^\alpha f_j(\xi_1,\dots,\xi_N)\right]\to E\left[\xi_0\partial_y^\alpha f(\xi_1,\dots,\xi_N)\right]
\end{equation}
as $j\to\infty$ for any $\xi_0,\xi_1,\dots,\xi_N\in L^{\infty-}$ and $\alpha\in\mathbb{Z}_+^N$. 
\end{lemma}

\begin{proof}
Take a $C^\infty$ function $\phi:\mathbb{R}^N\to[0,\infty)$ having compact support and satisfying $\phi(0)=1$. 
For every $j=1,2,\dots$, we define the function $\phi_j:\mathbb{R}^N\to[0,\infty)$ by $\phi_j(y)=\phi(j^{-1}y)$, $y\in\mathbb{R}^N$. 
Then we define $f_j:=f\phi_j$ for $j=1,2,\dots$. $f_j$ is evidently a $C^\infty$ function with compact support. Moreover, we have $\partial_{y}^\alpha f_j(y)\to\partial_{y}^\alpha f(y)$ as $j\to\infty$ for any $y\in\mathbb{R}^N$ and $\alpha\in\mathbb{Z}_+^N$. In addition, for any $s\in\mathbb{N}$, there is a constant $C>0$ which depends only on $\phi$ and $s$ such that $|f_j(y)|\leq C(|f(y)|+\sum_{k=1}^s\|\partial_y^{\otimes k}f(y)\|_{\ell_1})$ for any $y\in\mathbb{R}^N$; we can easily prove these facts by directly differentiating $f_j$ with the help of the Leibniz formula and the chain rule. Consequently, we have $\sup_{j\in\mathbb{N}}\|\xi_0 \partial_y^\alpha f_j(\xi_1,\dots,\xi_N)\|_2<\infty$ for any $\alpha\in\mathbb{Z}_+^N$ because $\xi_0,\xi_1,\dots,\xi_N\in L^{\infty-}$ and all the partial derivatives of $f$ have polynomial growth. Therefore, $(\xi_0 \partial_y^\alpha f_j(\xi_1,\dots,\xi_N))_{j\in\mathbb{N}}$ is uniformly integrable, \tcr{so} the Vitali convergence theorem yields \eqref{eq:vitali}. This completes the proof. 
\end{proof}

Now we get the following interpolation formula for \eqref{gap}:
\begin{lemma}\label{interpolation}
Let $f:\mathbb{R}^{2d}\to\mathbb{R}$ be a $C^\infty$ function all of whose partial derivatives are of polynomial growth. Under the assumptions of Theorem \ref{thm:main}, we have
\begin{align*}
&E[f(Z_n,X_n)]-E[f(\mathfrak{Z}_n,X_n)]\\
&=\sum_{i,j=1}^{d}\int_0^1\theta E\left[\left(\langle D^{q_j}M^i_n,u_n^j\rangle_{H^{\otimes q_j}}-\mathfrak{C}_n^{ij}\right)\partial_{z_i}\partial_{z_j}f(\theta Z_n+\sqrt{1-\theta^2}\mathfrak{Z}_n,X_n)\right]d\theta\\
&+\sum_{j=1}^d\sum_{\alpha\in\mathcal{A}(q_j)}\sum_{\nu\in\mathcal{N}_{4}^*(\alpha)}\mathsf{C}(\alpha,\nu)\int_0^1\theta^{|\nu_{\cdot 1}|}(2^{-1}(1-\theta^2))^{|\nu_{\cdot 2}|}
 E\left[\Delta_{n,j}(\nu)[\partial_z^{\otimes|\nu|_*}\partial_x^{\otimes|\nu_{\cdot 4}|}\partial_{z_j}f(\theta Z_n+\sqrt{1-\theta^2}\mathfrak{Z}_n,X_n)]\right]d\theta.
\end{align*}
\end{lemma}

\begin{proof}
Thanks to Lemma \ref{regularization}, it is enough to prove the lemma when $f$ is rapidly decreasing.  
In this case the Fourier inversion formula and the Fubini theorem yield
\begin{align*}
E[f(Z_n,X_n)]-E[f(\mathfrak{Z}_n,X_n)]
&=(2\pi)^{-2d}\int_{\mathbb{R}^{2d}}\hat{f}(\mathsf{z},\mathsf{x})\{\varphi(1;\mathsf{z},\mathsf{x})-\varphi(0;\mathsf{z},\mathsf{x})\}d\mathsf{z}d\mathsf{x}\\
&=(2\pi)^{-2d}\int_0^1d\theta\int_{\mathbb{R}^{2d}}\hat{f}(\mathsf{z},\mathsf{x})\partial_\theta\varphi(\theta;\mathsf{z},\mathsf{x})d\mathsf{z}d\mathsf{x}.
\end{align*}
Hence the desired result follows from Lemma \ref{lemma:cf-deriv}, \eqref{eq:rudin} and the Fourier inversion formula. 
\end{proof}

We will use the following elementary result in the proof:
\begin{lemma}\label{z-leibniz}
Let $k,l$ be two positive integers. Then we have
\[
\partial_{z^{i_1}\cdots z^{i_k}}\left(z^{j_1}\cdots z^{j_l}\right)
=\sum_{\begin{subarray}{c}
c_1,\dots,c_k=1\\
c_s\neq c_t
\end{subarray}}^l\prod_{s=1}^k1_{\{j_{c_s}=i_s\}}\prod_{b\neq c_1,\dots,c_k}z^{j_b}
\]
for any $i_1,\dots,i_k,j_1,\dots,j_l\in\{1,\dots,d\}$. 
\end{lemma}
One can easily prove the above lemma by induction on $k$ and application of the Leibniz rule, \tcr{so} we omit its proof.

{
Finally, as in the original CCK theory, a special approximation of the maximum function (called the ``smooth max function'') will play a crucial role in our proof. 
The following lemma summarizes the key properties of this smooth max function used in the proof:
\begin{lemma}\label{lemma:dz}
Let $\varepsilon>0$ and set $\beta=\varepsilon^{-1}\log m$. Define the function $\Phi_\beta:\mathbb{R}^{m}\to\mathbb{R}$ by
\begin{equation}\label{def:Phi}
\Phi_\beta(w)=\beta^{-1}\log\left(\sum_{j=1}^{m}\exp(\beta w^j)\right),\qquad w\in\mathbb{R}^{m}. 
\end{equation}
Then we have
\begin{equation}\label{max-smooth}
0\leq \Phi_\beta(w)-\max_{1\leq j\leq m}w^j\leq \beta^{-1}\log m=\varepsilon
\end{equation}
for every $w\in\mathbb{R}^{m}$. 
Moreover, for any $C^\infty$ function $g:\mathbb{R}\to\mathbb{R}$, $s\in\mathbb{N}$, $\varepsilon>0$ and $w\in\mathbb{R}^m$, it holds that
\begin{equation}\label{eq:dz}
\left\|\partial_w^{\otimes s}g(\varepsilon^{-1}\Phi_\beta(w))\right\|_{\ell_1}\leq C_{g,s}\max\{\varepsilon^{-s},\varepsilon^{-1}\beta^{s-1}\}
=C_{g,s}\varepsilon^{-s}(\log m)^{s-1},
\end{equation}
where $C_{g,s}>0$ depends only on $g$ and $s$. 
\end{lemma}

\begin{proof}
First, note that $\Phi_\beta$ is usually denoted by $F_\beta$ in the literature on the CCK theory. 
Now, \eqref{max-smooth} is stated in e.g.~Eq.(1) of \cite{CCK2015}. 
On the other hand, \eqref{eq:dz} is obtained by applying Lemma 5 in \cite{DZ2017} with $h=g$, $n=1$, $m=s$ and $b=\varepsilon^{-1}$ in their notation. 
\end{proof}
}

\begin{proof}[Proof of Theorem \ref{thm:main}]
First, as is already noted in the above, for the proof it is enough to focus only on the case that $\bs{X}_{n}^{1\cdot}=\cdots=\bs{X}_{n}^{m\cdot}=:X_n$ for every $n\in\mathbb{N}$ due to Lemma \ref{lemma:reduction}. 
Note that in this case we have $\Xi_nz=\Upsilon_n(z\circ X_n)$ for every $z\in\mathbb{R}^d$. 

We turn to the main body of the proof. 
Take a number $\varepsilon>0$ arbitrarily, and set $\beta=\varepsilon^{-1}\log m$. 
\tcr{We define the function $\Phi_\beta:\mathbb{R}^{m}\to\mathbb{R}$ by \eqref{def:Phi}}. 
We also take a $C^\infty$ function $g:\mathbb{R}\to[0,1]$ such that all the derivatives of $g$ \tcr{are} bounded and $g(t)=1$ for $t\leq0$ and $g(t)=0$ for $t\geq1$. 

Now let us fix a vector $y\in\mathbb{R}^{m}$ arbitrarily, and define the functions $\varphi:\mathbb{R}^{m}\to\mathbb{R}$, $\psi:\mathbb{R}^{d}\to\mathbb{R}$ and $f:\mathbb{R}^{2d}\to\mathbb{R}$ by
\begin{align*}
\varphi(w)&=g(\varepsilon^{-1}\Phi_\beta(w-y-\varepsilon)),\qquad w\in\mathbb{R}^{m},\\
\psi(v)&=\varphi(\Upsilon_nv),\qquad v\in\mathbb{R}^d,\\
f(z,x)&=\psi(z\circ x),\qquad z,x\in\mathbb{R}^{d}.
\end{align*}
For any $k,l\in\mathbb{Z}_+$ and any $z,x\in\mathbb{R}^d$, we have 
\begin{align*}
\left\|\partial_z^{\otimes k}\partial_x^{\otimes l}f(z,x)\right\|_{\ell_1}
&=\sum_{i_1,\dots,i_k,j_1,\dots,j_l=1}^d\left|\partial_{z^{i_1}\cdots z^{i_k}}\left(z^{j_1}\cdots z^{j_l}\partial_{v^{j_1}\cdots v^{j_l}}\psi(z\circ x)\right)\right|.
\end{align*}
Applying the Leibniz rule repeatedly (cf.~Proposition 5 of \cite{Hardy2006}), we deduce 
\begin{align*}
\left\|\partial_z^{\otimes k}\partial_x^{\otimes l}f(z,x)\right\|_{\ell_1}
&=\sum_{i_1,\dots,i_k,j_1,\dots,j_l=1}^d\left|\sum_{\mathcal{A}\subset\{1,\dots,k\}}\left(\prod_{a\in\mathcal{A}}\partial_{z^{i_a}}\right)\left(z^{j_1}\cdots z^{j_l}\right)\left(\prod_{a\notin\mathcal{A}}x^{i_a}\right)\left(\prod_{a\notin\mathcal{A}}\partial_{v^{i_a}}\right)\partial_{v^{j_1}\cdots v^{j_l}}\psi(z\circ x)\right|\\
&\leq\sum_{j_1,\dots,j_l=1}^d\sum_{\mathcal{A}\subset\{1,\dots,k\}}\sum_{i_1,\dots,i_k=1}^d\left|\left(\prod_{a\in\mathcal{A}}\partial_{z^{i_a}}\right)\left(z^{j_1}\cdots z^{j_l}\right)\left(\prod_{a\notin\mathcal{A}}x^{i_a}\right)\left(\prod_{a\notin\mathcal{A}}\partial_{v^{i_a}}\right)\partial_{v^{j_1}\cdots v^{j_l}}\psi(z\circ x)\right|.
\end{align*}
Now let us fix a subset $\mathcal{A}$ of $\{1,\dots,k\}$. Let $r$ be the number of elements of $\mathcal{A}$ and we write $\mathcal{A}=\{a_1,\dots,a_r\}$ and $\{1,\dots,k\}\setminus\mathcal{A}=\{b_1,\dots,b_{k-r}\}$. Assume $1\leq r\leq l$. Then, by Lemma \ref{z-leibniz} we obtain
\begin{align*}
&\sum_{i_1,\dots,i_k=1}^d\left|\left(\prod_{a\in\mathcal{A}}\partial_{z^{i_a}}\right)\left(z^{j_1}\cdots z^{j_l}\right)\left(\prod_{a\notin\mathcal{A}}x^{i_a}\right)\left(\prod_{a\notin\mathcal{A}}\partial_{v^{i_a}}\right)\partial_{v^{j_1}\cdots v^{j_l}}\psi(z\circ x)\right|\\
&=\sum_{i_1,\dots,i_k=1}^d\left|\left(\sum_{\begin{subarray}{c}
c_1,\dots,c_r=1\\
c_s\neq c_t
\end{subarray}}^l\prod_{s=1}^r1_{\{j_{c_s}=i_{a_s}\}}\prod_{b\neq c_1,\dots,c_r}z^{j_b}\right)\left(\prod_{t=1}^{k-r}x^{i_{b_t}}\right)\partial_{v^{i_{b_1}}\cdots v^{i_{b_{k-r}}}v^{j_1}\cdots v^{j_l}}\psi(z\circ x)\right|\\
&=\sum_{i_{a_1},\dots,i_{a_r}=1}^d\left|\sum_{\begin{subarray}{c}
c_1,\dots,c_r=1\\
c_s\neq c_t
\end{subarray}}^l\prod_{s=1}^r1_{\{j_{c_s}=i_{a_s}\}}\prod_{b\neq c_1,\dots,c_r}z^{j_b}\right|
\sum_{i_{b_1},\dots,i_{b_{k-r}}=1}^d\left|\left(\prod_{t=1}^{k-r}x^{i_{b_t}}\right)\partial_{v^{i_{b_1}}\cdots v^{i_{b_{k-r}}}v^{j_1}\cdots v^{j_l}}\psi(z\circ x)\right|\\
&\leq\tcr{\frac{l!}{(l-r)!}}\|z\|_{\ell_\infty}^{l-r}\|x\|_{\ell_\infty}^{k-r}\left\|\partial_v^{\otimes(k-r)}\partial_{v^{j_1}\cdots v^{j_l}}\psi(z\circ x)\right\|_{\ell_1}.
\end{align*}
Note that the above inequality evidently holds true if $\mathcal{A}=\emptyset$. Moreover, we obviously have 
$
\left(\prod_{a\in\mathcal{A}}\partial_{z^{i_a}}\right)\left(z^{j_1}\cdots z^{j_l}\right)=0
$ 
if $r>l$. Consequently, we infer that
\begin{align*}
\left\|\partial_z^{\otimes k}\partial_x^{\otimes l}f(z,x)\right\|_{\ell_1}
&\leq\sum_{r=0}^{k\wedge l}\tcr{r!}\binom{k}{r}\binom{l}{r}\|z\|_{\ell_\infty}^{l-r}\|x\|_{\ell_\infty}^{k-r}\left\|\partial_v^{\otimes(k+l-r)}\psi(z\circ x)\right\|_{\ell_1}.
\end{align*}
Meanwhile, we can easily verify that
\[
\partial_{v^{i_1}\dots v^{i_s}}\psi(v)=\sum_{j_1,\dots,j_s=1}^m\partial_{w^{j_1}\dots w^{j_s}}\varphi(\Upsilon_nv)\Upsilon_n^{j_1i_1}\cdots\Upsilon_n^{j_si_s}
\]
for any $s\in\mathbb{N}$ and $i_1,\dots,i_s\in\{1,\dots,d\}$. Hence we have
\begin{align*}
\|\partial_v^{\otimes s}\psi(v)\|_{\ell_1}
\leq\left(\max_{1\leq j\leq m}\sum_{i=1}^d|\Upsilon_n^{ji}|\right)^s\sum_{j_1,\dots,j_s=1}^m|\partial_{w^{j_1}\dots w^{j_s}}\varphi(\Upsilon_nv)|
=\opnorm{\Upsilon_n}_\infty^s\|\partial_w^{\otimes s}\varphi(\Upsilon_nv)\|_{\ell_1}.
\end{align*}
Now, by \tcr{\eqref{eq:dz}} it holds that
\[
\|\partial_w^{\otimes s}\varphi(w)\|_{\ell_1}\leq C_{g,s}\varepsilon^{-s}(\log m)^{s-1}
\]
for all $w\in\mathbb{R}^{m}$, where $C_{g,s}>0$ is a constant which depends only on $g$ and $s$. 
Therefore, we obtain
\begin{align*}
\left\|\partial_z^{\otimes k}\partial_x^{\otimes l}f(z,x)\right\|_{\ell_1}
&\leq c_{g,k,l}\sum_{r=0}^{k\wedge l}\tcr{r!}\binom{k}{r}\binom{l}{r}\|z\|_{\ell_\infty}^{l-r}\|x\|_{\ell_\infty}^{k-r}\opnorm{\Upsilon_n}_\infty^{k+l-r}\varepsilon^{-(k+l-r)}(\log m)^{k+l-r-1}\\
&\leq c'_{g,k,l}\left(\|z\|_{\ell_\infty}\vee1\right)^{l}\left(\|x\|_{\ell_\infty}\vee1\right)^k\tcr{\opnorm{\Upsilon_n}_\infty^{k+l}}\tcr{\varepsilon_1^{-(k+l)}}(\log m)^{k+l-1},
\end{align*}
where \tcr{$\varepsilon_1:=\varepsilon\wedge1$ and} $c_{g,k,l},c_{g,k,l}'>0$ are constants which depend only on $g$ and $k,l$ \tcr{(recall $\opnorm{\Upsilon_n}_\infty\geq1$ by assumption)}. 
We especially infer that all the partial derivatives of $f$ are of polynomial growth. Therefore, \tcr{noting that $\Delta_{n,j}(\nu)=0$ when $\nu\notin\bigcup_{\alpha\in\mathcal{A}(q_j)}\mathcal{N}_4^*(\alpha)$,} by \eqref{tensor:holder} and Lemma \ref{interpolation} we obtain
\begin{align*}
\eta_n(\varepsilon)&:=|E[f(Z_n,X_n)]-E[f(\mathfrak{Z}_n,X_n)]|\\
&\tcr{\leq\int_0^1E\left[\|\Delta_n\|_{\ell_\infty}\sum_{i,j=1}^{d}\left|\partial_{z_i}\partial_{z_j}f(\theta Z_n+\sqrt{1-\theta^2}\mathfrak{Z}_n,X_n)\right|\right]d\theta}\\
&\quad\tcr{+K_{\ol{q}}\sum_{j=1}^d\sum_{\alpha\in\mathcal{A}(q_j)}\sum_{\nu\in\mathcal{N}_{4}^*(\alpha)}\int_0^1E\left[\|\Delta_{n,j}(\nu)\|_{\ell_\infty}\|\partial_z^{\otimes|\nu|_*}\partial_x^{\otimes|\nu_{\cdot 4}|}\partial_{z_j}f(\theta Z_n+\sqrt{1-\theta^2}\mathfrak{Z}_n,X_n)\|_{\ell_1}\right]d\theta}\\
&\leq\int_0^1E\left[\|X_n\|_{\ell_\infty}^2\opnorm{\Upsilon_n}_\infty^2\|\Delta_n\|_{\ell_\infty}\sum_{i,j=1}^{\tcr{m}}\left|\partial_{w_i}\partial_{w_j}\varphi(X_n\circ(\theta Z_n+\sqrt{1-\theta^2}\mathfrak{Z}_n))\right|\right]d\theta\\
&\tcr{\quad+K_{\ol{q}}\sum_{\alpha\in\ol{\mathcal{A}}(\ol{q})}\sum_{\nu\in\mathcal{N}_{4}^*(\alpha)}\int_0^1E\left[\max_{1\leq j\leq d}\|\Delta_{n,j}(\nu)\|_{\ell_\infty}\|\partial_z^{\otimes(|\nu|_*+1)}\partial_x^{\otimes|\nu_{\cdot 4}|}f(\theta Z_n+\sqrt{1-\theta^2}\mathfrak{Z}_n,X_n)\|_{\ell_1}\right]d\theta}\\
&\leq c''_{g,\overline{q}}\tcr{\varepsilon_1^{-2}}(\log m)\opnorm{\Upsilon_n}_\infty^2E\left[\|X_n\|_{\ell_\infty}^2\|\Delta_n\|_{\ell_\infty}\right]\\
&\quad+c''_{g,\overline{q}}\sum_{\alpha\in\tcr{\ol{\mathcal{A}}}(\overline{q})}\sum_{\nu\in\mathcal{N}_{4}^*(\alpha)}\tcr{\varepsilon_1^{-|\nu|_{**}-1}}(\log m)^{|\nu|_{**}}\opnorm{\Upsilon_n}_\infty^{|\nu|_{**}+1}
E\left[\left(1+\|X_n\|_{\ell_\infty}^{|\nu|_*+1}\right)\left(1+\|Z_n\|_{\ell_\infty}^{|\nu_{\cdot 4}|}+\|\mathfrak{Z}_n\|_{\ell_\infty}^{|\nu_{\cdot 4}|}\right)\max_{1\leq j\leq d}\|\Delta_{n,j}(\nu)\|_{\ell_\infty}\right],
\end{align*}
where \tcr{$K_{\ol{q}}>0$ depends only on $\ol{q}$ and $c''_{g,\ol{q}}>0$ depends only on $g$ and $\overline{q}$}. 
Now we have
\begin{align*}
P(\Xi_nZ_n\leq y)&\leq P(\Phi_\beta(\Upsilon_n(Z_n\circ X_n)-y-\varepsilon)\leq0)~(\because\text{Eq.\eqref{max-smooth}})\\
&\leq E[f(Z_n,X_n)]
\leq E[f(\mathfrak{Z}_n,X_n)]+\eta_n(\varepsilon)\\
&\leq P(\Phi_\beta(\Upsilon_n(\mathfrak{Z}_n\circ X_n)-y-\varepsilon)<\varepsilon)+\eta_n(\varepsilon)~(\because\text{the definition of $g$})\\
&\leq P(\Xi_n\mathfrak{Z}_n\leq y+2\varepsilon)+\eta_n(\varepsilon)~(\because\text{Eq.\eqref{max-smooth}}).
\end{align*}
Set $\Gamma_n:=\Xi_n\mathfrak{C}_n\Xi_n^\top$. Then, by Lemma \ref{mixed-nazarov} we obtain
\begin{align*}
P(\Xi_nZ_n\leq y)
&\leq P(\min\diag(\Gamma_n)<b)
+P(\Xi_n\mathfrak{Z}_n\leq y)+\frac{2\varepsilon}{\sqrt{b}}(\sqrt{2\log m}+2)+\eta_n(\varepsilon)
\end{align*}
for every $b>0$. 
By an analogous argument we also obtain
\begin{align*}
P(\Xi_nZ_n\leq y)
&\geq P(\min\diag(\Gamma_n)<b)
-P(\Xi_n\mathfrak{Z}_n\leq y)-\frac{2\varepsilon}{\sqrt{b}}(\sqrt{2\log m}+2)-\eta_n(\varepsilon).
\end{align*}
Therefore, we conclude that
\begin{align*}
\sup_{y\in\mathbb{R}^m}\left|P(\Xi_nZ_n\leq y)-P(\Xi_n\mathfrak{Z}_n\leq y)\right|
&\leq P(\min\diag(\Gamma_n)<b)
+\frac{2\varepsilon}{\sqrt{b}}(\sqrt{2\log m}+2)+\eta_n(\varepsilon).
\end{align*}
Taking
\begin{multline*}
\varepsilon=\left(\sqrt{\log m}\opnorm{\Upsilon_n}_\infty^2E\left[\|X_n\|_{\ell_\infty}^2\|\Delta_n\|_{\ell_\infty}\right]\right)^{1/3}\vee\\
\max_{\nu\in\ol{\mathcal{N}}_4^*(\ol{q})}\left\{(\log m)^{|\nu|_{**}-\frac{1}{2}}\opnorm{\Upsilon_n}_\infty^{|\nu|_{**}+1}E\left[(1+\|X_n\|_{\ell_\infty}^{|\nu|_*+1})\left(1+\|Z_n\|_{\ell_\infty}^{|\nu_{\cdot 4}|}+\|\mathfrak{Z}_n\|_{\ell_\infty}^{|\nu_{\cdot 4}|}\right)\max_{1\leq j\leq d}\|\Delta_{n,j}(\nu)\|_{\ell_\infty}\right]\right\}^{\frac{1}{|\nu|_{**}+2}},
\end{multline*}
we obtain
\begin{align*}
\limsup_{n\to\infty}\sup_{y\in\mathbb{R}^m}\left|P(\Xi_nZ_n\leq y)-P(\Xi_n\mathfrak{Z}_n\leq y)\right|
&\leq \limsup_{n\to\infty}P(\min\diag(\Gamma_n)<b)
\end{align*}
by assumption. Letting $b\to0$, we complete the proof. 
\end{proof}

\subsection{Proof of Lemma \ref{lemma:approx}}

Take a number $\varepsilon>0$ arbitrarily. For any $y\in\mathbb{R}^m$, we have
\begin{align*}
P(Y_n\leq y)
&\leq P(\sqrt{\log m}\|Y_n-\Xi_nZ_n\|_{\ell_\infty}>\varepsilon)
+P(\Xi_nZ_n\leq y+\varepsilon/\sqrt{\log m})\\
&\leq P(\sqrt{\log m}\|Y_n-\Xi_nZ_n\|_{\ell_\infty}>\varepsilon)
+P(\Xi_n\mathfrak{Z}_n\leq y+\varepsilon/\sqrt{\log m})+\rho_n,
\end{align*}
where
\[
\rho_n:=\sup_{y\in\mathbb{R}^m}|P(\Xi_nZ_n\leq y)-P(\Xi_n\mathfrak{Z}_n\leq y)|.
\]
Therefore, Lemma \ref{mixed-nazarov} yields
\begin{multline*}
P(Y_n\leq y)
\leq P(\sqrt{\log m}\|Y_n-\Xi_nZ_n\|_{\ell_\infty}>\varepsilon)
+P(\Xi_n\mathfrak{Z}_n\leq y)\\
+\frac{2\varepsilon}{\sqrt{b\log m}}(\sqrt{2\log m}+2)+P(\min\diag(\Gamma_n)<b)
+\rho_n
\end{multline*}
for any $b>0$. An analogous argument yields
\begin{multline*}
P(Y_n\leq y)
\geq -P(\sqrt{\log m}\|Y_n-\Xi_nZ_n\|_{\ell_\infty}>\varepsilon)
+P(\Xi_n\mathfrak{Z}_n\leq y)\\
-\frac{2\varepsilon}{\sqrt{b\log m}}(\sqrt{2\log m}+2)+P(\min\diag(\Gamma_n)<b)
-\rho_n,
\end{multline*}
\tcr{so} we conclude that
\begin{multline*}
\sup_{y\in\mathbb{R}^m}|P(Y_n\leq y)-P(\Xi_n\mathfrak{Z}_n\leq y)|
\leq P(\sqrt{\log m}\|Y_n-\Xi_nZ_n\|_{\ell_\infty}>\varepsilon)\\
+\frac{2\varepsilon}{\sqrt{b\log m}}(\sqrt{2\log m}+2)+P(\min\diag(\Gamma_n)<b)
+\rho_n.
\end{multline*}
Now, by assumption we obtain
\[
\limsup_{n\to\infty}\sup_{y\in\mathbb{R}^m}|P(Y_n\leq y)-P(\Xi_n\mathfrak{Z}_n\leq y)|
\leq \frac{2\varepsilon}{\sqrt{b}}\left(\sqrt{2}+\frac{2}{\sqrt{\log 2}}\right)+\limsup_{n\to\infty}P(\min\diag(\Gamma_n)<b). 
\]
We first let $\varepsilon\to0$. After that, we let $b\to0$. Then we conclude that
\[
\limsup_{n\to\infty}\sup_{y\in\mathbb{R}^m}|P(Y_n\leq y)-P(\Xi_n\mathfrak{Z}_n\leq y)|
=0. 
\]
This completes the proof.\hfill\qed

\subsection{Proof of Proposition \ref{prop:comparison}}

The proof is analogous to that of Theorem 2 from \cite{CCK2015} (see also the proof of Theorem 4.1 from \cite{CCK2017}). 
Setting 
\begin{align*}
\Gamma_n:=\Xi_n\mathfrak{C}_n\Xi_n^\top,\qquad 
\widehat{\Gamma}_n:=\wh{\Xi}_n\widehat{\mathfrak{C}}_n\wh{\Xi}_n^\top,\qquad
\mu_n:=\Xi_nW_n,\qquad
\wh{\mu}_n:=\wh{\Xi}_n\wh{W}_n,
\end{align*}
we have
\[
P(\Xi_n\mathfrak{Z}_n\leq y|\mathcal{F})=P(\Gamma_n^{1/2}\xi_n+\mu_n\leq y|\mathcal{F}),\qquad
P(\wh{\Xi}_n\widehat{\mathfrak{Z}}_n\leq y|\mathcal{F})=P(\widehat{\Gamma}_n^{1/2}\xi_n'+\wh{\mu}_n\leq y|\mathcal{F})
\]
for all $y\in\mathbb{R}^m$, where $\xi_n$ and $\xi_n'$ are two independent $m$-dimensional standard Gaussian vectors jointly independent of $\mathcal{F}$. Therefore, it is enough to prove
\[
\sup_{y\in\mathbb{R}^m}|P(\widehat{\Gamma}_n^{1/2}\xi_n'+\wh{\mu}_n\leq y|\mathcal{F})-P(\Gamma_n^{1/2}\xi_n+\mu_n\leq y|\mathcal{F})|\to^p0
\]
as $n\to\infty$. In addition, thanks to the condition \eqref{diag-tight}, it suffices to prove the above convergence on the set $\Omega_b:=\{\diag(\Gamma_n)\geq b\}$ for an arbitrarily fixed $b>0$. More precisely, it is enough to prove
\begin{equation*}
P\left(\Omega_b\cap\left\{\sup_{y\in\mathbb{R}^m}|P(\widehat{\Gamma}_n^{1/2}\xi_n'+\wh{\mu}_n\leq y|\mathcal{F})-P(\Gamma_n^{1/2}\xi_n+\mu_n\leq y|\mathcal{F})|>\eta\right\}\right)\to0
\end{equation*}
as $n\to\infty$ for any $\eta>0$. 

We first prove
\begin{equation}\label{comp-aim1}
P\left(\Omega_b\cap\left\{\sup_{y\in\mathbb{R}^m}|P(\Gamma_n^{1/2}\xi_n+\wh{\mu}_n\leq y|\mathcal{F})-P(\Gamma_n^{1/2}\xi_n+\mu_n\leq y|\mathcal{F})|>\eta\right\}\right)\to0
\end{equation}
as $n\to\infty$ for any $\eta>0$. By Nazarov's inequality we have
\[
|P(\Gamma_n^{1/2}\xi_n+\wh{\mu}_n\leq y|\mathcal{F})-P(\Gamma_n^{1/2}\xi_n+\mu_n\leq y|\mathcal{F})|
\leq\frac{\|\wh{\mu}_n-\mu_n\|_{\ell_\infty}}{\sqrt{b}}\left(\sqrt{2\log m}+2\right)
\]
a.s.~on the set $\Omega_b$ for every $y\in\mathbb{R}^m$. Since the function $\tcr{y}\mapsto|P(\Gamma_n^{1/2}\xi_n+\wh{\mu}_n\leq y|\mathcal{F})-P(\Gamma_n^{1/2}\xi_n+\mu_n\leq y|\mathcal{F})|$ is a.s.~right-continuous, the above result yields
\[
\sup_{y\in\mathbb{R}^m}|P(\Gamma_n^{1/2}\xi_n+\wh{\mu}_n\leq y|\mathcal{F})-P(\Gamma_n^{1/2}\xi_n+\mu_n\leq y|\mathcal{F})|
\leq\frac{\|\wh{\mu}_n-\mu_n\|_{\ell_\infty}}{\sqrt{b}}\left(\sqrt{2\log m}+2\right)
\]
a.s.~on the set $\Omega_b$. Hence \eqref{comp-aim1} follows from the assumption \eqref{eq:consistent}. 

Thanks to \eqref{comp-aim1}, it suffices to prove
\begin{equation*}
P\left(\Omega_b\cap\left\{\sup_{y\in\mathbb{R}^m}|P(\widehat{\Gamma}_n^{1/2}\xi_n'+\wh{\mu}_n\leq y|\mathcal{F})-P(\Gamma_n^{1/2}\xi_n+\wh{\mu}_n\leq y|\mathcal{F})|>\eta\right\}\right)\to0
\end{equation*}
as $n\to\infty$ for any $\eta>0$. However, since we have
\[
\sup_{y\in\mathbb{R}^m}|P(\widehat{\Gamma}_n^{1/2}\xi_n'+\wh{\mu}_n\leq y|\mathcal{F})-P(\Gamma_n^{1/2}\xi_n+\wh{\mu}_n\leq y|\mathcal{F})|
=\sup_{y\in\mathbb{R}^m}|P(\widehat{\Gamma}_n^{1/2}\xi_n'\leq y|\mathcal{F})-P(\Gamma_n^{1/2}\xi_n\leq y|\mathcal{F})|,
\]
this amounts to proving
\begin{equation}\label{comp-aim2}
P\left(\Omega_b\cap\left\{\sup_{y\in\mathbb{R}^m}|P(\widehat{\Gamma}_n^{1/2}\xi_n'\leq y|\mathcal{F})-P(\Gamma_n^{1/2}\xi_n\leq y|\mathcal{F})|>\eta\right\}\right)\to0
\end{equation}
as $n\to\infty$ for any $\eta>0$. To prove this claim, we take a number $\varepsilon>0$ arbitrarily and set $\beta=\varepsilon^{-1}\log m$ as in the proof of Theorem \ref{thm:main}. 
Then we define the function $\Phi_\beta:\mathbb{R}^{m}\to\mathbb{R}$ by \eqref{def:Phi}. 
We also take a $C^\infty$ function $g:\mathbb{R}\to[0,1]$ such that all the derivatives of $g$ is bounded and $g(t)=1$ for $t\leq0$ and $g(t)=0$ for $t\geq1$. 

Fix a vector $y\in\mathbb{R}^{m}$ arbitrarily and define the function $\varphi:\mathbb{R}^{m}\to\mathbb{R}$ by
\begin{align*}
\varphi(w)&=g(\varepsilon^{-1}\Phi_\beta(w-y-\varepsilon)),\qquad w\in\mathbb{R}^{m}.
\end{align*}
Then we define the stochastic process $\Psi=(\Psi(t))_{t\in[0,1]}$ by
\[
\Psi(t)=E\left[\varphi\left(\sqrt{t}\widehat{\Gamma}_n^{1/2}\xi_n'+\sqrt{1-t}\Gamma_n^{1/2}\xi_n\right)|\mathcal{F}\right],\qquad t\in[0,1].
\]
We evidently have
\[
\frac{d\Psi(t)}{dt}=E\left[\partial_{w}^{\otimes1}\varphi\left(\sqrt{t}\widehat{\Gamma}_n^{1/2}\xi_n'+\sqrt{1-t}\Gamma_n^{1/2}\xi_n\right)\left[\frac{\widehat{\Gamma}_n^{1/2}\xi_n'}{\sqrt{t}}-\frac{\Gamma_n^{1/2}\xi_n}{\sqrt{1-t}}\right]|\mathcal{F}\right]\qquad \text{for all }t\in(0,1)
\]
with probability one. Then, Stein's identity yields
\[
\frac{d\Psi(t)}{dt}=E\left[\partial_{w}^{\otimes2}\varphi\left(\sqrt{t}\widehat{\Gamma}_n^{1/2}\xi_n'+\sqrt{1-t}\Gamma_n^{1/2}\xi_n\right)\left[\widehat{\Gamma}_n-\Gamma_n\right]|\mathcal{F}\right]\qquad \text{for all }t\in(0,1)
\]
with probability one. Consequently, we obtain
\begin{align*}
\left|E\left[\varphi\left(\widehat{\Gamma}_n^{1/2}\xi_n'\right)|\mathcal{F}\right]-E\left[\varphi\left(\Gamma_n^{1/2}\xi_n\right)|\mathcal{F}\right]\right|
&\leq\int_0^1\left|\frac{d\Psi(t)}{dt}\right|dt\\
&\leq C\varepsilon^{-2}(\log m)\|\widehat{\Gamma}_n-\Gamma_n\|_{\ell_\infty}
\end{align*}
by Lemmas 3--4 of \cite{CCK2015}, where $C>0$ is a constant which depends only on $g$. 
Now we have
\begin{align*}
P\left(\widehat{\Gamma}_n^{1/2}\xi_n'\leq y|\mathcal{F}\right)&\leq P\left(\Phi_\beta(\widehat{\Gamma}_n^{1/2}\xi_n'-y-\varepsilon)\leq0|\mathcal{F}\right)~(\because\text{Eq.\eqref{max-smooth}})\\
&\leq E\left[\varphi\left(\widehat{\Gamma}_n^{1/2}\xi_n'\right)|\mathcal{F}\right]
\leq E\left[\varphi\left(\Gamma_n^{1/2}\xi_n\right)|\mathcal{F}\right]+C\varepsilon^{-2}(\log m)\|\widehat{\Gamma}_n-\Gamma_n\|_{\ell_\infty}\\
&\leq P\left(\Phi_\beta(\Gamma_n^{1/2}\xi_n-y-\varepsilon)<\varepsilon|\mathcal{F}\right)+C\varepsilon^{-2}(\log m)\|\widehat{\Gamma}_n-\Gamma_n\|_{\ell_\infty}~(\because\text{the definition of $g$})\\
&\leq P\left(\Gamma_n^{1/2}\xi_n\leq y+2\varepsilon\right)+C\varepsilon^{-2}(\log m)\|\widehat{\Gamma}_n-\Gamma_n\|_{\ell_\infty}~(\because\text{Eq.\eqref{max-smooth}}).
\end{align*}
Since we have on the set $\Omega_b$
\begin{equation*}
P(\Gamma_n^{1/2}\xi_n\leq y+2\varepsilon|\mathcal{F})
\leq P(\Gamma_n^{1/2}\xi_n\leq y|\mathcal{F})+\frac{2\varepsilon}{\sqrt{b}}(\sqrt{2\log m}+2)
\end{equation*}
by the Nazarov inequality, we obtain
\begin{align*}
P\left(\widehat{\Gamma}_n^{1/2}\xi_n'\leq y|\mathcal{F}\right)
&\leq P(\Gamma_n^{1/2}\xi_n\leq y|\mathcal{F})+\frac{2\varepsilon}{\sqrt{b}}(\sqrt{2\log m}+2)+C\varepsilon^{-2}(\log m)\|\widehat{\Gamma}_n-\Gamma_n\|_{\ell_\infty}
\end{align*}
a.s.~on the set $\Omega_b$. By an analogous argument we also obtain
\begin{align*}
P\left(\widehat{\Gamma}_n^{1/2}\xi_n'\leq y|\mathcal{F}\right)
&\geq P(\Gamma_n^{1/2}\xi_n\leq y|\mathcal{F})-\frac{2\varepsilon}{\sqrt{b}}(\sqrt{2\log m}+2)-C\varepsilon^{-2}(\log m)\|\widehat{\Gamma}_n-\Gamma_n\|_{\ell_\infty}
\end{align*}
a.s.~on the set $\Omega_b$. Therefore, we conclude that
\begin{align*}
\left|P\left(\widehat{\Gamma}_n^{1/2}\xi_n'\leq y|\mathcal{F}\right)-P(\Gamma_n^{1/2}\xi_n\leq y|\mathcal{F})\right|
\leq \frac{2\varepsilon}{b}(\sqrt{2\log m}+2)+C\varepsilon^{-2}(\log m)\|\widehat{\Gamma}_n-\Gamma_n\|_{\ell_\infty}
\end{align*}
a.s.~on the set $\Omega_b$. Since the function $y\mapsto \left|P\left(\widehat{\Gamma}_n^{1/2}\xi_n'\leq y|\mathcal{F}\right)-P(\Gamma_n^{1/2}\xi_n\leq y|\mathcal{F})\right|$ is a.s.~right-continuous, the above result implies that
\begin{align*}
\sup_{y\in\mathbb{R}^m}\left|P\left(\widehat{\Gamma}_n^{1/2}\xi_n'\leq y|\mathcal{F}\right)-P(\Gamma_n^{1/2}\xi_n\leq y|\mathcal{F})\right|
\leq \frac{2\varepsilon}{b}(\sqrt{2\log m}+2)+C\varepsilon^{-2}(\log m)\|\widehat{\Gamma}_n-\Gamma_n\|_{\ell_\infty}
\end{align*}
a.s.~on the set $\Omega_b$. Hence we deduce
\begin{align*}
&P\left(\Omega_b\cap\left\{\sup_{y\in\mathbb{R}^m}|P(\widehat{\Gamma}_n^{1/2}\xi_n'\leq y|\mathcal{F})-P(\Gamma_n^{1/2}\xi_n\leq y|\mathcal{F})|>\eta\right\}\right)\\
&\leq P\left(\frac{2\varepsilon}{b}(\sqrt{2\log m}+2)+C\varepsilon^{-2}(\log m)\|\widehat{\Gamma}_n-\Gamma_n\|_{\ell_\infty}>\eta\right)
\end{align*}
for all $n\in\mathbb{N}$. Now, take a number $a>0$ such that $\frac{2a}{b}(\sqrt{2}+2/\sqrt{\log 2})\leq\frac{\eta}{2}$ and set $\varepsilon=a/\sqrt{\log m}$. Then the above inequality yields 
\begin{align*}
&P\left(\Omega_b\cap\left\{\sup_{y\in\mathbb{R}^m}|P(\widehat{\Gamma}_n^{1/2}\xi_n'\leq y|\mathcal{F})-P(\Gamma_n^{1/2}\xi_n\leq y|\mathcal{F})|>\eta\right\}\right)
\leq P\left(\frac{C}{a}(\log m)^2\|\widehat{\Gamma}_n-\Gamma_n\|_{\ell_\infty}>\frac{\eta}{2}\right).
\end{align*}
Therefore, 
\eqref{comp-aim2} follows from the assumption \eqref{eq:consistent}, which yields the desired result. 
\hfill\qed

\subsection{Proof of Proposition \ref{prop:quantile}}

We follow Step 3 in the proof of Theorem 2 from \cite{KS2016}. First, by assumption and Theorem 9.2.2 of \cite{Dudley2002} there is a sequence $\varepsilon_n$ of positive numbers tending to 0 such that 
\[
P\left(\mathcal{E}_n^c\right)\leq\varepsilon_n,\qquad
\sup_{x\in\mathbb{R}}\left|P\left(T_n\leq x\right)-P\left(T_n^\dagger\leq x\right)\right|\leq\varepsilon_n
\]
for all $n\in\mathbb{N}$, where
\[
\mathcal{E}_n=\left\{\sup_{x\in\mathbb{R}}\left|P\left(T_n^\dagger\leq x|\mathcal{F}\right)-P\left(T^*_n\leq x|\mathcal{F}\right)\right|\leq\varepsilon_n\right\}.
\]
Next, let us denote by $q_n^\dagger$ the $\mathcal{F}$-conditional quantile function of $T_n^\dagger$. 
Then, on the set $\mathcal{E}_n\cap E_n$ we have
\begin{align*}
P\left(T^*_n\leq q_n^\dagger(\alpha+\varepsilon_n)|\mathcal{F}\right)
\geq P\left(T_n^\dagger\leq q_n^\dagger(\alpha+\varepsilon_n)|\mathcal{F}\right)-\varepsilon_n
=\alpha.
\end{align*}
\tcr{Hence,} on $\mathcal{E}_n\cap E_n$ it holds that $q_n^*(\alpha)\leq q_n^\dagger(\alpha+\varepsilon_n).$ 
Therefore, we obtain
\begin{align*}
P\left(T_n\leq q_n^*(\alpha)\right)
&\leq P\left(T_n\leq q_n^\dagger(\alpha+\varepsilon_n)\right)+P(\mathcal{E}_n^c)+P(E_n^c)\\
&\leq P\left(T_n^\dagger\leq q_n^\dagger(\alpha+\varepsilon_n)\right)+2\varepsilon_n+P(E_n^c)
=\alpha+3\varepsilon_n+P(E_n^c).
\end{align*}
Meanwhile, for any $\omega\in\mathcal{E}_n\cap E_n$ and any $z\in\mathbb{R}$ such that $P(T_n^*\leq z|\mathcal{F})(\omega)\geq\alpha$, we have
\begin{align*}
P\left(T_n^\dagger\leq q_n^\dagger(\alpha-\varepsilon_n)(\omega)|\mathcal{F}\right)(\omega)
=\alpha-\varepsilon_n
\leq P(T_n^*\leq z|\mathcal{F})(\omega)-\varepsilon_n
\leq P\left(T_n^\dagger\leq z|\mathcal{F}\right)(\omega).
\end{align*}
\tcr{H}ence it holds that $q_n^\dagger(\alpha-\varepsilon_n)(\omega)\leq z$. This implies that 
$q_n^*(\alpha)\geq q_n^\dagger(\alpha-\varepsilon_n)$ 
on $\mathcal{E}_n\cap E_n$. Therefore, we obtain
\begin{align*}
P\left(T_n< q_n^*(\alpha)\right)
&\geq P\left(T_n< q_n^\dagger(\alpha-\varepsilon_n)\right)-P(\mathcal{E}_n^c)-P(E_n^c)\\
&\geq P\left(T_n^\dagger< q_n^\dagger(\alpha-\varepsilon_n)\right)-2\varepsilon_n-P(E_n^c)
=\alpha-3\varepsilon_n-P(E_n^c).
\end{align*}
Consequently, we obtain $P\left(T_n\leq q_n^*(\alpha)\right)\to\alpha$ as $n\to\infty$.\hfill\qed

\section{Proofs for Section \ref{sec:rc}}

\subsection{Proof of Theorem \ref{thm:rc}}

We first introduce some notation. 
%
For two sequences $(x_n),(y_n)$ of numbers, the notation $x_n\lesssim y_n$ means that there is a \textit{universal} constant $C>0$ such that $x_n\leq Cy_n$ for all $n$. Here, the value of the constant $C$ will change from line to line.  
We define the $d$-dimensional processes $\mathsf{A}=(\mathsf{A}_t)_{t\in[0,1]}$ and $\mathsf{M}=(\mathsf{M}_t)_{t\in[0,1]}$ by
\[
\mathsf{A}_t=\int_0^t\mu_sds,\qquad
\mathsf{M}_t=\int_0^t\sigma_sdB_s
\]
for every $t\in[0,1]$. 
If $\phi=(\phi_t)_{t\in[0,1]}$ is an $r$-dimensional $(\mathcal{F}_t)$-progressively measurable process such that $\int_0^1\|\phi_t\|_{\ell_2}^2dt<\infty$ a.s., we define
\[
\int_0^t\phi_s\cdot dB_s:=\sum_{a=1}^r\int_0^t\phi_s^adB^a_s
\]
for all $t\in[0,1]$. 

For every $n\in\mathbb{N}$, we set $I_h=I_h^n:=(t_{h-1},t_h]$ for every $h=1,\dots,n$ and define the filtration $(\mathcal{G}^n_t)_{t\in[0,1]}$ by $\mathcal{G}^n_0:=\mathcal{F}_0$ and
\[
\mathcal{G}^n_t:=\mathcal{F}_{t_{h-1}}
\]
when $t\in I_h$ for some $h=1,\dots,n$. Then we define the process $(\varsigma_t)_{t\in[0,1]}$ by
\[
\varsigma_t=E[\sigma_t|\mathcal{G}^n_{t}],\qquad t\in[0,1]
\]
(we subtract the index $n$ from $\varsigma_t$ although it depends on $n$). 
%
For all $i,j=1,\dots,d$, we define the symmetric $H^{\otimes2}$-valued random variable $u_n^{ij}$ by
\[
u_n^{ij}:=\sqrt{n}\sum_{h=1}^nf_n^{ij}1_{I_h\times I_h},
\]
where $f_n^{ij}=\symm\left(\varsigma^{i\cdot}\otimes\varsigma^{j\cdot}\right)$. We note the following result:
\begin{lemma}\label{lemma:double}
Given an index $n\in\mathbb{N}$, let $\xi=(\xi_t)_{t\in[0,1]}$ and $\eta=(\eta_t)_{t\in[0,1]}$ be $(\mathcal{G}^n_t)$-adapted $r$-dimensional processes such that $\sup_{t\in[0,1]}E[\|\xi_t\|_{\ell_2}^4+\|\eta_t\|_{\ell_2}^4]<\infty$. Then $\xi\otimes\eta1_{I_h\times I_h}\in\domain(\delta^2)$ and 
\begin{equation*}
\delta^2(\xi\otimes\eta1_{I_h\times I_h})=\int_{t_{h-1}}^{t_h}\left(\int_{t_{h-1}}^t\xi_s\cdot dB_s\right)\eta_t\cdot dB_t+\int_{t_{h-1}}^{t_h}\left(\int_{t_{h-1}}^t\xi_s\cdot dB_s\right)\eta_t\cdot dB_t
\end{equation*}
for every $h=1,\dots,n$.
\end{lemma}

\begin{proof}
Set $S:=\{(s,t)\in[0,1]^2:s\leq t\}$. For any $t\in[0,1]$, the process $(\xi_s\eta_t1_{(I_h\times I_h)\cap S}(s,t))_{s\in[0,1]}$ is evidently $\mathbf{F}$-predictable and $H$-valued, \tcr{so} it belongs to $\domain(\delta)$ and
\[
\delta(\xi\eta_t1_{(I_h\times I_h)\cap S}(\cdot,t))=\eta_t1_{I_h}(t)\int_0^t\xi_s1_{I_h}(s)\cdot dB_s
\]
by Proposition 1.3.11 of \cite{Nualart2006}. Moreover, from the above expression the process $(\delta(\xi\eta_t1_{(I_h\times I_h)\cap S}(\cdot,t)))_{t\in[0,1]}$ is evidently $\mathbf{F}$-predictable and $H$-valued. Therefore, Proposition 1.3.11 of \cite{Nualart2006} and Proposition 2.6 of \cite{NZ1988} imply that $\xi\otimes\eta1_{(I_h\times I_h)\cap S}$ belongs to $\domain(\delta^2)$ and
\[
\delta^2(\xi\otimes\eta1_{(I_h\times I_h)\cap S})
=\int_{t_{h-1}}^{t_h}\left(\int_{t_{h-1}}^t\xi_s\cdot dB_s\right)\eta_t\cdot dB_t.
\]
Similarly, we can show that $\xi\otimes\eta1_{(I_h\times I_h)\cap S^c}\in \domain(\delta^2)$ and
\[
\delta^2(\xi\otimes\eta1_{(I_h\times I_h)\cap S^c})
=\int_{t_{h-1}}^{t_h}\left(\int_{t_{h-1}}^t\eta_s\cdot dB_s\right)\xi_t\cdot dB_t.
\]
This completes the proof. 
\end{proof}
Thanks to Lemma \ref{lemma:double}, we have $u_n^{ij}\in\domain(\delta^2)$, \tcr{so} we can define the variable $M_n^{ij}$ by
\[
M_n^{ij}=\delta^2(u_n^{ij}).
\]

Next we prove some auxiliary results. 
%
We begin by noting some elementary facts which are frequently used throughout the proof. First, for any random variable $\xi$ and any $p,q\in(0,\infty)$, it holds that
\[
\||\xi|^q\|_p=\|\xi\|_{pq}^q.
\]
Second, for two random variables $\xi,\eta$ and numbers $p\in(0,\infty)$, $q\in(1,\infty)$, we have
\begin{equation*}
\|\xi\eta\|_p\leq\|\xi\|_{qp}\|\eta\|_{\frac{q}{q-1}p}.
\end{equation*}
This is a consequence of the H\"older inequality. These facts will be used without reference in the following. 
We also refer to two inequalities which are repeatedly used throughout the proof. The first one is the following integral version of the Minkowski inequality:
\begin{proposition}\label{minkowski}
Let $(\mathcal{X},\mathcal{A},\mathfrak{m})$ be a $\sigma$-finite measure space and $f:\mathcal{X}\times \Omega\to[0,\infty]$ be an $\mathcal{A}\otimes\mathcal{F}$-measurable function. Then we have
\[
\left\|\int_{\mathcal{X}}f(x)\mathfrak{m}(dx)\right\|_p
\leq\int_{\mathcal{X}}\|f(x)\|_p\mathfrak{m}(dx)
\]
for all $p\in[1,\infty]$. 
\end{proposition}
Proposition \ref{minkowski} is an easy consequence of the standard Minkowski inequality via approximating the function $f$ by simple functions (see also Proposition C.4 of \cite{Janson1997}). 

The second one is the following Burkholder-Davis-Gundy inequality with a sharp constant:
\begin{proposition}[\citet{BY1982}, Proposition 4.2]\label{sharp-BDG}
There is a universal constant $c>0$ such that
\[
\left\|\sup_{0\leq t\leq T}|M_t|\right\|_p\leq c\sqrt{p}\left\|\langle M\rangle_T^{1/2}\right\|_p
\]
for any $p\in[2,\infty)$ and any continuous martingale $M=(M_t)_{t\in[0,T]}$ with $M_0=0$. 
\end{proposition}

We then prove some auxiliary estimates. 
\begin{lemma}\label{lemma:BDG}
There is a universal constant $C>0$ such that
\[
\left\|\sum_{h=h_0+1}^{h_1}\int_{t_{h-1}}^{t_h}\left(\int_{t_{h-1}}^t\eta_s\cdot dB_s\right)\xi_t\cdot dB_t\right\|_p
\leq C\frac{p\sqrt{h_1-h_0}}{n}\sup_{t_{h_0}< t< t_{h_1}}\|\xi_t\|_{qp,\ell_2}\sup_{t_{h_0}< s< t_{h_1}}\|\eta_s\|_{\frac{q}{q-1}p,\ell_2}
\]
for any $p\in[2,\infty)$, $q\in(1,\infty)$, $n\in\mathbb{N}$, $h_0,h_1=0,1,\dots,n$ such that $h_0< h_1$ and any $r$-dimensional $(\mathcal{F}_t)$-progressively measurable processes $\xi$ and $\eta$ such that $\sup_{t\in[0,1]}(\|\xi_t\|_p+\|\eta_t\|_p)<\infty$ for all $p\in[1,\infty)$. 
\end{lemma}

\begin{proof}
Set $q'=q/(1-q)$. By Propositions \ref{minkowski}--\ref{sharp-BDG} we have
\begin{align*}
\left\|\sum_{h=h_0+1}^{h_1}\int_{t_{h-1}}^{t_h}\left(\int_{t_{h-1}}^t\eta_s\cdot dB_s\right)\xi_t\cdot dB_t\right\|_p
&\lesssim\sqrt{p}\left\|\sqrt{\sum_{h=h_0+1}^{h_1}\int_{t_{h-1}}^{t_h}\|\xi_t\|_{\ell_2}^2\left(\int_{t_{h-1}}^t\eta_s\cdot dB_s\right)^2dt}\right\|_p\\
&\leq\sqrt{p\sum_{h=h_0+1}^{h_1}\int_{t_{h-1}}^{t_h}\|\xi_t\|_{qp,\ell_2}^2\left\|\int_{t_{h-1}}^t\eta_s\cdot dB_s\right\|_{q'p}^2dt}\\
&\lesssim p\sqrt{\sum_{h=h_0+1}^{h_1}\int_{t_{h-1}}^{t_h}\|\xi_t\|_{qp,\ell_2}^2\left\|\sqrt{\int_{t_{h-1}}^t\|\eta_s\|_{\ell_2}^2ds}\right\|_{q'p}^2dt}\\
&\leq\frac{p\sqrt{h_1-h_0}}{n}\sup_{t_{h_0}< t< t_{h_1}}\|\xi_t\|_{qp,\ell_2}\sup_{t_{h_0}< s< t_{h_1}}\|\eta_s\|_{q'p,\ell_2}.
\end{align*}
Hence we obtain the desired result. 
\end{proof}

\begin{lemma}\label{ocone}
There is a universal constant $C>0$ such that
\[
\max_{1\leq h\leq n-1}\sup_{t\in[t_{h-1},t_{h+1}]}\left\|\xi_t-E[\xi_t|\mathcal{F}_{t_{h-1}}]\right\|_{2p,\ell_2}\leq C\sqrt{\frac{p}{n}}\sup_{0\leq u\leq v\leq1}\|D_{u}\xi_{v}\|_{2p,\ell_2}
\]
for all $n\in\mathbb{N}$, $p\in[2,\infty)$, $\mathsf{r}\in\mathbb{N}$ and any $\mathsf{r}$-dimensional $(\mathcal{F}_t)$-progressively measurable process $\xi$ such that $\xi_t\in\mathbb{D}_{1,\infty}(\mathbb{R}^{\mathsf{r}})$ for all $t\in[0,1]$. 
\end{lemma}

\begin{proof}
By the Clark-Ocone formula (Proposition A.1 of \cite{NP1988}) we have
\[
\xi^{a}_t=E[\xi^{a}_t|\mathcal{F}_{t_{h-1}}]+\int_{t_{h-1}}^tE[D_s\xi^{a}_t|\mathcal{F}_{s}]\cdot dB_s\qquad\text{a.s.}
\]
for all $t\in [t_{h-1},t_{h+1}]$ and $a=1,\dots,\mathsf{r}$. Therefore, it suffices to show that there is a universal constant $C'>0$ such that
\[
\max_{1\leq h\leq n-1}\sup_{t\in[t_{h-1},t_{h+1}]}\left\|\sum_{a=1}^\mathsf{r}\left(\int_{t_{h-1}}^tE[D_s\xi^{a}_t|\mathcal{F}_{s}]\cdot dB_s\right)^2\right\|_p\leq C'\frac{p}{n}\sup_{0\leq u\leq v\leq1}\|D_{u}\xi_{v}\|_{2p,\ell_2}^2
\]
for all $n\in\mathbb{N}$ and $p\in[2,\infty)$. 

Fix $h=1,\dots,n-1$ and $t\in[t_{h-1},t_{h+1}]$ arbitrarily. By It\^o's formula we have
\begin{align*}
&\sum_{a=1}^\mathsf{r}\left(\int_{t_{h-1}}^\tau E[D_s\xi^{a}_t|\mathcal{F}_{s}]\cdot dB_s\right)^2\\
&=2\int_{t_{h-1}}^\tau \sum_{a=1}^\mathsf{r}\left(\int_{t_{h-1}}^sE[D_u\xi^{a}_t|\mathcal{F}_{u}]\cdot dB_u\right)E[D_s\xi^{a}_t|\mathcal{F}_{s}]\cdot dB_s
+\int_{t_{h-1}}^\tau \sum_{a=1}^\mathsf{r}\left\|E\left[D_s\xi^{a}_t|\mathcal{F}_{s}\right]\right\|_{\ell_2}^2ds\\
&=:\mathbf{I}_\tau+\mathbf{II}_\tau
\end{align*}
for every $\tau\in[t_{h-1},t]$. The Lyapunov inequality and Proposition \ref{minkowski} yield
\begin{align*}
\left\|\mathbf{II}_\tau\right\|_p
\leq\frac{1}{n}\sup_{0\leq u\leq v\leq1}\|D_{u}\xi_{v}\|_{2p,\ell_2}^2.
\end{align*}
Meanwhile, 
\tcr{we have}
\begin{align*}
\left\|\mathbf{I}_\tau\right\|_{p}
&\lesssim\sqrt{p}\left\|\sqrt{\sum_{b=1}^r\int_{t_{h-1}}^\tau\left|\sum_{a=1}^\mathsf{r}\left(\int_{t_{h-1}}^sE[D_u\xi^{a}_t|\mathcal{F}_{u}]\cdot dB_u\right)E[D^{(b)}_s\xi^{a}_t|\mathcal{F}_{s}]\right|^2ds}\right\|_{p}
~\tcr{(\because\text{Proposition \ref{sharp-BDG}})}\\
&\leq\sqrt{p}\left\|\int_{t_{h-1}}^\tau\sum_{a=1}^\mathsf{r}\left(\int_{t_{h-1}}^sE[D_u\xi^{a}_t|\mathcal{F}_{u}]\cdot dB_u\right)^2\sum_{a=1}^\mathsf{r}\sum_{b=1}^rE[D^{(b)}_s\xi^{a}_t|\mathcal{F}_{s}]^2ds\right\|_{p/2}^{1/2}
~\tcr{(\because\text{Schwarz})}\\
&\leq\sqrt{p\int_{t_{h-1}}^\tau\left\|\sum_{a=1}^\mathsf{r}\left(\int_{t_{h-1}}^sE[D_u\xi^{a}_t|\mathcal{F}_{u}]dB_u\right)^2E\left[\|D_s\xi_t\|_{\ell_2}^2|\mathcal{F}_{s}\right]\right\|_{p/2}ds}
~(\because\text{Lyapunov, Proposition \ref{minkowski}})\\
&\leq\sqrt{p\int_{t_{h-1}}^\tau\left\|\sum_{a=1}^\mathsf{r}\left(\int_{t_{h-1}}^sE[D_u\xi^{a}_t|\mathcal{F}_{u}]\cdot dB_u\right)^2\right\|_p\left\|E\left[\|D_s\xi_t\|_{\ell_2}^2|\mathcal{F}_{s}\right]\right\|_{p}ds}
~(\because\text{Schwarz})\\
&\leq\sqrt{p\sup_{0\leq u\leq v\leq1}\|D_{u}\xi_{v}\|_{2p,\ell_2}^2\int_{t_{h-1}}^\tau\left\|\sum_{a=1}^\mathsf{r}\left(\int_{t_{h-1}}^sE[D_u\xi^{a}_t|\mathcal{F}_{u}]\cdot dB_u\right)^2\right\|_pds}
~\tcr{(\because\text{Lyapunov})}.
\end{align*}
Therefore, defining the function $g:[t_{h-1},t]\to[0,\infty)$ by
\[
g(\tau)=\left\|\sum_{a=1}^\mathsf{r}\left(\int_{t_{h-1}}^\tau E[D_s\xi^{a}_t|\mathcal{F}_{s}]\cdot dB_s\right)^2\right\|_{p}^2,\qquad \tau\in[t_{h-1},t],
\]
we obtain
\[
g(\tau)\leq\frac{2}{n^2}\sup_{0\leq u\leq v\leq1}\|D_{u}\xi_{v}\|_{2p,\ell_2}^4+C_0p\sup_{0\leq u\leq v\leq1}\|D_{u}\xi_{v}\|_{2p,\ell_2}^2\int_{t_{h-1}}^\tau\sqrt{g(s)}ds
\]
for any $\tau\in[t_{h-1},t]$ with some universal constant $C_0>0$. Hence the Bihari inequality (cf.~Section 3 of \cite{Bihari1956}) yields
\[
\sqrt{g(t)}\leq\frac{\sqrt{2}}{n}\sup_{0\leq u\leq v\leq1}\|D_{u}\xi_{v}\|_{2p,\ell_2}^2
+\frac{C_0p}{2n}\sup_{0\leq u\leq v\leq1}\|D_{u}\xi_{v}\|_{2p,\ell_2}^2.
\]
This implies that the desired result holds true with the constant $C'=1/\sqrt{2}+C_0/2$. 
\end{proof}

\begin{lemma}\label{sigma-l2}
Under the assumptions of Theorem \ref{thm:rc}, it holds that 
$\|\varsigma^{i\cdot}_t\|_{2p,\ell_2}\leq\|\sigma^{i\cdot}_t\|_{2p,\ell_2}=\|\Sigma^{ii}_t\|_{p}^{1/2}$
for any $t\in[0,1]$, $i=1,\dots,d$ and $p\geq1$. 
\end{lemma}

\begin{proof}
The last equality is evident from the identity $\|\sigma^{i\cdot}_t\|_{\ell_2}^2=\Sigma^{ii}_t$. 
Meanwhile, the Lyapunov inequality yields
\begin{align*}
\|\varsigma^{i\cdot}_t\|_{\ell_2}^2
=\sum_{a=1}^r\left(E\left[\sigma^{ia}_t|\mathcal{G}^n_t\right]\right)^2
\leq \sum_{a=1}^rE\left[\left(\sigma^{ia}_t\right)^2|\mathcal{G}^n_t\right]
=E\left[\left\|\sigma^{i\cdot}_t\right\|_{\ell_2}^2|\mathcal{G}^n_t\right].
\end{align*}
Therefore, the Lyapunov inequality again yields 
$
E\left[\|\varsigma^{i\cdot}_t\|_{\ell_2}^{2p}\right]
\leq E\left[\left(E\left[\left\|\sigma^{i\cdot}_t\right\|_{\ell_2}^2|\mathcal{G}^n_t\right]\right)^p\right]
\leq E\left[\left\|\sigma^{i\cdot}_t\right\|_{\ell_2}^{2p}\right].
$ 
This means $\|\varsigma^{i\cdot}_t\|_{2p,\ell_2}\leq\|\sigma^{i\cdot}_t\|_{2p,\ell_2}$. 
\end{proof}

\begin{lemma}\label{sigma-deriv}
Under the assumptions of Theorem \ref{thm:rc}, 
for all $i=1,\dots,d$ and $u\in[0,1]$, $\varsigma^{i\cdot}_u\in\mathbb{D}_{2,\infty}(\mathbb{R}^r)$ and $D_s\varsigma^{i\cdot}_u,D_{s,t}\varsigma^{i\cdot}_u$ are $\mathcal{G}^n_u$-measurable for any $s,t\in[0,1]$. Moreover, the following estimates hold true for any $p\in[1,\infty)$ and $s,t\in[0,1]$:
\begin{align}
\|D_{t}\varsigma^{i\cdot}_u\|_{2p,\ell_2}&\leq\|D_{t}\sigma^{i\cdot}_u\|_{2p,\ell_2},\label{eq:deriv1}\\
\|D_{s,t}\varsigma^{i\cdot}_u\|_{2p,\ell_2}&\leq\|D_{s,t}\sigma^{i\cdot}_u\|_{2p,\ell_2},\label{eq:deriv2}\\
\max_{1\leq k\leq d}\left\|\sum_{a=1}^r\varsigma^{ka}_{s}D_s^{(a)}\varsigma^{i\cdot}_{u}\right\|_{p,\ell_2}
&\leq\max_{1\leq k\leq d}\left\|\Sigma^{kk}_u\right\|_{p}^{1/2}\left\|D_s\sigma^{i\cdot}_{u}\right\|_{2p,\ell_2},\label{eq:deriv3}\\
\max_{1\leq k,l\leq d}\left\|\sum_{a,b=1}^r\varsigma^{ka}_{s}\varsigma^{lb}_{t}D^{(a,b)}_{s,t}\varsigma^{i\cdot}_{u}\right\|_{p,\ell_2}
&\leq\max_{1\leq k\leq d}\left\|\Sigma^{kk}_s\right\|_{\frac{3}{2}p}\left\|D_{s,t}\sigma^{i\cdot}_{u}\right\|_{3p,\ell_2}.\label{eq:deriv4}
\end{align}
\end{lemma}

\begin{proof}
First, by Proposition 3.1 of \cite{JS1990} $\varsigma^{i\cdot}_u\in\mathbb{D}_{2,2}(\mathbb{R}^r)$ and we have
\[
D_s^{(a)}\varsigma^{i\cdot}_u=E\left[D_s^{(a)}\sigma^{i\cdot}_u|\mathcal{G}_u^n\right]1_{\left[0,(\lceil nu\rceil-1)/n\right]}(s),\qquad
D_{s,t}^{(a,b)}\varsigma^{i\cdot}_u=E\left[D_{s,t}^{(a,b)}\sigma^{i\cdot}_u|\mathcal{G}_u^n\right]1_{\left[0,(\lceil nu\rceil-1)/n\right]^2}(s,t)
\]
for any $s,t\in[0,1]$ and $a,b=1,\dots,r$. In particular, $D_s\varsigma^{i\cdot}_u,D_{s,t}\varsigma^{i\cdot}_u$ are $\mathcal{G}^n_u$-measurable. Moreover, \eqref{eq:deriv1}--\eqref{eq:deriv2} can be shown in an analogous way to the proof of Lemma \ref{sigma-l2}, which also implies that $\varsigma^{i\cdot}_u\in\mathbb{D}_{2,\infty}(\mathbb{R}^r)$. 

Next, the Schwarz inequality, Lemma \ref{sigma-l2} and \eqref{eq:deriv1} yield
\begin{align*}
\left\|\sum_{a=1}^r\varsigma^{ka}_{s}D_s^{(a)}\varsigma^{i\cdot}_{u}\right\|_{p,\ell_2}
&=\left\|\sqrt{\sum_{b=1}^r\left(\sum_{a=1}^r\varsigma^{ka}_{s}D_s^{(a)}\varsigma^{ib}_{u}\right)^2}\right\|_{p}
\leq\left\|\|\varsigma^{k\cdot}_s\|_{\ell_2}\left\|D_s\varsigma^{i\cdot}_{u}\right\|_{\ell_2}\right\|_{p}
\leq\left\|\Sigma^{kk}_u\right\|_{p}^{1/2}\left\|D_s\sigma^{i\cdot}_{u}\right\|_{2p,\ell_2},
\end{align*}
\tcr{and thus} we obtain \eqref{eq:deriv3}.

Finally, the Schwarz inequality, Lemma \ref{sigma-l2} and \eqref{eq:deriv2} yield
\begin{align*}
\left\|\sum_{a,b=1}^r\varsigma^{ka}_{s}\varsigma^{lb}_{t}D^{(a,b)}_{s,t}\varsigma^{i\cdot}_{u}\right\|_{p,\ell_2}
&=\left\|\sqrt{\sum_{c=1}^r\left(\sum_{a,b=1}^r\varsigma^{ka}_{s}\varsigma^{lb}_{t}D^{(a,b)}_{s,t}\varsigma^{ic}_{u}\right)^2}\right\|_{p}
\leq\left\|\|\varsigma^{k\cdot}_{s}\|_{\ell_2}\|\varsigma^{l\cdot}_{t}\|_{\ell_2}\|D^2_{s,t}\varsigma^{i\cdot}_{u}\|_{\ell_2}\right\|_{p}\\
&\leq\left\|\|\varsigma^{k\cdot}_{s}\|_{\ell_2}\|\varsigma^{l\cdot}_{t}\|_{\ell_2}\right\|_{\frac{3}{2}p}\|D^2_{s,t}\varsigma^{i\cdot}_{u}\|_{3p,\ell_2}
\leq\left\|\Sigma^{kk}_s\right\|_{\frac{3}{2}p}^{1/2}\left\|\Sigma^{ll}_t\right\|_{\frac{3}{2}p}^{1/2}\|D^2_{s,t}\sigma^{i\cdot}_{u}\|_{3p,\ell_2},
\end{align*}
\tcr{so} we obtain \eqref{eq:deriv4} and thus complete the proof.
\end{proof}

Now we turn to the main body of the proof. We begin by evaluating the approximation error between $\sqrt{n}(\widehat{[Y^i,Y^j]}^n_1-[Y^i,Y^j]_1)$ and $M_n^{ij}$. 
\begin{lemma}\label{drift}
Under the assumptions of Theorem \ref{thm:rc}, 
it holds that
\[
\max_{1\leq i,j\leq d}\left\|\sqrt{n}\widehat{[\mathsf{A}^i,\mathsf{A}^j]}^n_1\right\|_p
\leq\frac{1}{\tcr{\sqrt{n}}}\max_{1\leq i\leq d}\sup_{0\leq s\leq 1}\|\mu^i_s\|_{2p}^2
\]
for any $p\in[1,\infty)$ and $n\in\mathbb{N}$. 
Moreover, there is a universal constant $C>0$ such that
\begin{align*}
\left\|\sqrt{n}\widehat{[\mathsf{M}^i,\mathsf{A}^j]}^n_1\right\|_p
\leq\frac{C}{\sqrt{n}}\sup_{0\leq s\leq 1}\|\sigma^{i\cdot}_s\|_{2p,\ell_2}\left(
\sqrt{p}\sup_{0\leq t\leq 1}\left\|\mu^{j}_t\right\|_{2p}
+p\sup_{0\leq u\leq v\leq 1}\|D_u\mu^j_v\|_{2p,\ell_2}
\right)
\end{align*}
for any $p\in[2,\infty)$, $n\in\mathbb{N}$ and $i,j=1,\dots,d$.
\end{lemma}

\begin{proof}
The first claim is an immediate consequence of the H\"older inequality and Proposition \ref{minkowski}. 

To prove the second claim, by It\^o's formula we decompose $\sqrt{n}\widehat{[\mathsf{M}^i,\mathsf{A}^j]}^n_1$ as
\begin{align*}
\sqrt{n}\widehat{[\mathsf{M}^i,\mathsf{A}^j]}^n_1
&=\sqrt{n}\sum_{h=1}^n\left\{\int_{t_{h-1}}^{t_h}\left(\int_{t_{h-1}}^t\mu^{j}_sds\right)\sigma^{i\cdot}_t\cdot dB_t+\int_{t_{h-1}}^{t_h}\left(\int_{t_{h-1}}^t\sigma^{i\cdot}_s\cdot dB_s\right)\mu^{j}_tdt\right\}\\
&=:\mathbf{I}_n^{ij}+\mathbf{II}_n^{ij}.
\end{align*}
By Propositions \ref{minkowski}--\ref{sharp-BDG} we have
\begin{align*}
\|\mathbf{I}_n^{ij}\|_p
&\lesssim \sqrt{np}\left\|\sqrt{\sum_{h=1}^n\int_{t_{h-1}}^{t_h}\left(\int_{t_{h-1}}^t\mu^{j}_sds\right)^2\|\sigma^{i\cdot}_t\|_{\ell_2}^2 dt}\right\|_p\\
&\leq\sqrt{\frac{p}{n}}\sup_{0\leq s\leq 1}\left\|\mu^{j}_s\right\|_{2p}\sup_{0\leq t\leq 1}\|\sigma^{i\cdot}_t\|_{2p,\ell_2}.
\end{align*}
In the meantime, we further decompose $\mathbf{II}_n^{ij}$ as
\begin{align*}
\mathbf{II}_n^{ij}&=\sqrt{n}\sum_{h=1}^n\left\{\int_{t_{h-1}}^{t_h}\left(\int_{t_{h-1}}^t\sigma^{i\cdot}_s\cdot dB_s\right)E\left[\mu^{j}_t|\mathcal{G}_t^n\right]dt
+\int_{t_{h-1}}^{t_h}\left(\int_{t_{h-1}}^t\sigma^{i\cdot}_s\cdot dB_s\right)\left(\mu^j_t-E\left[\mu^{j}_t|\mathcal{G}_t^n\right]\right)dt\right\}\\
&=:\mathbf{II}_n^{ij}(1)+\mathbf{II}_n^{ij}(2).
\end{align*}
Since $E\left[\mu^{j}_t|\mathcal{G}_t^n\right]$ is $\mathcal{F}_{t_{h-1}}$-measurable for $t\in I_h$, we have
\[
\mathbf{II}_n^{ij}(1)=\sqrt{n}\sum_{h=1}^n\int_{t_{h-1}}^{t_h}\left(\int_{t_{h-1}}^{t_h}1_{(t_{h-1},t]}(s)E\left[\mu^{j}_t|\mathcal{G}_t^n\right]\sigma^{i\cdot}_s\cdot dB_s\right)dt.
\]
Therefore, the stochastic Fubini theorem (e.g.~Corollary 5.28 of \cite{Medv2007}) yields
\[
\mathbf{II}_n^{ij}(1)=\sqrt{n}\sum_{h=1}^n\int_{t_{h-1}}^{t_h}\left(\int_{t_{h-1}}^{t_h}1_{(t_{h-1},t]}(s)E\left[\mu^{j}_t|\mathcal{G}_t^n\right]dt\right)\sigma^{i\cdot}_s\cdot dB_s.
\]
Hence, Propositions \ref{minkowski}--\ref{sharp-BDG} and the Lyapunov inequality imply that
\begin{align*}
\left\|\mathbf{II}_n^{ij}(1)\right\|_p
&\lesssim\sqrt{np}\left\|\sqrt{\sum_{h=1}^n\int_{t_{h-1}}^{t_h}\left(\int_{t_{h-1}}^{t_h}1_{(t_{h-1},t]}(s)E\left[\mu^{j}_t|\mathcal{G}_t^n\right]dt\right)^2\|\sigma^{i\cdot}_s\|_{\ell_2}^2ds}\right\|_p\\
&\leq\sqrt{\frac{p}{n}}\sup_{0\leq t\leq 1}\|\mu^j_t\|_{2p}\sup_{0\leq s\leq 1}\|\sigma^{i\cdot}_s\|_{2p,\ell_2}.
\end{align*}
Meanwhile, Propositions \ref{minkowski}--\ref{sharp-BDG} and Lemma \ref{ocone} yield
\begin{align*}
\left\|\mathbf{II}_n^{ij}(2)\right\|_p
&\leq\sqrt{n}\sum_{h=1}^n
\int_{t_{h-1}}^{t_h}\left\|\int_{t_{h-1}}^t\sigma^{i\cdot}_s\cdot dB_s\right\|_{2p}\left\|\mu^j_t-E\left[\mu^{j}_t|\mathcal{G}_t^n\right]\right\|_{2p}dt\\
&\lesssim p\sum_{h=1}^n\int_{t_{h-1}}^{t_h}\left\|\sqrt{\int_{t_{h-1}}^t\|\sigma^{i\cdot}_s\|_{\ell_2}^2ds}\right\|_{2p}dt\sup_{0\leq u\leq v\leq 1}\|D_u\mu^j_v\|_{2p,\ell_2}\\
&\leq\frac{p}{\sqrt{n}}\sup_{0\leq s\leq1}\|\sigma^{i\cdot}_s\|_{2p,\ell_2}\sup_{0\leq u\leq v\leq 1}\|D_u\mu^j_v\|_{2p,\ell_2}.
\end{align*}
Combining these estimates, we complete the proof. 
\end{proof}

\begin{lemma}\label{approx-M}
Under the assumptions of Theorem \ref{thm:rc}, 
there is a universal constant $C>0$ such that
\[
\max_{1\leq i,j\leq d}\left\|\sqrt{n}\left(\widehat{[\mathsf{M}^i,\mathsf{M}^j]}^n_1-[Y^i,Y^j]_1\right)-M_n^{ij}\right\|_p
\leq C\frac{p^{3/2}}{\sqrt{n}}\max_{1\leq i\leq d}\sup_{0\leq s\leq 1}\left\|\Sigma^{ii}_s\right\|_{p}^{1/2}\max_{1\leq j\leq d}\sup_{0\leq u\leq v\leq1}\|D_{u}\sigma^{j\cdot}_{v}\|_{2p,\ell_2}
\]
for every $n\in\mathbb{N}$ and $p\in[2,\infty)$. 
\end{lemma}

\begin{proof}
By It\^o's formula we deduce the following decomposition:
\begin{align*}
&\sqrt{n}\left(\widehat{[\mathsf{M}^i,\mathsf{M}^j]}^n_1-[Y^i,Y^j]_1\right)\\
&=\sqrt{n}\sum_{h=1}^n\left\{\int_{t_{h-1}}^{t_h}\left(\int_{t_{h-1}}^t\sigma^{j\cdot}_s\cdot dB_s\right)\sigma^{i\cdot}_t\cdot dB_t+\int_{t_{h-1}}^{t_h}\left(\int_{t_{h-1}}^t\sigma^{i\cdot}_s\cdot dB_s\right)\sigma^{j\cdot}_t\cdot dB_t\right\}\\
&=:\mathbf{I}_n^{ij}+\mathbf{II}_n^{ij}.
\end{align*}
Propositions \ref{minkowski}--\ref{sharp-BDG} and Lemmas \ref{ocone}--\ref{sigma-l2} yield
\begin{align*}
&\left\|\sqrt{n}\sum_{h=1}^n\int_{t_{h-1}}^{t_h}\left(\int_{t_{h-1}}^t\sigma^{j\cdot}_s\cdot dB_s\right)\left(\sigma^{i\cdot}_t-\varsigma^{i\cdot}_t\right)\cdot dB_t\right\|_p\\
&\lesssim\sqrt{np}\left\|\sqrt{\sum_{h=1}^n\int_{t_{h-1}}^{t_h}\left(\int_{t_{h-1}}^t\sigma^{j\cdot}_s\cdot dB_s\right)^2\left\|\sigma^{i\cdot}_t-\varsigma^{i\cdot}_t\right\|_{\ell_2}^2dt}\right\|_p\\
&\leq\sqrt{np}\sqrt{\sum_{h=1}^n\int_{t_{h-1}}^{t_h}\left\|\int_{t_{h-1}}^t\sigma^{j\cdot}_s\cdot dB_s\right\|_{2p}^2\left\|\sigma^{i\cdot}_t-\varsigma^{i\cdot}_t\right\|_{2p,\ell_2}^2dt}\\
&\lesssim p\sqrt{\sup_{0\leq s\leq 1}\left\|\Sigma^{jj}_s\right\|_{p}\sup_{0\leq t\leq 1}\left\|\sigma^{i\cdot}_t-\varsigma^{i\cdot}_t\right\|_{2p,\ell_2}^2}
\lesssim\frac{p^{3/2}}{\sqrt{n}}\sup_{0\leq s\leq 1}\left\|\Sigma^{jj}_s\right\|_{p}^{1/2}\sup_{0\leq u\leq v\leq1}\|D_{u}\sigma^{i\cdot}_{v}\|_{2p,\ell_2}
\end{align*}
and
\begin{align*}
&\left\|\sqrt{n}\sum_{h=1}^n\int_{t_{h-1}}^{t_h}\left(\int_{t_{h-1}}^t\left(\sigma^{j\cdot}_s-\varsigma^{j\cdot}_s\right)\cdot dB_s\right)\varsigma^{i\cdot}_t\cdot dB_t\right\|_p\\
&\lesssim \sqrt{np}\left\|\sqrt{\sum_{h=1}^n\int_{t_{h-1}}^{t_h}\left(\int_{t_{h-1}}^t\left(\sigma^{j\cdot}_s-\varsigma^{j\cdot}_s\right)\cdot dB_s\right)^2\|\varsigma^{i\cdot}_t\|_{\ell_2}^2dt}\right\|_p\\
&\lesssim \sqrt{n}p\sqrt{\sum_{h=1}^n\int_{t_{h-1}}^{t_h}\left\|\sqrt{\int_{t_{h-1}}^t\left\|\sigma^{j\cdot}_s-\varsigma^{j\cdot}_s\right\|_{\ell_2}^2ds}\right\|_{2p}^2\|\varsigma^{i\cdot}_t\|_{2p,\ell_2}^2dt}\\
&\lesssim\frac{p^{3/2}}{\sqrt{n}}\sup_{0\leq u\leq v\leq1}\|D_{u}\sigma^{j\cdot}_{v}\|_{2p,\ell_2}\sup_{0\leq s\leq 1}\left\|\Sigma^{ii}_s\right\|_{p}^{1/2}.
\end{align*}
Hence we obtain
\[
\left\|\mathbf{I}_n^{ij}-\sqrt{n}\sum_{h=1}^n\int_{t_{h-1}}^{t_h}\left(\int_{t_{h-1}}^t\varsigma^{j\cdot}_s\cdot dB_s\right)\varsigma^{i\cdot}_t\cdot dB_t\right\|_p
\lesssim \frac{p^{3/2}}{\sqrt{n}}\max_{1\leq i,j\leq d}\sup_{0\leq u\leq v\leq1}\|D_{u}\sigma^{j\cdot}_{v}\|_{2p,\ell_2}\sup_{0\leq s\leq 1}\left\|\Sigma^{ii}_s\right\|_{p}^{1/2}.
\]
Analogously we can prove
\[
\left\|\mathbf{II}_n^{ij}-\sqrt{n}\sum_{h=1}^n\int_{t_{h-1}}^{t_h}\left(\int_{t_{h-1}}^t\varsigma^{i\cdot}_s\cdot dB_s\right)\varsigma^{j\cdot}_t\cdot dB_t\right\|_p
\lesssim \frac{p^{3/2}}{\sqrt{n}}\max_{1\leq i,j\leq d}\sup_{0\leq u\leq v\leq1}\|D_{u}\sigma^{j\cdot}_{v}\|_{2p,\ell_2}\sup_{0\leq s\leq 1}\left\|\Sigma^{ii}_s\right\|_{p}^{1/2}.
\]
Consequently, the desired result follows from Lemma \ref{lemma:double}. 
\end{proof}

Next we establish some properties of $M_n^{ij}$ which are necessary for \tcr{the} application of our main theorem. The first result gives the moment bounds.
\begin{lemma}\label{M-moment}
Under the assumptions of Theorem \ref{thm:rc}, 
there is a universal constant $C>0$ such that
\[
\max_{1\leq i,j\leq d}\|M_n^{ij}\|_p\leq Cp\max_{1\leq i\leq d}\sup_{0\leq t\leq 1}\|\Sigma^{ii}_t\|_p
\]
for all $n\in\mathbb{N}$ and $p\in[2,\infty)$. 
\end{lemma}

\begin{proof}
This is an immediate consequence of Lemmas \ref{lemma:double}--\ref{lemma:BDG} and \ref{sigma-l2}. 
\end{proof}

Second, we prove the Malliavin differentiability of $M_n^{ij}$ and compute its Malliavin derivatives. For this purpose we prove an auxiliary result. Recall that we have $\mathbb{D}_{1,2}(H)\subset\domain(\delta)$ by Proposition 1.3.1 of \cite{Nualart2006}. 
\begin{lemma}\label{heisenberg}
Let $k\in\mathbb{N}$. 
\begin{enumerate}[label={\normalfont(\alph*)}]

\item Suppose that $u\in\mathbb{D}_{k,2}(H)$ satisfies $D_{t_1,\dots,t_j}^{(a_1,\dots,a_j)}u\in\domain(\delta)$ for all $j=1,\dots,k$, $a_1,\dots,a_j\in\{1,\dots,r\}$ and $t_1,\dots,t_j\in[0,1]$ and 
\begin{equation}\label{ass:heisenberg}
\sum_{j=1}^kE\left[\int_{[0,1]^j}\left\|\delta\left(D_{t_1,\dots,t_j}u\right)\right\|_{\ell_2}^2dt_1\cdots dt_j\right]<\infty.
\end{equation} 
Then we have $\delta(u)\in\mathbb{D}_{k,2}$ and
\begin{equation}\label{eq:heisenberg}
D_{t_1,\dots,t_k}\delta(u)=\delta\left(D_{t_1,\dots,t_k}u\right)+k\symm\left(D^{k-1}u\right)(t_1,\dots,t_k)
\end{equation}
for all $t_1,\dots,t_k\in[0,1]$.   

\item If $u\in\mathbb{D}_{k,2}(H)$ is $\mathbf{F}$-adapted, then $\delta(u)\in\mathbb{D}_{k,2}$ and \eqref{eq:heisenberg} holds true for all $t_1,\dots,t_k\in[0,1]$.

\end{enumerate}
\end{lemma}

\begin{proof}
(a) We prove the claim by induction on $k$. When $k=1$, the claim follows from Proposition 1.3.8 of \cite{Nualart2006}. 
Next, supposing that the claim holds true for $k=K\in\mathbb{N}$, we prove the claim for $k=K+1$. 
From \eqref{ass:heisenberg} we have
\begin{equation}\label{fubini:heisenberg}
E\left[\int_0^1\left\|\delta\left(D_t\left(D_{t_1,\dots,t_{K}}u\right)\right)\right\|_{\ell_2}^2dt\right]<\infty
\end{equation}
for all $t_1,\dots,t_{K}\in[0,1]$. 
Moreover, by the assumption of the induction \eqref{eq:heisenberg} holds true for all $t_1,\dots,t_{K}\in[0,1]$. 
Now let us take $t_1,\dots,t_{K}\in[0,1]$ arbitrarily, and set $v:=D_{t_1,\dots,t_K}u$. 
Then, by assumptions and Proposition 1.3.8 of \cite{Nualart2006}, $\delta(v)\in\mathbb{D}_{1,2}$ and $D_t\delta(v)=\delta(D_tv)+v(t)$ for all $t\in[0,1]$. Therefore, by \eqref{eq:heisenberg} we have $D_{t_1,\dots,t_K}\delta(u)\in\mathbb{D}_{1,2}$ and
\begin{align*}
D_t\left(D_{t_1,\dots,t_K}\delta(u)\right)
&=\delta(D_tv)+v(t)+KD_t\symm\left(D^{K-1}u\right)(t_1,\dots,t_K)\\
&=\delta\left(D_{t_1,\dots,t_K,t}u\right)+(K+1)\symm\left(D^{K}u\right)(t_1,\dots,t_K,t)
\end{align*}
for all $t\in[0,1]$. This implies that the claim also holds true for $k=K+1$ and thus completes the proof. 

(b) This claim is an immediate consequence of claim (a) and Propositions 1.2.8 and 1.3.11 of \cite{Nualart2006}.
\end{proof}

We then obtain the following result. 
\begin{lemma}\label{M-deriv}
Under the assumptions of Theorem \ref{thm:rc}, the following statements hold true for any $n\in\mathbb{N}$ and $i,j=1,\dots,d$: 
\begin{enumerate}[label={\normalfont(\alph*)}]

\item $f_n^{ij}\in\mathbb{D}_{2,\infty}(H^{\otimes2})$ and it holds that
\begin{equation}\label{eq:1st-deriv}
D_s^{(a)}f_n^{ij}(u,v)=\frac{1}{2}(D^{(a)}_s\varsigma_u^{i\cdot}\otimes \varsigma_v^{j\cdot}+\varsigma_u^{i\cdot}\otimes D_s^{(a)}\varsigma_v^{j\cdot}
+D_s^{(a)}\varsigma_u^{j\cdot}\otimes \varsigma_v^{i\cdot}+\varsigma_u^{j\cdot}\otimes D_s^{(a)}\varsigma_v^{i\cdot})
\end{equation}
and
\begin{align}
D_{s,t}^{(a,b)}f^{ij}_{n}(u,v)
&=\frac{1}{2}\left(D^{(a,b)}_{s,t}\varsigma^{i\cdot}_{u}\otimes\varsigma^{j\cdot}_{v}+D_s^{(a)}\varsigma^{i\cdot}_{u}\otimes D_t^{(b)}\varsigma^{j\cdot}_{v}+D_t^{(b)}\varsigma^{i\cdot}_{u}\otimes D_s^{(a)}\varsigma^{j\cdot}_{v}+\varsigma^{i\cdot}_{u}\otimes D^{(a,b)}_{s,t}\varsigma^{j\cdot}_{v}\right.\nonumber\\
&\left.+D^{(a,b)}_{s,t}\varsigma^{j\cdot}_{u}\otimes\varsigma^{i\cdot}_{v}+D_s^{(a)}\varsigma^{j\cdot}_{u}\otimes D_t^{(b)}\varsigma^{i\cdot}_{v}+D_t^{(b)}\varsigma^{j\cdot}_{u}\otimes D_s^{(a)}\varsigma^{i\cdot}_{v}+\varsigma^{j\cdot}_{u}\otimes D^{(a,b)}_{s,t}\varsigma^{i\cdot}_{v}\right)\label{eq:2nd-deriv}
\end{align}
for any $a,b=1,\dots,r$ and $s,t,u,v\in[0,1]$. 

\item $u_n^{ij}\in\mathbb{D}_{2,\infty}(H^{\otimes2})$ and $u_n^{ij}(s,\cdot)^{a\cdot},D_s^{(a)}u_n^{ij}(t,\cdot)^{b\cdot}\in\domain(\delta)$, $D_{s,t}^{(a,b)}u_n^{ij}\in\domain(\delta^2)$ for any $a,b=1,\dots,r$ and $s,t\in[0,1]$. 

\item $M_n^{ij}\in\mathbb{D}_{2,\infty}$ and we have
\begin{align*}
D_s^{(a)}M_n^{ij}&=\delta^2(D_s^{(a)}u_n^{ij})+2\delta(u_n^{ij}(s,\cdot)^{a\cdot}),\\
D_{s,t}^{(a,b)}M_n^{ij}&=\delta^2\left(D_{s,t}^{(a,b)}u_n^{ij}\right)
+2\delta\left(D_{s}^{(a)}u_n^{ij}(t,\cdot)^{b\cdot}\right)
+2\delta\left(D_{t}^{(b)}u_n^{ij}(s,\cdot)^{a\cdot}\right)
+2u_n^{ij}(s,t)^{ab}
\end{align*}
for any $a,b=1,\dots,r$ and $s,t\in[0,1]$. 

\end{enumerate}
\end{lemma}

\begin{proof}
Claim (a) follows from Lemma \ref{sigma-deriv} as well as Lemma 15.82 and Theorem 15.83 of \cite{Janson1997}. 
Claim (b) is a consequence of claim (a), Proposition 1.3.11 of \cite{Nualart2006} and Lemmas \ref{lemma:double} and \ref{sigma-deriv}. 

Now we prove claim (c). By Lemma \ref{lemma:double} we can rewrite $M_n^{ij}$ as
\begin{equation}\label{M-deriv:eq1}
M_n^{ij}
=2\sqrt{n}\sum_{h=1}^n\int_{I_h}\delta(f_n^{ij}(\cdot,t)1_{I_h\cap[0,t]})\cdot dB_t.
\end{equation}
From claim (a), $f_n^{ij}(\cdot,t)\in\mathbb{D}_{2,\infty}(H)$ for all $t\in[0,1]$, \tcr{so} Lemma \ref{heisenberg}(b) implies that $\delta(f_n^{ij}(\cdot,t)1_{I_h\cap[0,t]})\in\mathbb{D}_{2,2}(\mathbb{R}^r)$ and
\begin{align*}
D_u^{(a)}\delta(f_n^{ij}(\cdot,t)1_{I_h\cap[0,t]})&=\delta(D_u^{(a)}f_n^{ij}(\cdot,t)1_{I_h\cap[0,t]})+f_n^{ij}(u,t)^{a\cdot}1_{I_h\cap[0,t]}(u),\\
D_{u,v}^{(a,b)}\delta(f_n^{ij}(\cdot,t)1_{I_h\cap[0,t]})&=\delta(D_{u,v}^{(a,b)}f_n^{ij}(\cdot,t)1_{I_h\cap[0,t]})
+D_v^{(b)}f_n^{ij}(u,t)^{a\cdot}1_{I_h\cap[0,t]}(u)
+D_u^{(a)}f_n^{ij}(v,t)^{b\cdot}1_{I_h\cap[0,t]}(v)
\end{align*}
for all $a,b=1,\dots,r$ and $u,v\in[0,1]$. These formulae imply that the process $(\delta(f_n^{ij}(\cdot,t)))_{t\in[0,1]}$ belongs to $\mathbb{D}_{2,2}(H)$, \tcr{so} Lemma \ref{heisenberg}(b) again implies that $\int_{I_h}\delta(f_n^{ij}(\cdot,t)1_{I_h\cap[0,t]})\cdot dB_t\in\mathbb{D}_{2,2}$ and
\begin{align}
&D_u^{(a)}\left(\int_{I_h}\delta(f_n^{ij}(\cdot,t)1_{I_h\cap[0,t]})\cdot dB_t\right)
=\int_{I_h}D_u^{(a)}\delta(f_n^{ij}(\cdot,t)1_{I_h\cap[0,t]})\cdot dB_t+\delta(f_n^{ij}(\cdot,u)^{\cdot a}1_{I_h\cap[0,u]})1_{I_h}(u)\nonumber\\
&=\int_{I_h}\delta(D_u^{(a)}f_n^{ij}(\cdot,t)1_{I_h\cap[0,t]})\cdot dB_t
+\int_{I_h}f_n^{ij}(u,t)^{a\cdot}1_{I_h\cap[0,t]}(u)\cdot dB_t
+\delta(f_n^{ij}(\cdot,u)^{\cdot a}1_{I_h\cap[0,u]})1_{I_h}(u)\nonumber\\
&=\int_{I_h}\left(\int_{t_{h-1}}^tD_u^{(a)}f_n^{ij}(s,t)dB_s\right)\cdot dB_t
+\delta(f_n^{ij}(u,\cdot)^{a \cdot}1_{I_h\times I_h}(u,\cdot))\label{M-deriv:eq2}
\end{align}
and
\begin{align}
&D_{u,v}^{(a,b)}\left(\int_{I_h}\delta(f_n^{ij}(\cdot,t)1_{I_h\cap[0,t]})\cdot dB_t\right)\nonumber\\
&=\int_{I_h}D_{u,v}^{(a,b)}\delta(f_n^{ij}(\cdot,t)1_{I_h\cap[0,t]})\cdot dB_t
+D_v^{(b)}\delta(f_n^{ij}(\cdot,u)^{\cdot a}1_{I_h\cap[0,u]})1_{I_h}(u)
+D_u^{(a)}\delta(f_n^{ij}(\cdot,v)^{\cdot b}1_{I_h\cap[0,v]})1_{I_h}(v)\nonumber\\
&=\int_{I_h}\left\{\delta(D_{u,v}^{(a,b)}f_n^{ij}(\cdot,t)1_{I_h\cap[0,t]})
+D_v^{(b)}f_n^{ij}(u,t)^{a\cdot}1_{I_h\cap[0,t]}(u)
+D_u^{(a)}f_n^{ij}(v,t)^{b\cdot}1_{I_h\cap[0,t]}(v)\right\}\cdot dB_t\nonumber\\
&\quad+\left\{\delta(D_v^{(b)}f_n^{ij}(\cdot,u)^{\cdot a}1_{I_h\cap[0,u]})+f_n^{ij}(v,u)^{ba}1_{I_h\cap[0,u]}(v)\right\}1_{I_h}(u)\nonumber\\
&\quad+\left\{\delta(D_u^{(a)}f_n^{ij}(\cdot,v)^{\cdot b}1_{I_h\cap[0,v]})+f_n^{ij}(u,v)^{ab}1_{I_h\cap[0,v]}(u)\right\}1_{I_h}(v)\nonumber\\
&=\int_{I_h}\left\{\int_{t_{h-1}}^tD_{u,v}^{(a,b)}f_n^{ij}(s,t)dB_s
+D_v^{(b)}f_n^{ij}(u,t)^{a\cdot}1_{I_h\cap[0,t]}(u)
+D_u^{(a)}f_n^{ij}(v,t)^{b\cdot}1_{I_h\cap[0,t]}(v)\right\}\cdot dB_t\nonumber\\
&\quad+\delta(D_v^{(b)}f_n^{ij}(\cdot,u)^{\cdot a}1_{I_h\cap[0,u]})1_{I_h}(u)
+\delta(D_u^{(a)}f_n^{ij}(\cdot,v)^{\cdot b}1_{I_h\cap[0,v]})1_{I_h}(v)
+f_n^{ij}(u,v)^{ab}1_{I_h\times I_h}(u,v)\nonumber\\
&=\int_{I_h}\left\{\int_{t_{h-1}}^tD_{u,v}^{(a,b)}f_n^{ij}(s,t)dB_s\right\}\cdot dB_t
+\delta(D_v^{(b)}f_n^{ij}(u,\cdot)^{a\cdot}1_{I_h\times I_h}(u,\cdot))\nonumber\\
&\quad+\delta(D_u^{(a)}f_n^{ij}(v,\cdot)^{b\cdot}1_{I_h\times I_h}(v,\cdot))
+f_n^{ij}(u,v)^{ab}1_{I_h\times I_h}(u,v)\label{M-deriv:eq3}
\end{align}
for all $a,b=1,\dots,r$ and $u,v\in[0,1]$. Now, noting that formulae \eqref{eq:1st-deriv}--\eqref{eq:2nd-deriv} can be rewritten as
\[
D_s^{(a)}f_n^{ij}(u,v)=\symm\left(D^{(a)}_s\varsigma^{i\cdot}\otimes \varsigma^{j\cdot}\right)(u,v)
+\symm\left(\varsigma^{i\cdot}\otimes D^{(a)}_s\varsigma^{j\cdot}\right)(u,v)
\]
and 
\begin{align*}
D_{s,t}^{(a,b)}f^{ij}_{n}(u,v)
&=\symm\left(D^{(a,b)}_{s,t}\varsigma^{i\cdot}\otimes\varsigma^{j\cdot}\right)(u,v)
+\symm\left(D_s^{(a)}\varsigma^{i\cdot}\otimes D_t^{(b)}\varsigma^{j\cdot}\right)(u,v)\\
&\quad+\symm\left(D_t^{(b)}\varsigma^{i\cdot}\otimes D_s^{(a)}\varsigma^{j\cdot}\right)(u,v)
+\symm\left(\varsigma^{i\cdot}_{u}\otimes D^{(a,b)}_{s,t}\varsigma^{j\cdot}_{v}\right)(u,v),
\end{align*}
by Lemma \ref{lemma:double} we obtain
\tcr{
\begin{align}
2\int_{I_h}\left(\int_{t_{h-1}}^tD_u^{(a)}f_n^{ij}(s,t)dB_s\right)\cdot dB_t
&=\delta^2\left(D_u^{(a)}f_n^{ij}\right),\label{M-deriv:eq4}\\
2\int_{I_h}\left\{\int_{t_{h-1}}^tD_{u,v}^{(a,b)}f_n^{ij}(s,t)dB_s\right\}\cdot dB_t
&=\delta^2\left(D_{u,v}^{(a,b)}f^{ij}_{n}\right).\label{M-deriv:eq5}
\end{align}
}
Now claim (c) follows from \tcr{\eqref{M-deriv:eq1}--\eqref{M-deriv:eq5}} and the assumptions of the lemma. 
\end{proof}

Third, we prove the Malliavin differentiability of $\Sigma$ and $\mathfrak{C}_n$ as well as establish moment estimates for their Malliavin derivatives. 
\begin{lemma}\label{S-deriv}
Under the assumptions of Theorem \ref{thm:rc}, 
for any $i,j=1,\dots,d$ and $t\in[0,1]$, $\Sigma_t^{ij}\in\mathbb{D}_{2,\infty}$ and
\begin{align*}
\|D_u\Sigma^{ij}_t\|_{p,\ell_2}
&\leq2\max_{1\leq i,j\leq d}\left\|\Sigma^{ii}_t\right\|_{p}^{1/2}\|D_u\sigma^{j\cdot}_t\|_{2p,\ell_2},\\
\|D_{u,v}\Sigma^{ij}_t\|_{p,\ell_2}
&\leq2\left(\max_{1\leq i,j\leq d}\left\|\Sigma^{ii}_t\right\|_{p}^{1/2}\|D_{u,v}\sigma^{j\cdot}_t\|_{2p,\ell_2}
+\max_{1\leq i,j\leq d}\|D_u\sigma^{i\cdot}_t\|_{2p,\ell_2}\|D_v\sigma^{j\cdot}_t\|_{2p,\ell_2}\right)
\end{align*}
for any $p\in[1,\infty)$ and $u,v\in[0,1]$. 
\end{lemma}

\begin{proof}
Since $\Sigma_t^{ij}=\sum_{a=1}^r\sigma^{ia}_t\sigma^{ja}_t$, Theorem 15.78 of \cite{Janson1997} implies that $\Sigma_t^{ij}\in\mathbb{D}_{2,\infty}$ for all $t\in[0,1]$ and
\[
D\Sigma^{ij}_t=\sum_{a=1}^r\left(\sigma^{ja}_tD\sigma^{ia}_t+\sigma^{ia}_tD\sigma^{ja}_t\right)
\]
and
\begin{align*}
D^2\Sigma^{ij}_t=\sum_{a=1}^r\left(\sigma^{ja}_tD^2\sigma^{ia}_t
+D\sigma^{ja}_t\otimes D\sigma^{ia}_t
+D\sigma^{ia}_t\otimes D\sigma^{ja}_t
+\sigma^{ia}_tD^2\sigma^{ja}_t\right).
\end{align*}
In particular, we have
\begin{align*}
\|D_u\Sigma^{ij}_t\|_{\ell_2}
&\leq\sum_{a=1}^r\left(|\sigma^{ja}_t|\|D\sigma^{ia}_t\|_{\ell_2}+|\sigma^{ia}_t|\|D\sigma^{ja}_t\|_{\ell_2}\right)
\leq\sqrt{\Sigma^{jj}_t}\|D\sigma^{i\cdot}_t\|_{\ell_2}+\sqrt{\Sigma^{ii}_t}\|D\sigma^{j\cdot}_t\|_{\ell_2}
\end{align*}
and
\begin{align*}
\|D_{u,v}\Sigma^{ij}_t\|_{\ell_2}
&\leq\sum_{a=1}^r\left(|\sigma^{ja}_t|\|D_{u,v}\sigma^{ia}_t\|_{\ell_2}
+\|D_u\sigma^{ja}_t\|_{\ell_2}\|D_v\sigma^{ia}_t\|_{\ell_2}
+\|D_u\sigma^{ia}_t\|_{\ell_2}\|D_v\sigma^{ja}_t\|_{\ell_2}
+|\sigma^{ia}_t|\|D_{u,v}\sigma^{ja}_t\|_{\ell_2}\right)\\
&\leq\sqrt{\Sigma^{jj}_t}\|D_{u,v}\sigma^{i\cdot}_t\|_{\ell_2}
+\|D_u\sigma^{j\cdot}_t\|_{\ell_2}\|D_v\sigma^{i\cdot}_t\|_{\ell_2}
+\|D_u\sigma^{i\cdot}_t\|_{\ell_2}\|D_v\sigma^{j\cdot}_t\|_{\ell_2}
+\sqrt{\Sigma^{ii}_t}\|D_{u,v}\sigma^{j\cdot}_t\|_{\ell_2}
\end{align*}
by the triangular and Schwarz inequalities. Hence we complete the proof by the H\"older inequality.
\end{proof}

\begin{lemma}\label{C-deriv}
Under the assumptions of Theorem \ref{thm:rc}, 
$\mathfrak{C}_n\in\mathbb{D}_{2,\infty}(\mathbb{R}^{d^2\times d^2})$ and 
\begin{align*}
&\|D_u\mathfrak{C}_n^{(i-1)d+j,(k-1)d+l}\|_{p,\ell_2}
\leq 8\max_{1\leq i,j\leq d}\sup_{0\leq s,t\leq 1}\|\Sigma^{ii}_s\|_{2p}^{3/2}\|D_u\sigma^{j\cdot}_t\|_{4p,\ell_2},\\
&\|D_{u,v}\mathfrak{C}_n^{(i-1)d+j,(k-1)d+l}\|_{p,\ell_2}\\
&\leq 8\max_{1\leq i,j,k\leq d}\sup_{0\leq s,t\leq 1}\|\Sigma^{ii}_s\|_{2p}^{3/2}\|D_{u,v}\sigma^{j\cdot}_t\|_{4p,\ell_2}
+24\max_{1\leq i,j,k\leq d}\sup_{0\leq s,t,\tau\leq 1}\|\Sigma^{ii}_s\|_{2p}\|D_{u}\sigma^{j\cdot}_t\|_{4p,\ell_2}\|D_{v}\sigma^{k\cdot}_\tau\|_{4p,\ell_2}
\end{align*}
for any $p\in[1,\infty)$, $n\in\mathbb{N}$, $i,j,k,l=1,\dots,d$ and $u,v\in[0,1]$. 
\end{lemma}

\begin{proof}
By Remark 15.87 of \cite{Janson1997}, $\int_{I_h}\Sigma^{ij}_tdt\in\mathbb{D}_{2,\infty}$ and
\[
D\left(\int_{I_h}\Sigma^{ij}_tdt\right)=\int_{I_h}D\Sigma^{ij}_tdt,\qquad
D^2\left(\int_{I_h}\Sigma^{ij}_tdt\right)=\int_{I_h}D^2\Sigma^{ij}_tdt
\]
for any $i,j=1,\dots,d$ and $h=1,\dots,n$. Therefore, by Theorem 15.78 of \cite{Janson1997}, the Schwarz inequality and Proposition \ref{minkowski} we obtain $\mathfrak{C}_n^{(i-1)d+j,(k-1)d+l}\in\mathbb{D}_{2,\infty}$ and
\begin{align*}
\|D_u\mathfrak{C}_n^{(i-1)d+j,(k-1)d+l}\|_{p,\ell_2}
&\leq 4\max_{1\leq i,j,k,l\leq d}\sup_{0\leq s,t\leq 1}\|\Sigma^{ik}_s\|_{2p}\|D_u\Sigma^{kl}_t\|_{2p,\ell_2},\\
\|D_{u,v}\mathfrak{C}_n^{(i-1)d+j,(k-1)d+l}\|_{p,\ell_2}
&\leq 4\max_{1\leq i,j,k,l\leq d}\sup_{0\leq s,t\leq 1}\left(\|\Sigma^{ik}_s\|_{2p}\|D^2_{u,v}\Sigma^{kl}_t\|_{2p,\ell_2}
+\|D_{u}\Sigma^{ik}_s\|_{2p,\ell_2}\|D_{v}\Sigma^{jl}_t\|_{2p,\ell_2}
\right)
\end{align*}
for any $p\geq1$, $i,j,k,l=1,\dots,d$ and $u,v\in[0,1]$. Now the desired result follows from the Schwarz inequality and Lemma \ref{S-deriv}. 
\end{proof}

Now we proceed to checking the conditions of Theorem \ref{thm:main} in the current setting. 
\begin{lemma}\label{quasi-torsion}
Under the assumptions of Theorem \ref{thm:rc}, there is a universal constant $C>0$ such that
\begin{multline*}
\max_{1\leq i,j,k,l\leq d}\left\|\left\langle D^2M_n^{ij},u_n^{kl}\right\rangle_{H^{\otimes2}}-\mathfrak{C}_n^{(i-1)d+j,(k-1)d+l}\right\|_p\\
\leq C\frac{p}{\sqrt{n}}
\left(
\max_{1\leq i,j\leq d}\sup_{0\leq s\leq1}\left\|\Sigma^{jj}_s\right\|_{2p}^{3/2}\sup_{0\leq s,t,u\leq1}\left\|D_{s,t}\sigma^{i\cdot}_{u}\right\|_{4p,\ell_2}
+\max_{1\leq i,k\leq d}\sup_{0\leq s,u\leq1}\left\|\Sigma^{kk}_u\right\|_{2p}\left\|D_s\sigma^{i\cdot}_{u}\right\|_{4p,\ell_2}^2
\right)\\
+C\sqrt{\frac{p}{n}}\max_{1\leq i,j\leq d}\sup_{0\leq s,t,u\leq1}\|D_s\sigma_u^{i\cdot}\|_{4p,\ell_2}\|\Sigma_t^{jj}\|_{2p}^{3/2}
\end{multline*}
for any $n\in\mathbb{N}$ and $p\in[2,\infty)$. 
\end{lemma}

\begin{proof}
By Lemma \ref{M-deriv} the desired result follows once we verify the following statements for all $p\in[2,\infty)$ (note \eqref{eq:symm}): 
\begin{align}
&\left\|\langle\delta^2(D^2u_n^{ij}),u_n^{kl}\rangle_{H^{\otimes2}}\right\|_p\nonumber\\
&\lesssim\frac{p}{\sqrt{n}}
\left(
\max_{1\leq i,j\leq d}\sup_{0\leq s\leq1}\left\|\Sigma^{jj}_s\right\|_{2p}^{3/2}\sup_{0\leq s,t,u\leq1}\left\|D_{s,t}\sigma^{i\cdot}_{u}\right\|_{4p,\ell_2}
+\max_{1\leq i,k\leq d}\sup_{0\leq s,u\leq1}\left\|\Sigma^{kk}_u\right\|_{2p}\left\|D_s\sigma^{i\cdot}_{u}\right\|_{4p,\ell_2}^2
\right)\label{eq:1st-term},\\
&\left\|\langle\delta(Du_n^{ij}),u_n^{kl}\rangle_{H^{\otimes2}}\right\|_p
\lesssim \sqrt{\frac{p}{n}}\max_{1\leq i,j\leq d}\sup_{0\leq s,t,u\leq1}\|D_s\sigma_u^{i\cdot}\|_{4p,\ell_2}\|\Sigma_t^{jj}\|_{2p}^{3/2},\label{eq:2nd-term}\\
&\left\|2\langle u_n^{ij},u_n^{kl}\rangle_{H^{\otimes2}}-\mathfrak{C}_n^{(i-1)d+j,(k-1)d+l}\right\|_p
\lesssim\sqrt{\frac{p}{n}}\max_{1\leq i\leq d}\sup_{0\leq u\leq v\leq1}\|D_u\sigma_v^{i\cdot}\|_{4p,\ell_2}\max_{1\leq j\leq d}\sup_{0\leq t\leq 1}\|\Sigma_t^{jj}\|_{2p}^{3/2}.\label{eq:3rd-term}
\end{align}

We first verify \eqref{eq:1st-term}. We can rewrite $\langle\delta^2(D^2u_n^{ij}),u_n^{kl}\rangle_{H^{\otimes2}}$ as
\begin{align*}
\langle\delta^2(D^2u_n^{ij}),u_n^{kl}\rangle_{H^{\otimes2}}
&=n\sum_{h,h'=1}^n\sum_{a,b=1}^r\int_{I_h\times I_h}\delta^2(D^{(a,b)}_{s,t}f^{ij}_{n}1_{I_{h'}}\times 1_{I_{h'}})f^{kl}_n(s,t)^{ab}dsdt\\
&=n\sum_{h,h'=1}^n\sum_{a,b=1}^r\int_{I_h\times I_h}\delta^2\left(D^{(a,b)}_{s,t}f^{ij}_{n}1_{I_{h'}}\times 1_{I_{h'}}\right)\varsigma^{ka}_s\varsigma^{lb}_tdsdt,
\end{align*}
\tcr{
where the last identity holds true because $\langle\varphi,\symm(\psi)\rangle_{H^{\otimes2}}=\langle\varphi,\psi\rangle_{H^{\otimes2}}$ for any $\varphi,\psi\in H^{\otimes2}$ if $\varphi$ is symmetric. 
Then, noting that $f_n^{ij}(u,v)$ is $\mathcal{F}_{t_{h'-1}}$-measurable when $u,v\in I_{h'}$ by construction, Corollary 1.2.1 of \cite{Nualart2006} yields
\begin{align*}
\langle\delta^2(D^2u_n^{ij}),u_n^{kl}\rangle_{H^{\otimes2}}
&=n\sum_{h=1}^n\sum_{h':h'>h}\sum_{a,b=1}^r\int_{I_h\times I_h}\delta^2(D^{(a,b)}_{s,t}f^{ij}_{n}1_{I_{h'}}\times 1_{I_{h'}})\varsigma^{ka}_s\varsigma^{lb}_tdsdt.
\end{align*}
Moreover, since $\varsigma_u$ is $\mathcal{F}_{t_{h-1}}$-measurable when $u\in I_{h}$, using Lemma \ref{lemma:double} and Exercise 2.30 in Chapter 3 of \cite{KS1998} repeatedly, we obtain
\begin{align*}
\langle\delta^2(D^2u_n^{ij}),u_n^{kl}\rangle_{H^{\otimes2}}
&=&=n\sum_{h=1}^n\int_{I_h\times I_h}\sum_{h':h'>h}\delta^2\left(\sum_{a,b=1}^r\varsigma^{ka}_s\varsigma^{lb}_tD^{(a,b)}_{s,t}f^{ij}_{n}1_{I_{h'}}\times 1_{I_{h'}}\right)dsdt.
\end{align*}
}%
\tcr{H}ence Proposition \ref{minkowski} yields
\[
\left\|\langle\delta^2(D^2u_n^{ij}),u_n^{kl}\rangle_{H^{\otimes2}}\right\|_p
\leq n\sum_{h=1}^n\int_{I_h\times I_h}\left\|\sum_{h':h'>h}\delta^2\left(\sum_{a,b=1}^r\varsigma^{ka}_s\varsigma^{lb}_tD^{(a,b)}_{s,t}f^{ij}_{n}1_{I_{h'}}\times 1_{I_{h'}}\right)\right\|_pdsdt.
\]
Now, from \eqref{eq:2nd-deriv} we infer that
\begin{align*}
\sum_{a,b=1}^r\varsigma^{ka}_{s}\varsigma^{lb}_{t}D^{(a,b)}_{s,t}f^{ij}_{n}(u,v)
&=\frac{1}{2}\left\{\left(\sum_{a,b=1}^r\varsigma^{ka}_{s}\varsigma^{lb}_{t}D^{(a,b)}_{s,t}\varsigma^{i\cdot}_{u}\right)\otimes\varsigma^{j\cdot}_{v}
+\left(\sum_{a=1}^r\varsigma^{ka}_{s}D_s^{(a)}\varsigma^{i\cdot}_{u}\right)\otimes \left(\sum_{b=1}^r\varsigma^{lb}_{t}D_t^{(b)}\varsigma^{j\cdot}_{v}\right)\right.\\
&+\left(\sum_{b=1}^r\varsigma^{lb}_{t}D_t^{(b)}\varsigma^{i\cdot}_{u}\right)\otimes \left(\sum_{a=1}^r\varsigma^{ka}_{s}D_s^{(a)}\varsigma^{j\cdot}_{v}\right)
+\varsigma^{i\cdot}_{u}\otimes \left(\sum_{a,b=1}^r\varsigma^{ka}_{s}\varsigma^{lb}_{t}D^{(a,b)}_{s,t}\varsigma^{j\cdot}_{v}\right)\\
&+\left(\sum_{a,b=1}^r\varsigma^{ka}_{s}\varsigma^{lb}_{t}D^{(a,b)}_{s,t}\varsigma^{j\cdot}_{u}\right)\otimes\varsigma^{i\cdot}_{v}+\left(\sum_{a=1}^r\varsigma^{ka}_{s}D_s^{(a)}\varsigma^{j\cdot}_{u}\right)\otimes \left(\sum_{b=1}^r\varsigma^{lb}_{t}D_t^{(b)}\varsigma^{i\cdot}_{v}\right)\\
&\left.+\left(\sum_{b=1}^r\varsigma^{lb}_{t}D_t^{(b)}\varsigma^{j\cdot}_{u}\right)\otimes \left(\sum_{a=1}^r\varsigma^{ka}_{s}D_s^{(a)}\varsigma^{i\cdot}_{v}\right)+\varsigma^{j\cdot}_{u}\otimes \left(\sum_{a,b=1}^r\varsigma^{ka}_{s}\varsigma^{lb}_{t}D^{(a,b)}_{s,t}\varsigma^{i\cdot}_{v}\right)
\right\}.
\end{align*}
Hence Lemmas \ref{lemma:BDG} and \ref{sigma-l2}--\ref{sigma-deriv} yield
\begin{align*}
&\sup_{0\leq s,t\leq 1}\left\|\sum_{h':h'>h}\delta^2\left(\sum_{a,b=1}^r\varsigma^{ka}_s\varsigma^{lb}_tD^{(a,b)}_{s,t}f^{ij}_{n}1_{I_{h'}}\times 1_{I_{h'}}\right)\right\|_p\\
&\lesssim\frac{p}{\sqrt{n}}\left(\sup_{0\leq s,t,u,v\leq1}\left\|\sum_{a,b=1}^r\varsigma^{ka}_{s}\varsigma^{lb}_{t}D^{(a,b)}_{s,t}\varsigma^{i\cdot}_{u}\right\|_{\frac{4}{3}p,\ell_2}\left\|\varsigma^{j\cdot}_v\right\|_{4p,\ell_2}\right.
+\sup_{0\leq s,t,u,v\leq1}\left\|\sum_{a=1}^r\varsigma^{ka}_{s}D_s^{(a)}\varsigma^{i\cdot}_{u}\right\|_{2p,\ell_2}\left\|\sum_{b=1}^r\varsigma^{lb}_{t}D_t^{(b)}\varsigma^{j\cdot}_{v}\right\|_{2p,\ell_2}\\
&\left.\qquad\qquad+\sup_{0\leq s,t,u,v\leq1}\left\|\varsigma^{i\cdot}_{u}\right\|_{4p,\ell_2}\left\|\sum_{a,b=1}^r\varsigma^{ka}_{s}\varsigma^{lb}_{t}D^{(a,b)}_{s,t}\varsigma^{j\cdot}_{v}\right\|_{\frac{4}{3}p,\ell_2}\right)\\
&\lesssim\frac{p}{\sqrt{n}}
\left(
\max_{1\leq i,j\leq d}\sup_{0\leq s\leq1}\left\|\Sigma^{jj}_s\right\|_{2p}^{3/2}\sup_{0\leq s,t,u\leq1}\left\|D_{s,t}\sigma^{i\cdot}_{u}\right\|_{4p,\ell_2}
+\max_{1\leq i,k\leq d}\sup_{0\leq s,u\leq1}\left\|\Sigma^{kk}_u\right\|_{2p}\left\|D_s\sigma^{i\cdot}_{u}\right\|_{4p,\ell_2}^2
\right).
\end{align*}
Therefore, we obtain \eqref{eq:1st-term}. 

Next we verify \eqref{eq:2nd-term}. \tcr{W}e have
\begin{align*}
\langle\delta(Du_n^{ij}),u_n^{kl}\rangle_{H^{\otimes2}}
&=n\sum_{h=1}^n\sum_{a=1}^r\int_{I_h\times I_h}\delta(D_s^{(a)}f_n^{ij}(\cdot,t)1_{I_h}(\cdot))\cdot f_n^{kl}(s,t)^{a\cdot}dsdt.
\end{align*}
\tcr{Since $f_n^{kl}(s,t)$ is $\mathcal{F}_{t_{h-1}}$-measurable when $s,t\in I_{h}$, by Proposition 1.3.11 of \cite{Nualart2006} and Exercise 2.30 in Chapter 3 of \cite{KS1998} we obtain} 
\begin{align*}
\langle\delta(Du_n^{ij}),u_n^{kl}\rangle_{H^{\otimes2}}
&=n\sum_{h=1}^n\int_{I_h\times I_h}\left(\int_{I_h}\left(\sum_{a,b=1}^rf_n^{kl}(s,t)^{ab}D_s^{(a)}f_n^{ij}(u,t)^{b\cdot}\right)\cdot dB_u\right)dsdt.
\end{align*}
Therefore, Propositions \ref{minkowski}--\ref{sharp-BDG} and the Schwarz inequality yield
\begin{align*}
\left\|\langle\delta(Du_n^{ij}),u_n^{kl}\rangle_{H^{\otimes2}}\right\|_p
&\lesssim\sqrt{p}n\sum_{h=1}^n\int_{I_h\times I_h}\left\|\sqrt{\int_{I_h}\sum_{c=1}^r\left(\sum_{a,b=1}^rf_n^{kl}(s,t)^{ab}D_s^{(a)}f_n^{ij}(u,t)^{bc}\right)^2du}\right\|_pdsdt\\
&\leq\sqrt{p}n\sum_{h=1}^n\int_{I_h\times I_h}\left\|\int_{I_h}\|f_n^{kl}(s,t)\|_{\ell_2}^2\|D_sf_n^{ij}(u,t)\|_{\ell_2}^2du\right\|_{p/2}^{1/2}dsdt\\
&\leq\sqrt{p}n\sum_{h=1}^n\int_{I_h\times I_h}\sqrt{\int_{I_h}\|f_n^{kl}(s,t)\|_{2p,\ell_2}^2\|D_sf_n^{ij}(u,t)\|_{2p,\ell_2}^2du}dsdt\\
&\leq\sqrt{\frac{p}{n}}\sup_{0\leq s,t,u\leq1}\|f_n^{kl}(s,t)\|_{2p,\ell_2}\|D_sf_n^{ij}(u,t)\|_{2p,\ell_2}.
\end{align*}
Since we have 
$
\|f_n^{kl}(s,t)\|_{2p,\ell_2}
\leq\sqrt{\|\Sigma^{kk}_s\|_{2p}\|\Sigma^{ll}_t\|_{2p}}
$ 
and 
\[
\|D_sf_n^{ij}(u,t)\|_{2p,\ell_2}
\leq2\max_{1\leq i,j\leq d}\sup_{0\leq s,t,u\leq1}\|D_s\sigma_u^{i\cdot}\|_{4p,\ell_2}\|\Sigma_t^{jj}\|_{2p}^{1/2}
\] 
by the Schwarz inequality and Lemmas \ref{sigma-l2}--\ref{sigma-deriv}, we obtain \eqref{eq:2nd-term}. 

Finally we verify \eqref{eq:3rd-term}. We can rewrite $2\langle u_n^{ij},u_n^{kl}\rangle_{H^{\otimes2}}$ as 
\begin{align*}
2\langle u_n^{ij},u_n^{kl}\rangle_{H^{\otimes2}}
&=n\sum_{h=1}^n\sum_{a,b=1}^r\int_{I_h\times I_h}\left(\varsigma_s^{ia}\varsigma_t^{jb}\varsigma_s^{ka}\varsigma_t^{lb}+\varsigma_s^{ia}\varsigma_t^{jb}\varsigma_t^{kb}\varsigma_s^{la}\right)dsdt\\
&=n\sum_{h=1}^n\left\{\left(\int_{I_h}\varsigma_s^{i\cdot}\cdot\varsigma_s^{k\cdot}ds\right)\left(\int_{I_h}\varsigma_s^{j\cdot}\cdot\varsigma_s^{l\cdot}ds\right)
+\left(\int_{I_h}\varsigma_s^{i\cdot}\cdot\varsigma_s^{l\cdot}ds\right)\left(\int_{I_h}\varsigma_s^{j\cdot}\cdot\varsigma_s^{k\cdot}ds\right)
\right\}.
\end{align*}
Note that we have
\begin{align*}
&\left\|n\sum_{h=1}^n\left(\int_{I_h}(\varsigma_s^{i\cdot}-\sigma^{i\cdot}_s)\cdot\varsigma_s^{k\cdot}ds\right)\left(\int_{I_h}\varsigma_s^{j\cdot}\cdot\varsigma_s^{l\cdot}ds\right)\right\|_p\\
&\leq n\sum_{h=1}^n\left(\int_{I_h}\|\varsigma_s^{i\cdot}-\sigma^{i\cdot}_s\|_{4p,\ell_2}\|\varsigma_s^{k\cdot}\|_{4p,\ell_2}ds\right)\left(\int_{I_h}\|\varsigma_s^{j\cdot}\|_{4p,\ell_2}\|\varsigma_s^{l\cdot}\|_{4p,\ell_2}ds\right)\\
&\lesssim\sqrt{\frac{p}{n}}\sup_{0\leq u\leq v\leq1}\|D_u\sigma_v^{i\cdot}\|_{4p,\ell_2}\max_{1\leq j\leq d}\sup_{0\leq t\leq 1}\|\Sigma_t^{jj}\|_{2p}^{3/2}
\end{align*}
for any $i,j,k,l=1,\dots,d$ by Proposition \ref{minkowski} and Lemmas \ref{ocone} and \ref{sigma-l2}. Therefore, we obtain \eqref{eq:3rd-term} and complete the proof of the lemma. 
\end{proof}


\begin{lemma}\label{lemma:2nd-deriv}
Under the assumptions of Theorem \ref{thm:rc}, 
we have
\[
\left\|\langle D^2F,u^n_{ij}\rangle_{H^{\otimes2}}\right\|_p
\leq\frac{1}{\sqrt{n}}\sup_{0\leq s,t\leq 1}\left\|D_{s,t}F\right\|_{2p,\ell_2}\max_{1\leq i\leq d}\sup_{0\leq s\leq 1}\left\|\Sigma_s^{ii}\right\|_{2p}
\]
for any $p\in[1,\infty)$, $n\in\mathbb{N}$ and $F\in\mathbb{D}_{2,\infty}$.
\end{lemma}

\begin{proof}
Since $D^2F$ is symmetric, \eqref{eq:symm} yields
\begin{align*}
\langle D^2F,u^n_{ij}\rangle_{H^{\otimes2}}
&=\sqrt{n}\sum_{h=1}^n\langle D^2F,\varsigma_s^{i\cdot}\otimes \varsigma_t^{j\cdot}1_{I_h\times I_h}\rangle_{H^{\otimes2}}
=\sqrt{n}\sum_{h=1}^n\int_{I_h\times I_h}D_{s,t}F\cdot\varsigma_s^{i\cdot}\otimes \varsigma_t^{j\cdot}dsdt.
\end{align*}
Therefore, by Proposition \ref{minkowski}, the Schwarz inequality and Lemma \ref{sigma-l2} we have
\begin{align*}
\left\|\langle D^2F,u^n_{ij}\rangle_{H^{\otimes2}}\right\|_p
&\leq\sqrt{n}\sum_{h=1}^n\int_{I_h\times I_h}\left\|D_{s,t}F\cdot\varsigma_s^{i\cdot}\otimes \varsigma_t^{j\cdot}\right\|_pdsdt\\
&\leq\sqrt{n}\sum_{h=1}^n\int_{I_h\times I_h}\left\|D_{s,t}F\right\|_{2p,\ell_2}\left\|\varsigma_s^{i\cdot}\right\|_{4p,\ell_2}\left\|\varsigma_t^{j\cdot}\right\|_{4p,\ell_2}dsdt\\
&=\frac{1}{\sqrt{n}}\sup_{0\leq s,t\leq 1}\left\|D_{s,t}F\right\|_{2p,\ell_2}\max_{1\leq i\leq d}\sup_{0\leq s\leq 1}\left\|\Sigma_s^{ii}\right\|_{2p}.
\end{align*}
This completes the proof. 
\end{proof}

\begin{lemma}\label{lemma:DM}
Under the assumptions of Theorem \ref{thm:rc}, 
there is a universal constant $C>0$ such that
\begin{align*}
\max_{1\leq i,j,k\leq d}\|D_sM_n^{ij}\cdot\varsigma_s^{k\cdot}\|_p
\leq C\left(p\max_{1\leq i\leq d}\sup_{0\leq s,t\leq 1}\left\|D_s\sigma^{i\cdot}_t\right\|_{3p,\ell_2}\max_{1\leq j\leq d}\sup_{0\leq t\leq 1}\|\Sigma^{jj}_t\|_{\frac{3}{2}p}
+\sqrt{p}\max_{1\leq i\leq d}\sup_{0\leq t\leq 1}\left\|\Sigma^{ii}_t\right\|_{\frac{3}{2}p}^{3/2}
\right)
\end{align*}
for any $p\in[2,\infty)$, $n\in\mathbb{N}$ and $s\in(0,1]$.
\end{lemma}

\begin{proof}
By Lemma \ref{M-deriv} the desired result follows once we verify the following statements for all $p\in[2,\infty)$: 
\begin{align}
\|\delta^2(D_su_n^{ij})\cdot\varsigma_s^{k\cdot}\|_p
&\lesssim p\max_{1\leq i\leq d}\sup_{0\leq s,t\leq 1}\left\|D_s\sigma^{i\cdot}_t\right\|_{3p,\ell_2}\max_{1\leq j\leq d}\sup_{0\leq t\leq 1}\|\Sigma^{jj}_t\|_{\frac{3}{2}p},\label{eq:1st-DM},\\
\left\|\delta(u_n^{ij}(s,\cdot))\cdot\varsigma^{k\cdot}_s\right\|_p
&\lesssim \sqrt{p}\max_{1\leq i\leq d}\sup_{0\leq t\leq 1}\left\|\Sigma^{ii}_t\right\|_{\frac{3}{2}p}^{3/2}.\label{eq:2nd-DM}
\end{align}
Let $h$ be the unique integer such that $s\in I_h$. First we verify \eqref{eq:1st-DM}. 
\tcr{Since $f_n^{ij}(u,v)$ is $\mathcal{F}_{t_{h'-1}}$-measurable when $u,v\in I_{h'}$, by Corollary 1.2.1 of \cite{Nualart2006} we have}
\begin{align*}
\delta^2(D^{(a)}_su_n^{ij})
=\sqrt{n}\sum_{h'=1}^n\delta^2(D^{(a)}_sf_n^{ij}1_{I_{h'}\times I_{h'}})
=\sqrt{n}\sum_{h':h'>h}\delta^2(D^{(a)}_sf_n^{ij}1_{I_{h'}\times I_{h'}})
\end{align*}
for any $a=1,\dots,r$. \tcr{H}ence Lemmas \ref{lemma:double} and \ref{sigma-deriv} \tcr{as well as Exercise 2.30 in Chapter 3 of \cite{KS1998}} yield 
\begin{align*}
\delta^2(D_su_n^{ij})\cdot\varsigma_s^{k\cdot}
=\sqrt{n}\sum_{h':h'>h}\delta^2\left(\sum_{a=1}^r\varsigma_s^{ka}D^{(a)}_sf_n^{ij}1_{I_{h'}\times I_{h'}}\right).
\end{align*}
Since \tcr{\eqref{eq:1st-deriv} implies that}
\begin{align*}
&\sum_{a=1}^r\varsigma_s^{ka}D^{(a)}_sf_n^{ij}\\
&=\frac{1}{2}\left\{\left(
\sum_{a=1}^r\varsigma_s^{ka}D^{(a)}_s\varsigma^{i\cdot}\right)\otimes\varsigma^{j\cdot}+\varsigma^{i\cdot}\otimes \left(\sum_{a=1}^r\varsigma_s^{ka}D^{(a)}_s\varsigma^{j\cdot}\right)
+\left(\sum_{a=1}^r\varsigma_s^{ka}D^{(a)}_s\varsigma^{j\cdot}\right)\otimes\varsigma^{i\cdot}+\varsigma^{j\cdot}\otimes \left(\sum_{a=1}^r\varsigma_s^{ka}D^{(a)}_s\varsigma^{i\cdot}\right)
\right\},
\end{align*}
Lemmas \ref{lemma:double}--\ref{lemma:BDG} and \ref{sigma-l2}--\ref{sigma-deriv} imply that
\begin{align*}
\|\delta^2(D_su_n^{ij})\cdot\varsigma_s^{k\cdot}\|_p
&\lesssim p\max_{1\leq i,k\leq d}\sup_{0\leq t\leq 1}\left\|\sum_{a=1}^r\varsigma_s^{ka}D^{(a)}_s\varsigma^{i\cdot}_t\right\|_{\frac{3}{2}p,\ell_2}\max_{1\leq j\leq d}\sup_{0\leq t\leq 1}\|\varsigma^{j\cdot}_t\|_{3p,\ell_2}\\
&\leq p\max_{1\leq i,k\leq d}\sup_{0\leq t\leq 1}\|\Sigma_s^{kk}\|_{\frac{3}{2}p}^{1/2}\left\|D_s\sigma^{i\cdot}_t\right\|_{3p,\ell_2}\max_{1\leq j\leq d}\sup_{0\leq t\leq 1}\|\Sigma^{jj}_t\|_{\frac{3}{2}p,\ell_2}^{1/2}\\
&\leq p\max_{1\leq i\leq d}\sup_{0\leq s,t\leq 1}\left\|D_s\sigma^{i\cdot}_t\right\|_{3p,\ell_2}\max_{1\leq j\leq d}\sup_{0\leq t\leq 1}\|\Sigma^{jj}_t\|_{\frac{3}{2}p}.
\end{align*}

Next we verify \eqref{eq:2nd-DM}. Proposition 1.3.11 of \cite{Nualart2006} yields
\begin{align*}
\delta(u_n^{ij}(s,\cdot))
&=\frac{\sqrt{n}}{2}\sum_{h'=1}^n1_{I_{h'}}(s)\int_{I_{h'}}\left(\varsigma^{i\cdot}_s\otimes \varsigma^{j\cdot}_t+\varsigma^{j\cdot}_s\otimes \varsigma^{i\cdot}_t\right)dB_t
=\frac{\sqrt{n}}{2}\int_{I_{h}}\left(\varsigma^{i\cdot}_s\otimes \varsigma^{j\cdot}_t+\varsigma^{j\cdot}_s\otimes \varsigma^{i\cdot}_t\right)dB_t,
\end{align*}
\tcr{so} we obtain
\begin{align*}
\delta(u_n^{ij}(s,\cdot))\cdot\varsigma^{k\cdot}_s
&=\frac{\sqrt{n}}{2}\int_{I_{h}}\left\{\left(\varsigma^{k\cdot}_s\cdot\varsigma^{i\cdot}_s\right)\varsigma^{j\cdot}_t+\left(\varsigma^{k\cdot}_s\cdot\varsigma^{j\cdot}_s\right)\varsigma^{i\cdot}_t\right\}dB_t.
\end{align*}
Therefore, Propositions \ref{minkowski}--\ref{sharp-BDG}, the Schwarz inequality and Lemma \ref{sigma-l2} yield
\begin{align*}
\left\|\delta(u_n^{ij}(s,\cdot))\cdot\varsigma^{k\cdot}_s\right\|_p
&\lesssim \sqrt{np}\left\|\sqrt{\int_{I_{h}}\left\|\left(\varsigma^{k\cdot}_s\cdot\varsigma^{i\cdot}_s\right)\varsigma^{j\cdot}_t+\left(\varsigma^{k\cdot}_s\cdot\varsigma^{j\cdot}_s\right)\varsigma^{i\cdot}_t\right\|_{\ell_2}^2dt}\right\|_p\\
&\leq \sqrt{np\int_{I_{h}}\left\|\left(\varsigma^{k\cdot}_s\cdot\varsigma^{i\cdot}_s\right)\varsigma^{j\cdot}_t+\left(\varsigma^{k\cdot}_s\cdot\varsigma^{j\cdot}_s\right)\varsigma^{i\cdot}_t\right\|_{p,\ell_2}^2dt}\\
&\leq 2\max_{1\leq i,j,k\leq d}\sqrt{np\int_{I_{h}}\left\|\varsigma^{k\cdot}_s\cdot\varsigma^{i\cdot}_s\right\|_{\frac{3}{2}p}^2\left\|\varsigma^{j\cdot}_t\right\|_{3p,\ell_2}^2dt}\\
&\leq 2\sqrt{p}\max_{1\leq i\leq d}\sup_{0\leq t\leq 1}\left\|\varsigma^{i\cdot}_t\right\|_{3p,\ell_2}^3
\leq2\sqrt{p}\max_{1\leq i\leq d}\sup_{0\leq t\leq 1}\left\|\Sigma^{ii}_t\right\|_{\frac{3}{2}p}^{3/2}.
\end{align*}
This completes the proof. 
\end{proof}

\begin{lemma}\label{lemma:1st-deriv}
Under the assumptions of Theorem \ref{thm:rc}, 
there is a universal constant $C>0$ such that
\begin{align*}
&\left\|\langle DM_n^{ij}\otimes DM_n^{i'j'},u_n^{kl}\rangle_{H^{\otimes}}\right\|_p\\
&\leq \frac{C}{\sqrt{n}}\left(p\max_{1\leq i\leq d}\sup_{0\leq s,t\leq 1}\left\|D_s\sigma^{i\cdot}_t\right\|_{6p,\ell_2}\max_{1\leq j\leq d}\sup_{0\leq t\leq 1}\|\Sigma^{jj}_t\|_{3p}
+\sqrt{p}\max_{1\leq i\leq d}\sup_{0\leq t\leq 1}\left\|\Sigma^{ii}_t\right\|_{3p}^{3/2}
\right)^2,\\
&\left\|\langle DM_n^{ij}\otimes DF,u_n^{kl}\rangle_{H^{\otimes}}\right\|_p\\
&\leq \frac{C}{\sqrt{n}}\left(p\max_{1\leq i\leq d}\sup_{0\leq s,t\leq 1}\left\|D_s\sigma^{i\cdot}_t\right\|_{6p,\ell_2}\max_{1\leq j\leq d}\sup_{0\leq t\leq 1}\|\Sigma^{jj}_t\|_{3p}^{3/2}
+\sqrt{p}\max_{1\leq i\leq d}\sup_{0\leq t\leq 1}\left\|\Sigma^{ii}_t\right\|_{3p}^2
\right)\sup_{0\leq s\leq 1}\left\|D_sF\right\|_{3p,\ell_2},\\
&\left\|\langle DF\otimes DG,u_n^{kl}\rangle_{H^{\otimes}}\right\|_p
\leq \frac{C}{\sqrt{n}}\max_{1\leq i\leq d}\sup_{0\leq t\leq 1}\left\|\Sigma^{ii}_t\right\|_{3p}\sup_{0\leq s\leq 1}\left\|D_sF\right\|_{3p,\ell_2}\sup_{0\leq s\leq 1}\left\|D_sG\right\|_{3p,\ell_2}
\end{align*}
for all $n\in\mathbb{N}$, $i,j,i',j',k,l=1,\dots,d$, $F,G\in\mathbb{D}_{1,\infty}$ and $p\in[2,\infty)$.
\end{lemma}

\begin{proof}
For any $H$-valued random variables $\xi,\eta$, we have
\begin{align*}
\langle\xi\otimes\eta,u_n^{ij}\rangle_{H^{\otimes2}}
=\frac{\sqrt{n}}{2}\sum_{h=1}^n\left\{\left(\int_{I_h}\xi_s\cdot\varsigma_s^{i\cdot}ds\right)\left(\int_{I_h}\eta_s\cdot\varsigma_s^{j\cdot}ds\right)
+\left(\int_{I_h}\xi_s\cdot\varsigma_s^{j\cdot}ds\right)\left(\int_{I_h}\eta_s\cdot\varsigma_s^{i\cdot}ds\right)
\right\}.
\end{align*}
\tcr{Hence,} Proposition \ref{minkowski} yields
\begin{align*}
\left\|\langle\xi\otimes\eta,u_n^{ij}\rangle_{H^{\otimes2}}\right\|_p
\leq\sqrt{n}\max_{1\leq i,j\leq d}\sum_{h=1}^n\left(\int_{I_h}\left\|\xi_s\cdot\varsigma_s^{i\cdot}\right\|_{2p}ds\right)\left(\int_{I_h}\left\|\eta_s\cdot\varsigma_s^{j\cdot}\right\|_{2p}ds\right).
\end{align*}
Now the desired result follows from the Schwarz inequality and Lemmas \ref{sigma-l2} and \ref{lemma:DM}.
\end{proof}

\begin{proof}[Proof of Theorem \ref{thm:rc}]
Set $\mathfrak{S}_n:=\mathfrak{C}_n^{1/2}\zeta_n$. By the hypercontractivity of Gaussian variables, we have
\begin{align*}
E\left[|\mathfrak{S}_n^k|^p\mid\mathcal{F}\right]\leq\left(\sqrt{(p-1)\mathfrak{C}_n^{kk}}\right)^p
\end{align*}
for any $k=1,\dots,d^2$ and $p\in[2,\infty)$. \tcr{H}ence we obtain
\begin{equation}\label{eq:hyper}
\max_{1\leq k\leq d^2}\left\|\mathfrak{S}_n^k\right\|_p\leq\sqrt{p-1}\left\|\mathfrak{C}_n^{kk}\right\|_{p/2}^{1/2}
\leq2\sqrt{p-1}\max_{1\leq i\leq d}\sup_{0\leq t\leq1}\|\Sigma^{ii}_t\|_p
\end{equation}
for any $p\in[2,\infty)$ by the H\"older inequality and Proposition \ref{minkowski}. 

\tcr{
We turn to the main body. First we prove claim (a). Note that we have $\|\|\xi\|_{\ell_\infty}\|_p\leq k^{1/p}\max_{1\leq i\leq k}\|\xi^i\|_p$ for any $p\geq1$, $k\in\mathbb{N}$ and $k$-dimensional random vector $\xi$. 
Then, thanks to Lemmas \ref{lemma:approx}, \ref{drift} and \ref{approx-M}, it suffices to prove
\[
\lim_{n\to\infty}\sup_{y\in\mathbb{R}^m}\left|P\left(\Xi_n\left(M_n+W_n\right)\leq y\right)-P(\Xi_n(\mathfrak{S}_n+W_n)\leq y)\right|=0.
\]
To prove this equation, we apply Theorem \ref{thm:main}. For this purpose we need to verify conditions \eqref{diag-tight}--\eqref{main2:eq10}. 
\eqref{diag-tight} follows from \eqref{rc:diag}. 
\eqref{main2:eq1} follows from Lemma \ref{quasi-torsion}. 
\eqref{main2:eq2} follows from Lemmas \ref{C-deriv} and \ref{lemma:2nd-deriv}. 
\eqref{main2:eq3} follows from Lemma \ref{lemma:2nd-deriv}. 
\eqref{main2:eq4} follows from Lemmas \ref{M-moment}, \ref{lemma:2nd-deriv} and \eqref{eq:hyper}. 
\eqref{main2:eq5}--\eqref{main2:eq6} follow from Lemmas \ref{C-deriv} and \ref{lemma:1st-deriv}. 
\eqref{main2:eq7} follows from Lemmas \ref{M-moment}, \ref{lemma:1st-deriv} and \eqref{eq:hyper}.
\eqref{main2:eq8} follows from Lemmas \ref{C-deriv} and \ref{lemma:1st-deriv}. 
\eqref{main2:eq9}--\eqref{main2:eq10} follow from Lemmas \ref{M-moment}, \ref{C-deriv}, \ref{lemma:1st-deriv} and \eqref{eq:hyper}. So we complete the proof of claim (a).
}

\tcr{
Next, if a $k$-dimensional random vector $\xi$ satisfy $\max_{1\leq i\leq k}\|\xi^i\|_p\leq Ap^{r/2}$ for any $p\in\mathbb{N}$ with some constants $A>0$ and $r\in\mathbb{N}$, then Lemma A.7 and Proposition A.1 of \cite{Koike2017stein} imply that $\|\|\xi\|_{\ell_\infty}\|_p\leq A\log^{r/2}(2k-1+e^{pr/2-1})$ for any $p>0$ with $pr\geq2$. Using this fact, we can prove claim (b) in the same way as the proof of claim (a). 
}
%
\end{proof}

\subsection{Proof of Theorem \ref{thm:rc-local}}

For every $\nu\in\mathbb{N}$, define the process $Y(\nu)=(Y(\nu)_t)_{t\in[0,1]}$ by
\[
Y(\nu)_t=Y_0+\int_0^t\mu(\nu)_sds+\int_0^t\sigma(\nu)_sdB_s,\qquad t\in[0,1].
\]
By the local property of It\^o integrals (cf.~pages 17--18 of \cite{Nualart2006}) we have $Y_t=Y(\nu)_t$ on $\Omega_n(\nu)$ for all $t\in[0,1]$. Therefore, setting $S_n(\nu):=\vectorize\left[\sqrt{n}\left(\widehat{[Y(\nu),Y(\nu)]}^n_1-[Y(\nu),Y(\nu)]_1\right)\right]$, we obtain
\begin{align*}
\rho_n(\nu):=\sup_{y\in\mathbb{R}^m}\left|P\left(\Xi_n(\nu)\left(S_n(\nu)+W_n(\nu)\right)\leq y\right)-P(\Xi_n(\nu)(\mathfrak{C}_n(\nu)^{1/2}\zeta_n+W_n(\nu))\leq y)\right|\to0
\end{align*}
as $n\to\infty$ by Theorem \ref{thm:rc}. Now, for every $y\in\mathbb{R}^m$, we have
\begin{align*}
P\left(\Xi_n\left(S_n+W_n\right)\leq y\right)
&\leq P\left(\Xi_n(\nu)\left(S_n(\nu)+W_n(\nu)\right)\leq y\right)+P(\Omega_n(\nu)^c)\\
&\leq P(\Xi_n(\nu)(\mathfrak{C}_n(\nu)^{1/2}\zeta_n+W_n(\nu))\leq y)+\rho_n(\nu)+P(\Omega_n(\nu)^c)\\
&\leq P(\Xi_n(\mathfrak{C}_n^{1/2}\zeta_n+W_n)\leq y)+\rho_n(\nu)+2P(\Omega_n(\nu)^c).
\end{align*}
By an analogous argument we also have
\[
P\left(\Xi_n\left(S_n+W_n\right)\leq y\right)\geq P(\Xi_n(\mathfrak{C}_n^{1/2}\zeta_n+W_n)\leq y)-\rho_n(\nu)-2P(\Omega_n(\nu)^c).
\]
Consequently, we obtain
\begin{align*}
\limsup_{n\to\infty}\sup_{y\in\mathbb{R}^m}\left|P\left(\Xi_n\left(S_n+W_n\right)\leq y\right)-P(\Xi_n(\mathfrak{C}_n^{1/2}\zeta_n+W_n)\leq y)\right|
\leq 2\limsup_{n\to\infty}P(\Omega_n(\nu)^c).
\end{align*}
Letting $\nu\to\infty$, we complete the proof.\hfill\qed

\subsection{Proof of Proposition \ref{prop:acov}}

We introduce some notation. Given a process $\xi=(\xi_t)_{t\in[0,1]}$ and an interval $I=(S,T]\subset[0,1]$, we set
\[
\xi(I):=\xi_T-\xi_S,\qquad
\xi(I)_t:=\xi_{t\wedge T}-\xi_{t\wedge S}\quad(t\in[0,1]).
\]
Also, we define
\[
L(h)^{ij}:=\mathsf{M}^i(I_h)\mathsf{M}^j(I_h)-[\mathsf{M}^i,\mathsf{M}^j](I_h),\qquad
L(h)^{ij}_t:=\mathsf{M}^i(I_h)_t\mathsf{M}^j(I_h)_t-[\mathsf{M}^i,\mathsf{M}^j](I_h)_t\quad(t\in[0,1])
\]
for $i,j=1,\dots,d$ and $h=1,\dots,n$, where $I_h:=(t_{h-1},t_h]$. 

Next we remark that a localization procedure allows us to reduce the situation of the proposition to the case that $\mu=\mu(\nu)$ and $\sigma=\sigma(\nu)$ for all $n,\nu\in\mathbb{N}$:
\begin{lemma}\label{lemma:local}
Suppose that the statement of Proposition \ref{prop:acov} holds true when we additionally assume $\mu=\mu(\nu)$ and $\sigma=\sigma(\nu)$ for all $n,\nu\in\mathbb{N}$. Then the original statement of Proposition \ref{prop:acov} holds true as well. 
\end{lemma}
The proof of Lemma \ref{lemma:local} is analogous to the one of Theorem \ref{thm:rc-local}, so we omit it. 

\begin{lemma}\label{lemma:driftV3}
There is a universal constant $C>0$ such that
\begin{align*}
\max_{1\leq i,j,k,l\leq d}\sup_{1\leq h\leq n}n^2\left\|\mathsf{A}^i(I_h)\mathsf{X}^j(I_h)\mathsf{Y}^k(I_h)\mathsf{Z}^l(I_h)\right\|_p
&\leq \frac{C}{\sqrt{n}}\max_{1\leq i,j\leq d}\sup_{0\leq ,ts\leq 1}\left\|\mu_s^{ii}\right\|_{4p}\left(
\left\|\mu_t^{j}\right\|_{4p}^3
+p^{3/2}\left\|\Sigma_t^{jj}\right\|_{2p}^{3/2}\right)
\end{align*}
for any $\mathsf{X},\mathsf{Y},\mathsf{Z}\in\{\mathsf{A},\mathsf{M}\}$, $p\in[2,\infty)$ and $n\in\mathbb{N}$.
\end{lemma}

\begin{proof}
This is an immediate consequence of the H\"older inequality and Propositions \ref{minkowski}--\ref{sharp-BDG}.
\end{proof}

\begin{lemma}\label{lemma:V3dV}
There is a universal constant $C>0$ such that
\[
\max_{1\leq i,j,k,l\leq d}\left\|n\sum_{h=1}^n\int_{t_{h-1}}^{t_h}\mathsf{M}^j(I_h)_s\mathsf{M}^k(I_h)_s\mathsf{M}^l(I_h)_sd\mathsf{M}^i_s\right\|_p
\leq C\frac{p^2}{\sqrt{n}}\max_{1\leq i\leq d}\sup_{0\leq s\leq 1}\left\|\Sigma_s^{ii}\right\|_{2p}^2
\]
for all $p\in[2,\infty)$ and $n\in\mathbb{N}$.
\end{lemma}

\begin{proof}
Propositions \ref{minkowski}--\ref{sharp-BDG} yield
\begin{align*}
&\left\|n\sum_{h=1}^n\int_{t_{h-1}}^{t_h}\mathsf{M}^j(I_h)_s\mathsf{M}^k(I_h)_s\mathsf{M}^l(I_h)_sd\mathsf{M}^i_s\right\|_p\\
&\lesssim n\sqrt{p}\left\|\sqrt{\sum_{h=1}^n\int_{t_{h-1}}^{t_h}\mathsf{M}^j(I_h)_s^2\mathsf{M}^k(I_h)^2_s\mathsf{M}^l(I_h)^2_sd[\mathsf{M}^i,\mathsf{M}^i]_s}\right\|_p\\
&\leq n\sqrt{p\sum_{h=1}^n\int_{t_{h-1}}^{t_h}\left\|\mathsf{M}^j(I_h)_s\right\|_{4p}^2\left\|\mathsf{M}^k(I_h)_s\right\|_{4p}^2\left\|\mathsf{M}^l(I_h)_s\right\|_{4p}^2\left\|\Sigma_s^{ii}\right\|_{2p}ds}\\
&\lesssim np^2\sqrt{\sum_{h=1}^n\int_{t_{h-1}}^{t_h}\left\|\sqrt{\int_{t_{h-1}}^s\Sigma^{jj}_udu}\right\|_{4p}^2\left\|\sqrt{\int_{t_{h-1}}^s\Sigma^{kk}_udu}\right\|_{4p}^2\left\|\sqrt{\int_{t_{h-1}}^s\Sigma^{ll}_udu}\right\|_{4p}^2\left\|\Sigma_s^{ii}\right\|_{2p}ds}\\
&\leq \frac{p^2}{\sqrt{n}}\max_{1\leq i\leq d}\sup_{0\leq s\leq 1}\left\|\Sigma_s^{ii}\right\|_{2p}^2.
\end{align*}
This completes the proof. 
\end{proof}

\begin{lemma}\label{l-qv}
Suppose that the assumptions of of Proposition \ref{prop:acov} hold true under the additional assumption that $\mu=\mu(\nu)$ and $\sigma=\sigma(\nu)$ for all $n,\nu\in\mathbb{N}$. 
Then there is a universal constant $C>0$ such that
\begin{align*}
&\max_{1\leq i,j,k,l\leq d}\left\|n\sum_{h=1}^{n-\nu}L(h)^{ij}[\mathsf{M}^k,\mathsf{M}^l](I_{h+\nu})\right\|_p\\
&\leq\frac{C}{\sqrt{n}}\left(p\max_{1\leq i\leq d}\sup_{0\leq t\leq1}\|\Sigma^{ii}_t\|_{2p}^2
+p^{3/2}\max_{1\leq i\leq d}\sup_{0\leq t\leq 1}\|\Sigma_t^{ii}\|_{2p}^{3/2}
\max_{1\leq k\leq d}\sup_{0\leq u\leq v\leq 1}\|D_u\sigma^{k\cdot}_v\|_{4p,\ell_2}\right)
\end{align*}
for all $p\in[2,\infty)$, $n\in\mathbb{N}$ and $\nu\in\{0,1\}$.
\end{lemma}

\begin{proof}
We decompose the target quantity as
\begin{align*}
n\sum_{h=1}^{n-\nu}L(h)^{ij}[\mathsf{M}^k,\mathsf{M}^l](I_{h+\nu})
&=n\sum_{h=1}^{n-\nu}L(h)^{ij}\int_{I_{h+\nu}}E\left[\Sigma^{kl}_t|\mathcal{F}_{t_{h-1}}\right]dt
+n\sum_{h=1}^{n-\nu}L(h)^{ij}\int_{I_{h+\nu}}\left(\Sigma^{kl}_t-E\left[\Sigma^{kl}_t|\mathcal{F}_{t_{h-1}}\right]\right)dt\\
&=:\mathbf{I}_n+\mathbf{II}_n.
\end{align*}
First we consider $\mathbf{I}_n$. Set 
\[
\phi_h:=\int_{I_{h+\nu}}E\left[\Sigma^{kl}_t|\mathcal{F}_{t_{h-1}}\right]dt,\qquad
h=1,\dots,n-\nu.
\]
Then we can rewrite $\mathbf{I}_n$ as
\begin{align*}
\mathbf{I}_n=n\sum_{h=1}^{n-\nu}\left\{\int_{I_h}\left(\int_{t_{h-1}}^t\sigma^{i\cdot}_s\cdot dB_s\right)\phi_h\sigma^{j\cdot}\cdot dB_t
+\int_{I_h}\left(\int_{t_{h-1}}^t\phi_h\sigma^{j\cdot}_s\cdot dB_s\right)\sigma^{i\cdot}\cdot dB_t
\right\}.
\end{align*}
Therefore, Lemmas \ref{sigma-l2}--\ref{lemma:BDG} and the H\"older inequality imply that
\begin{align*}
\left\|\mathbf{I}_n\right\|_p
&\lesssim p\sqrt{n}\sup_{0\leq s\leq 1}\|\sigma^{i\cdot}_s\|_{4p,\ell_2}\sup_{0\leq s< 1-\nu/n}\|\phi_{\lceil ns\rceil}\sigma^{j\cdot}_s\|_{\frac{4}{3}p,\ell_2}
\leq p\sqrt{n}\sup_{0\leq s\leq 1}\|\Sigma^{ii}_s\|_{2p}^{1/2}\sup_{0\leq s< 1-\nu/n}\|\phi_{\lceil ns\rceil}\|_{2p}\|\Sigma^{jj}_s\|_{2p}^{1/2}.
\end{align*}
Now, Proposition \ref{minkowski} and the Lyapunov and Schwarz inequalities yield
\begin{align*}
\sup_{0\leq s< 1-\nu/n}\|\phi_{\lceil ns\rceil}\|_{2p}
\leq\frac{1}{n}\sup_{0\leq t\leq 1}\|\Sigma^{kl}_t\|_{2p}
\leq\frac{1}{n}\max_{1\leq k\leq d}\sup_{0\leq t\leq 1}\|\Sigma^{kk}_t\|_{2p}.
\end{align*}
Consequently, we obtain
\[
\left\|\mathbf{I}_n\right\|_p\lesssim \frac{p}{\sqrt{n}}\max_{1\leq i\leq d}\sup_{0\leq t\leq1}\|\Sigma^{ii}_t\|_{2p}^2.
\]
Next we consider $\mathbf{II}_n$. By Proposition \ref{minkowski} we have
\[
\|\mathbf{II}_n\|_p
\leq n\sum_{h=1}^{n-\nu}\|L(h)^{ij}\|_{2p}\int_{I_{h+\nu}}\left\|\Sigma^{kl}_t-E\left[\Sigma^{kl}_t|\mathcal{F}_{t_{h-1}}\right]\right\|_{2p}dt.
\]
It\^o's formula, Lemmas \ref{lemma:BDG} and \ref{sigma-l2} yield
\[
\|L(h)^{ij}\|_{2p}\lesssim\frac{p}{n}\max_{1\leq i\leq d}\sup_{0\leq t\leq 1}\|\Sigma_t^{ii}\|_{2p}.
\]
Meanwhile, Lemmas \ref{ocone} and \ref{S-deriv} yield
\[
\left\|\Sigma^{kl}_t-E\left[\Sigma^{kl}_t|\mathcal{F}_{t_{h-1}}\right]\right\|_{2p}
\lesssim\sqrt{\frac{p}{n}}\sup_{0\leq u\leq v\leq 1}\|D_u\Sigma^{kl}_v\|_{2p,\ell_2}
\lesssim\sqrt{\frac{p}{n}}\max_{1\leq k,l\leq d}\sup_{0\leq u\leq v\leq 1}\|\Sigma^{ll}_v\|_{2p}^{1/2}\|D_u\sigma^{k\cdot}_v\|_{4p,\ell_2}.
\]
Consequently, we obtain
\[
\|\mathbf{II}_n\|_p
\lesssim\sqrt{\frac{p^3}{n}}\max_{1\leq i\leq d}\sup_{0\leq t\leq 1}\|\Sigma_t^{ii}\|_{2p}^{3/2}
\max_{1\leq k\leq d}\sup_{0\leq u\leq v\leq 1}\|D_u\sigma^{k\cdot}_v\|_{4p,\ell_2}.
\]
We thus complete the proof. 
\end{proof}

\begin{lemma}\label{qv-dl}
Under the assumptions of of Lemma \ref{l-qv}, 
there is a universal constant $C>0$ such that
\begin{align*}
&\max_{1\leq i,j,k,l\leq d}\left\|n\sum_{h=1}^n\int_{t_{h-1}}^{t_h}L(h)^{kl}_td[\mathsf{M}^i,\mathsf{M}^j]_t\right\|_p\\
&\leq\frac{C}{\sqrt{n}}\left(p\max_{1\leq i\leq d}\sup_{0\leq t\leq1}\|\Sigma^{ii}_t\|_{2p}^2
+p^{3/2}\max_{1\leq i\leq d}\sup_{0\leq t\leq 1}\|\Sigma_t^{ii}\|_{2p}^{3/2}
\max_{1\leq k\leq d}\sup_{0\leq u\leq v\leq 1}\|D_u\sigma^{k\cdot}_v\|_{4p,\ell_2}\right)
\end{align*}
for all $p\in[2,\infty)$ and $n\in\mathbb{N}$.
\end{lemma}

\begin{proof}
By It\^o's formula we can rewrite the target quantity as
\begin{align*}
n\sum_{h=1}^n\int_{t_{h-1}}^{t_h}L(h)^{kl}_td[\mathsf{M}^i,\mathsf{M}^j]_t
=n\sum_{h=1}^nL(h)^{kl}[\mathsf{M}^i,\mathsf{M}^j](I_h)
+n\sum_{h=1}^n\int_{t_{h-1}}^{t_h}[\mathsf{M}^i,\mathsf{M}^j](I_h)_tdL(h)^{kl}_t.
\end{align*}
Since we have
\begin{align*}
&n\sum_{h=1}^n\int_{t_{h-1}}^{t_h}[\mathsf{M}^i,\mathsf{M}^j](I_h)_tdL(h)^{kl}_t\\
&=n\sum_{h=1}^n\left\{\int_{t_{h-1}}^{t_h}\left(\int_{t_{h-1}}^t\sigma_s^{k\cdot}\cdot dB_s\right)[\mathsf{M}^i,\mathsf{M}^j](I_h)_t\sigma_t^{l\cdot}\cdot dB_t
+\int_{t_{h-1}}^{t_h}\left(\int_{t_{h-1}}^t\sigma_s^{l\cdot}\cdot dB_s\right)[\mathsf{M}^i,\mathsf{M}^j](I_h)_t\sigma_t^{k\cdot}\cdot dB_t
\right\}
\end{align*}
by It\^o's formula, Lemmas \ref{lemma:BDG} and \ref{sigma-l2} yield
\begin{align*}
\left\|n\sum_{h=1}^n\int_{t_{h-1}}^{t_h}[\mathsf{M}^i,\mathsf{M}^j](I_h)_sdL(h)^{kl}_s\right\|_p
&\lesssim \sqrt{n}p\max_{1\leq k,l\leq d}\sup_{0\leq s\leq1}\|\sigma_s^{k\cdot}\|_{4p,\ell_2}\sup_{0\leq s\leq1}\|[\mathsf{M}^i,\mathsf{M}^j](I_h)_s\sigma_s^{l\cdot}\|_{\frac{4}{3}p,\ell_2}\\
&\leq \sqrt{n}p\max_{1\leq k,l\leq d}\sup_{0\leq s\leq1}\|\Sigma_s^{kk}\|_{2p}^{1/2}\sup_{0\leq s\leq1}\|[\mathsf{M}^i,\mathsf{M}^j](I_h)_s\|_{2p}\|\Sigma_s^{ll}\|_{2p}^{1/2}.
\end{align*}
Since Proposition \ref{minkowski} and the Schwarz inequality imply that
\[
\|[\mathsf{M}^i,\mathsf{M}^j](I_h)_s\|_{2p}\leq\frac{1}{n}\sup_{0\leq s\leq1}\|\Sigma_s^{ij}\|_{2p}
\leq\frac{1}{n}\max_{1\leq i\leq d}\sup_{0\leq s\leq1}\|\Sigma_s^{ii}\|_{2p},
\]
we obtain
\begin{align*}
\left\|n\sum_{h=1}^n\int_{t_{h-1}}^{t_h}[\mathsf{M}^i,\mathsf{M}^j](I_h)_sdL(h)^{kl}_s\right\|_p
\lesssim \frac{p}{\sqrt{n}}\max_{1\leq i\leq d}\sup_{0\leq s\leq1}\|\Sigma_s^{ii}\|_{2p}^2.
\end{align*}
Combining this estimate with Lemma \ref{l-qv}, we obtain the desired result. 
\end{proof}

\begin{lemma}\label{lemma:no-lag}
Under the assumptions of of Lemma \ref{l-qv}, 
there is a universal constant $C>0$ such that
\begin{align*}
&\max_{1\leq i,j,k,l\leq d}\left\|n\sum_{h=1}^n\mathsf{M}^i(I_h)\mathsf{M}^j(I_h)\mathsf{M}^k(I_h)\mathsf{M}^l(I_h)\right.\\
&\left.\hphantom{\max_{1\leq i,j,k,l\leq d}}-n\sum_{h=1}^n\left\{[\mathsf{M}^i,\mathsf{M}^j](I_h)[\mathsf{M}^k,\mathsf{M}^l](I_h)
+[\mathsf{M}^i,\mathsf{M}^k](I_h)[\mathsf{M}^j,\mathsf{M}^l](I_h)
+[\mathsf{M}^j,\mathsf{M}^k](I_h)[\mathsf{M}^i,\mathsf{M}^l](I_h)\right\}\right\|_p\\
&\leq\frac{C}{\sqrt{n}}\left(p^2\max_{1\leq i\leq d}\sup_{0\leq t\leq1}\|\Sigma^{ii}_t\|_{2p}^2
+p^{3/2}\max_{1\leq i\leq d}\sup_{0\leq t\leq 1}\|\Sigma_t^{ii}\|_{2p}^{3/2}
\max_{1\leq k\leq d}\sup_{0\leq u\leq v\leq 1}\|D_u\sigma^{k\cdot}_v\|_{4p,\ell_2}\right)
\end{align*}
for all $p\in[2,\infty)$ and $n\in\mathbb{N}$.
\end{lemma}

\begin{proof}
Using It\^o's formula repeatedly, we have
\begin{align*}
&\mathsf{M}^i(I_h)\mathsf{M}^j(I_h)\mathsf{M}^k(I_h)\mathsf{M}^l(I_h)\\
&=\int_{t_{h-1}}^{t_h}\mathsf{M}^j(I_h)_s\mathsf{M}^k(I_h)_s\mathsf{M}^l(I_h)_sd\mathsf{M}^i_s
+\int_{t_{h-1}}^{t_h}\mathsf{M}^i(I_h)_s\mathsf{M}^k(I_h)_s\mathsf{M}^l(I_h)_sd\mathsf{M}^j_s\\
&\quad+\int_{t_{h-1}}^{t_h}\mathsf{M}^i(I_h)_s\mathsf{M}^j(I_h)_s\mathsf{M}^l(I_h)_sd\mathsf{M}^k_s
+\int_{t_{h-1}}^{t_h}\mathsf{M}^i(I_h)_s\mathsf{M}^j(I_h)_s\mathsf{M}^k(I_h)_sd\mathsf{M}^l_s\\
&\quad+\int_{t_{h-1}}^{t_h}L(h)^{kl}_sd[\mathsf{M}^i,\mathsf{M}^j]_s
+\int_{t_{h-1}}^{t_h}L(h)^{jl}_sd[\mathsf{M}^i,\mathsf{M}^k]_s
+\int_{t_{h-1}}^{t_h}L(h)^{jk}_sd[\mathsf{M}^i,\mathsf{M}^l]_s\\
&\quad+\int_{t_{h-1}}^{t_h}L(h)^{il}_sd[\mathsf{M}^j,\mathsf{M}^k]_s
+\int_{t_{h-1}}^{t_h}L(h)^{ik}_sd[\mathsf{M}^j,\mathsf{M}^l]_s
+\int_{t_{h-1}}^{t_h}L(h)^{ij}_sd[\mathsf{M}^k,\mathsf{M}^l]_s\\
&\quad+[\mathsf{M}^i,\mathsf{M}^j](I_h)[\mathsf{M}^k,\mathsf{M}^l](I_h)
+[\mathsf{M}^i,\mathsf{M}^k](I_h)[\mathsf{M}^j,\mathsf{M}^l](I_h)
+[\mathsf{M}^j,\mathsf{M}^k](I_h)[\mathsf{M}^i,\mathsf{M}^l](I_h)
\end{align*}
for every $h$. Therefore, the desired result follows from Lemmas \ref{lemma:V3dV} and \ref{qv-dl}.
\end{proof}

\begin{lemma}\label{vv-lag}
Under the assumptions of of Lemma \ref{l-qv}, 
there is a universal constant $C>0$ such that
\begin{align*}
\max_{1\leq i,j,k,l\leq d}\left\|n\sum_{h=1}^{n-1}\mathsf{M}^i(I_h)\mathsf{M}^j(I_h)\int_{t_{h}}^{t_{h+1}}\mathsf{M}^l(I_{h+1})_sd\mathsf{M}^k_s\right\|_p
\leq C\frac{p^2}{\sqrt{n}}\max_{1\leq i\leq d}\sup_{0\leq s\leq 1}\|\Sigma^{ii}_s\|_{2p}^2
\end{align*}
for all $p\in[2,\infty)$ and $n\in\mathbb{N}$.
\end{lemma}

\begin{proof}
Proposition \ref{minkowski}--\ref{sharp-BDG} yield
\begin{align*}
&\left\|n\sum_{h=1}^{n-1}\mathsf{M}^i(I_h)\mathsf{M}^j(I_h)\int_{t_{h}}^{t_{h+1}}\mathsf{M}^l(I_{h+1})_sd\mathsf{M}^k_s\right\|_p\\
&\lesssim n\sqrt{p}\left\|\sqrt{\sum_{h=1}^{n-1}\mathsf{M}^i(I_h)^2\mathsf{M}^j(I_h)^2\int_{t_{h}}^{t_{h+1}}\mathsf{M}^l(I_{h+1})_s^2d[\mathsf{M}^k,\mathsf{M}^k]_s}\right\|_p\\
&\leq n\sqrt{p\sum_{h=1}^{n-1}\left\|\mathsf{M}^i(I_h)\right\|_{4p}^2\left\|\mathsf{M}^j(I_h)\right\|_{4p}^2\int_{t_{h}}^{t_{h+1}}\left\|\mathsf{M}^l(I_{h+1})_s\right\|_{4p}^2\left\|\Sigma_s^{kk}\right\|_{2p}ds}\\
&\lesssim np^2\sqrt{\sum_{h=1}^{n-1}\left\|\sqrt{\int_{I_h}\Sigma^{ii}_sds}\right\|_{4p}^2\left\|\sqrt{\int_{I_h}\Sigma^{jj}_sds}\right\|_{4p}^2\int_{t_{h}}^{t_{h+1}}\left\|\sqrt{\int_{t_h}^s\Sigma^{ll}_sds}\right\|_{4p}^2\left\|\Sigma_s^{kk}\right\|_{2p}ds}\\
&\leq\frac{p^2}{\sqrt{n}}\max_{1\leq i\leq d}\sup_{0\leq s\leq 1}\|\Sigma^{ii}_s\|_{2p}^2.
\end{align*}
This completes the proof.
\end{proof}

\begin{lemma}\label{lemma:lag}
Under the assumptions of of Lemma \ref{l-qv}, 
there is a universal constant $C>0$ such that
\begin{align*}
&\max_{1\leq i,j,k,l\leq d}\left\|n\sum_{h=1}^{n-1}\mathsf{M}^i(I_h)\mathsf{M}^j(I_h)\mathsf{M}^k(I_{h+1})\mathsf{M}^l(I_{h+1})-n\sum_{h=1}^{n-1}[\mathsf{M}^i,\mathsf{M}^j](I_h)[\mathsf{M}^k,\mathsf{M}^l](I_{h+1})\right\|_p\\
&\leq\frac{C}{\sqrt{n}}\left(p^2\max_{1\leq i\leq d}\sup_{0\leq t\leq1}\|\Sigma^{ii}_t\|_{2p}^2
+p^{3/2}\max_{1\leq i\leq d}\sup_{0\leq t\leq 1}\|\Sigma_t^{ii}\|_{2p}^{3/2}
\max_{1\leq k\leq d}\sup_{0\leq u\leq v\leq 1}\|D_u\sigma^{k\cdot}_v\|_{4p,\ell_2}\right)
\end{align*}
for all $p\in[2,\infty)$ and $n\in\mathbb{N}$.
\end{lemma}

\begin{proof}
By It\^o's formula we have
\begin{align*}
\mathsf{M}^i(I_h)\mathsf{M}^j(I_h)\mathsf{M}^k(I_{h+1})\mathsf{M}^l(I_{h+1})
&=\mathsf{M}^i(I_h)\mathsf{M}^j(I_h)\int_{t_{h}}^{t_{h+1}}\mathsf{M}^l(I_{h+1})_sd\mathsf{M}^k_s
+\mathsf{M}^i(I_h)\mathsf{M}^j(I_h)\int_{t_{h}}^{t_{h+1}}\mathsf{M}^k(I_{h+1})_sd\mathsf{M}^l_s\\
&\quad+L(h)^{ij}[\mathsf{M}^k,\mathsf{M}^l](I_{h+1})
+[\mathsf{M}^i,\mathsf{M}^j](I_h)[\mathsf{M}^k,\mathsf{M}^l](I_{h+1})
\end{align*}
for every $h$. Therefore, the desired result follows from Lemmas \ref{l-qv} and \ref{vv-lag}.
\end{proof}

\begin{proof}[Proof of Proposition \ref{prop:acov}]
Thanks to Lemma \ref{lemma:local}, throughout the proof we may assume $\mu=\mu(\nu)$ and $\sigma=\sigma(\nu)$ for all $n,\nu\in\mathbb{N}$. 

(a) According to Lemmas \ref{lemma:driftV3}, \ref{lemma:no-lag} and \ref{lemma:lag}, it suffices to show that
\begin{align}
&E\left[\max_{1\leq i,j,k,l\leq d}\left|n[\mathsf{M}^i,\mathsf{M}^j](I_n)[\mathsf{M}^k,\mathsf{M}^l](I_{n})\right|\right]=O(n^{-\varpi}),\label{acov:aim1}\\
&E\left[\max_{1\leq i,j,k,l\leq d}\left|n\sum_{h=1}^{n-1}[\mathsf{M}^i,\mathsf{M}^j](I_h)\left\{[\mathsf{M}^k,\mathsf{M}^l](I_{h+1})-[\mathsf{M}^k,\mathsf{M}^l](I_{h})\right\}\right|\right]=O(n^{-\varpi}).\label{acov:aim2}
\end{align}
\eqref{acov:aim1} is evident from assumptions. 
In the meantime, the Schwarz inequality and Proposition \ref{minkowski} yield
\begin{align}
&E\left[\max_{1\leq i,j,k,l\leq d}\left|n\sum_{h=1}^{n-1}[\mathsf{M}^i,\mathsf{M}^j](I_h)\left\{[\mathsf{M}^k,\mathsf{M}^l](I_{h+1})-[\mathsf{M}^k,\mathsf{M}^l](I_{h})\right\}\right|\right]\nonumber\\
&\leq n\sum_{h=1}^{n-1}E\left[\max_{1\leq i,j\leq d}\left|[\mathsf{M}^i,\mathsf{M}^j](I_h)\right|\max_{1\leq k,l\leq d}\left|[\mathsf{M}^k,\mathsf{M}^l](I_{h+1})-[\mathsf{M}^k,\mathsf{M}^l](I_{h})\right|\right]\nonumber\\
&\leq \sup_{0\leq t\leq1}\left\|\max_{1\leq i,j\leq d}\left|\Sigma^{ij}_t\right|\right\|_2\sup_{0< t\leq 1-\frac{1}{n}}\left\|\max_{1\leq k,l\leq d}\left|\Sigma^{kl}_{t+\frac{1}{n}}-\Sigma^{kl}_t\right|\right\|_2,\label{proof:modulus}
\end{align}
\tcr{so} \eqref{acov:aim2} also follows from assumptions. This completes the proof. 

(b) By assumptions we have $E[\max_{1\leq i,j,k,l\leq d}\left|n[\mathsf{M}^i,\mathsf{M}^j](I_n)[\mathsf{M}^k,\mathsf{M}^l](I_{n})\right|]=O(n^{-1})$. Moreover, from \eqref{proof:modulus} and assumptions, we also have
\[
E\left[\max_{1\leq i,j,k,l\leq d}\left|n\sum_{h=1}^{n-1}[\mathsf{M}^i,\mathsf{M}^j](I_h)\left\{[\mathsf{M}^k,\mathsf{M}^l](I_{h+1})-[\mathsf{M}^k,\mathsf{M}^l](I_{h})\right\}\right|\right]=O(n^{-\gamma}).
\]
Therefore, the desired result follows from Lemmas \ref{lemma:driftV3}, \ref{lemma:no-lag} and \ref{lemma:lag} as well as Lemma A.7 and Proposition A.1 of \cite{Koike2017stein}. 
\end{proof}

\subsection{Proof of Proposition \ref{prop:factor-test}}

An analogous argument to the proof of Theorem \ref{thm:rc-local} allows us to assume $\mu=\mu(\nu)$ and $\Sigma=\Sigma(\nu)$ for all $n,\nu\in\mathbb{N}$. 

Define the $\ul{d}^2\times d^2$ random matrix $\hat{\bs{X}}_n$ by
\[
\hat{\bs{X}}_n^{(i-1)\ul{d}+j,(k-1)d+l}=
\left\{
\begin{array}{cl}
~\wh{[Y^j,Y^d]}^n_1/\sqrt{\hat{\mathfrak{V}}_n^{ij}}  & \text{if }k=i,~l=d,  \\
~[Y^i,Y^d]_1/\sqrt{\hat{\mathfrak{V}}_n^{ij}}  & \text{if }k=j,~l=d, \\
~-\wh{[Y^d,Y^d]}^n_1/\sqrt{\hat{\mathfrak{V}}_n^{ij}}  & \text{if }k=l=d  \\
~-[Y^i,Y^j]_1/\sqrt{\hat{\mathfrak{V}}_n^{ij}} & \text{if }k=i, l=j\\
0 & \text{otherwise}.   
\end{array}
\right.
\]
for $i,j=1,\dots,\ul{d}$ and $k,l=1,\dots,d$. 
We also define the $\ul{d}^2\times d^2$ matrix $\ul{\Upsilon}_n$ by
\[
\ul{\Upsilon}_n^{(i-1)\ul{d}+j,(k-1)d+l}=
\left\{
\begin{array}{cl}
1  & \text{if }k\in\{i,j\}, l=d , \\
-1 & \text{if }k=i, l=j\text{ or }k=l=d, \\
0 & \text{otherwise}   
\end{array}
\right.
\]
for $i,j=1,\dots,\ul{d}$ and $k,l=1,\dots,d$. Then we set
\[
\Xi_n=
\left(
\begin{array}{c}
\bs{X}_n   \\ 
-\bs{X}_n
\end{array}
\right),\qquad
\hat{\Xi}_n=
\left(
\begin{array}{c}
\hat{\bs{X}}_n   \\ 
-\hat{\bs{X}}_n
\end{array}
\right),\qquad
\Upsilon_n=
\left(
\begin{array}{c}
\ul{\Upsilon}_n   \\ 
-\ul{\Upsilon}_n
\end{array}
\right). 
\]
Since we have
\[
\hat{\Xi}_nS_n=
\left(
\begin{array}{c}
\vectorize(T_n)   \\ 
-\vectorize(T_n)
\end{array}
\right),\qquad
\hat{\Xi}_nS_n^*=
\left(
\begin{array}{c}
\vectorize(T_n^*)   \\ 
-\vectorize(T_n^*)
\end{array}
\right)
\]
as well as all the diagonal entries of $\Xi_n\mathfrak{C}_n\Xi_n^\top$ are equal to 1 by the definition of $\bs{X}_n$, 
Lemma \ref{lemma:approx} and Proposition \ref{prop:comparison} imply that it suffices to prove the following equations:
\begin{align}
&\sup_{y\in(\infty,\infty]^{2\ul{d}^2}}|P(\Xi_nS_n\leq y)-P(\Xi_n\mathfrak{C}_n^{1/2}\zeta_n\leq y)|\to0,\label{factor-aim1}\\
&\sqrt{\log d}\|\hat{\Xi}_nS_n-\Xi_nS_n\|_{\ell_\infty}\to^p0,\label{factor-aim2}\\
&(\log d)^2\|\hat{\Xi}_n\hat{\mathfrak{C}}_n\hat{\Xi}_n^\top-\Xi_n\mathfrak{C}_n\Xi_n^\top\|_{\ell_\infty}\to^p0\label{factor-aim3}
\end{align}
as $n\to\infty$. 

We begin by proving \eqref{factor-aim1}. Since $\bs{X}_n=\bs{X}_n\circ\ul{\Upsilon}_n$ and $\opnorm{\ul{\Upsilon}_n}_\infty=4$, an application of Theorem \ref{thm:rc} implies that the desired result follows once we show that $\bs{X}_n\in\mathbb{D}_{2,\infty}(\mathbb{R}^{\ul{d}^2}\otimes \mathbb{R}^{d^2})$ and
\[
\sup_{n\in\mathbb{N}}\max_{1\leq i\leq\ul{d}^2,1\leq j\leq d^2}\left(\|\bs{X}_n^{ij}\|_p+\sup_{0\leq t\leq 1}\|D_t\bs{X}_n^{ij}\|_{p,\ell_2}
+\sup_{0\leq s,t\leq 1}\|D_{s,t}\bs{X}_n^{ij}\|_{p,\ell_2}
\right)<\infty
\]
for all $p\in[2,\infty)$. 
By Remark 15.87 of \cite{Janson1997}, we have $[Y^i,Y^j]_1\in\mathbb{D}_{2,\infty}$ and $D^q[Y^i,Y^j]_1=\int_0^1D^q\Sigma_t^{ij}dt$ for any $i,j=1,\dots,d$ and $q=1,2$. Therefore, by Corollary 15.80 of \cite{Janson1997}, the desired result follows once we show that $1/\sqrt{\mathfrak{V}_n^{ij}}\in\mathbb{D}_{2,\infty}$ for any $i,j=1,\dots,\ul{d}$ and 
\begin{equation}\label{v-half-moment}
\sup_{n\in\mathbb{N}}\max_{1\leq i,j\leq\ul{d}}\left(\left\|\frac{1}{\sqrt{\mathfrak{V}_n^{ij}}}\right\|_p
+\sup_{0\leq t\leq 1}\left\|D_t\left(\frac{1}{\sqrt{\mathfrak{V}_n^{ij}}}\right)\right\|_{p,\ell_2}
+\sup_{0\leq s,t\leq 1}\left\|D_{s,t}\left(\frac{1}{\sqrt{\mathfrak{V}_n^{ij}}}\right)\right\|_{p,\ell_2}
\right)<\infty
\end{equation}
for all $p\in[2,\infty)$. Note that we have
\[
\|D_t[Y^i,Y^j]_1\|_{p,\ell_2}\leq\sup_{0\leq u\leq1}\|D_t\Sigma_u^{ij}\|_{p,\ell_2},\qquad
\left\|D_{s,t}[Y^i,Y^j]_1\right\|_{p,\ell_2}\leq\sup_{0\leq u\leq1}\left\|D_{s,t}\Sigma_u^{ij}\right\|_{p,\ell_2}
\]
for all $i,j=1,\dots,d$, $p\in[2,\infty)$ and $s,t\in[0,1]$ by Proposition \ref{minkowski}. 
Therefore, Lemmas \ref{S-deriv}--\ref{C-deriv} and Corollary 15.80 of \cite{Janson1997} imply that $\mathfrak{V}_n^{ij}\in\mathbb{D}_{2,\infty}$ for any $i,j=1,\dots,\ul{d}$ and
\[
\sup_{n\in\mathbb{N}}\max_{1\leq i,j\leq\ul{d}}\left(\left\|\mathfrak{V}_n^{ij}\right\|_p
+\sup_{0\leq t\leq 1}\left\|D_t\mathfrak{V}_n^{ij}\right\|_{p,\ell_2}
+\sup_{0\leq s,t\leq 1}\left\|D_{s,t}\mathfrak{V}_n^{ij}\right\|_{p,\ell_2}
\right)<\infty
\]
for all $p\in[2,\infty)$. Now, since we can write $1/\sqrt{\mathfrak{V}_n^{ij}}=(\mathfrak{V}_n^{ij})^{5/2}(\mathfrak{V}_n^{ij})^{-3}$, Theorem 15.78 and Lemma 15.152 of \cite{Janson1997} as well as \eqref{v-inv-moment} imply that $1/\sqrt{\mathfrak{V}_n^{ij}}\in\mathbb{D}_{2,\infty}$ for any $i,j=1,\dots,\ul{d}$ and \eqref{v-half-moment} holds true for all $p\in[2,\infty)$. Hence we complete the proof of \eqref{factor-aim1}. 

Next we prove \eqref{factor-aim2}--\eqref{factor-aim3}. 
First, note that we have $\|S_n\|_{\ell_\infty}=O_p(n^\eta)$ as $n\to\infty$ for any $\eta>0$ by Corollary \ref{coro:rc} and \eqref{eq:hyper}. 
Since $\|[Y,Y]_1\|_{\ell_\infty}=O_p(n^\eta)$ as $n\to\infty$ for any $\eta>0$ by assumptions, this especially yields $\|\wh{[Y,Y]}^n_1\|_{\ell_\infty}=O_p(n^\eta)$ as $n\to\infty$ for any $\eta>0$. 
Next we verify
\begin{equation*}
\max_{1\leq i,j\leq \ul{d}}|\hat{\mathfrak{V}}_n^{ij}-\mathfrak{V}_n^{ij}|=O_p(n^{-\varpi})
\end{equation*} 
as $n\to\infty$ for any $\varpi\in(0,\gamma)$. In fact, by definition we have
\begin{align*}
&\max_{1\leq i,j\leq \ul{d}}|\hat{\mathfrak{V}}_n^{ij}-\mathfrak{V}_n^{ij}|\\
&\lesssim \left(\left\|\wh{[Y,Y]}^n_1\right\|_{\ell_\infty}+\left\|[Y,Y]_1\right\|_{\ell_\infty}\right)
\left(\left\|\hat{\mathfrak{C}}_n\right\|_{\ell_\infty}+\left\|\mathfrak{C}_n\right\|_{\ell_\infty}\right)
\left(\left\|\wh{[Y,Y]}^n_1-[Y,Y]_1\right\|_{\ell_\infty}
+\left\|\hat{\mathfrak{C}}_n-\mathfrak{C}_n\right\|_{\ell_\infty}\right).
\end{align*}
Since $\|\mathfrak{C}_n\|_{\ell_\infty}=O_p(n^\eta)$ for any $\eta>0$ by assumptions, the desired result follows from Proposition \ref{prop:acov} and the results noted above. 
In particular, it holds that $\max_{1\leq i,j\leq\ul{d}}|1/\sqrt{\hat{\mathfrak{V}}_n^{ij}}|=O_p(n^\eta)$ for any $\eta>0$ because $\max_{1\leq i,j\leq\ul{d}}|1/\sqrt{\mathfrak{V}_n^{ij}}|=O_p(n^\eta)$ for any $\eta>0$ by assumptions. 
Moreover, we have
\begin{align*}
\|\hat{\bs{X}}_n-\bs{X}_n\|_{\ell_\infty}
\leq\max_{1\leq i,j\leq \ul{d}}\left|\frac{1}{\sqrt{\hat{\mathfrak{V}}_n^{ij}}}\right|\left\|\wh{[Y,Y]}^n_1-[Y,Y]_1\right\|_{\ell_\infty}
+\max_{1\leq i,j\leq \ul{d}}\left|\frac{1}{\sqrt{\hat{\mathfrak{V}}_n^{ij}}}-\frac{1}{\sqrt{\mathfrak{V}_n^{ij}}}\right|\left\|[Y,Y]_1\right\|_{\ell_\infty},
\end{align*}
\tcr{and thus} it holds that $\|\hat{\bs{X}}_n-\bs{X}_n\|_{\ell_\infty}=O_p(n^{-\varpi})$ as $n\to\infty$ for any $\varpi\in(0,\gamma)$. Noting that we have $\|\bs{X}_n\|_{\ell_\infty}=O_p(n^\eta)$ for any $\eta>0$ by assumptions, this particularly implies that $\|\hat{\bs{X}}_n\|_{\ell_\infty}=O_p(n^\eta)$ for any $\eta>0$. 

Now since we have
\begin{align*}
\|\hat{\Xi}_nS_n-\Xi_nS_n\|_{\ell_\infty}
&\leq4\|\hat{\bs{X}}_n-\bs{X}_n\|_{\ell_\infty}\|S_n\|_{\ell_\infty}
\end{align*}
and
\begin{align*}
\|\hat{\Xi}_n\hat{\mathfrak{C}}_n\hat{\Xi}_n^\top-\Xi_n\mathfrak{C}_n\Xi_n^\top\|_{\ell_\infty}
&\leq16\left\{
\|\hat{\bs{X}}_n\|_{\ell_\infty}^2\|\hat{\mathfrak{C}}_n-\mathfrak{C}_n\|_{\ell_\infty}
+\left(\|\hat{\bs{X}}_n\|_{\ell_\infty}+\|\bs{X}_n\|_{\ell_\infty}\right)\|\mathfrak{C}_n\|_{\ell_\infty}\|\hat{\bs{X}}_n-\bs{X}_n\|_{\ell_\infty}
\right\},
\end{align*}
\eqref{factor-aim2}--\eqref{factor-aim3} follow from the results remarked above. Thus we complete the proof. \hfill\qed

\subsection{Proof of Corollary \ref{coro:testing}}

By construction both the Bonferroni-Holm and Romano-Wolf methods evidently satisfy condition \ref{monotone}. So it remains to check that they also satisfy \ref{max-quantile}. 
Since it holds that $\max_{\ell\in\mathcal{L}_n(\theta_n)}\mathsf{T}_n^\ell=\max_{\ell\in\mathcal{L}_n(\theta_n)}\max_{\lambda\in\Lambda^{\ell}_n}|\tilde{T}_n^\lambda|$, Proposition \ref{prop:factor-test} yields
\[
P\left(\max_{\ell\in\mathcal{L}_n(\theta_n)}\mathsf{T}_n^\ell >c_n^{\mathcal{L}_n(\theta_n)}(1-\alpha)\right)
-P\left(\max_{\ell\in\mathcal{L}_n(\theta_n)}\max_{k\in\mathcal{K}^\ell_n}\left|\tilde{\zeta}_n^k\right| >c_n^{\mathcal{L}_n(\theta_n)}(1-\alpha)\right)\to0
\]
as $n\to\infty$, where $\mathcal{K}^\ell_n:=\{(i-1)\ul{d}+j:(i,j)\in\Lambda^{\ell}_n\}$ and $\tilde{\zeta}_n:=\bs{X}_n\mathfrak{C}_n^{1/2}\zeta_n$. 
Now if we use the Bonferroni-Holm method, we have
\begin{align*}
P\left(\max_{\ell\in\mathcal{L}_n(\theta_n)}\max_{k\in\mathcal{K}^\ell_n}\left|\tilde{\zeta}_n^k\right| >c_n^{\mathcal{L}_n(\theta_n)}(1-\alpha)\right)
\leq\sum_{\ell\in\mathcal{L}_n(\theta_n)}\sum_{k\in\mathcal{K}^\ell_n}P\left(\left|\tilde{\zeta}_n^k\right| >q_{N(0,1)}\left(1-\frac{\alpha}{2\#[\bigcup_{\ell\in\mathcal{L}_n(\theta_n)}\mathcal{K}^\ell_n]}\right)\right)
=\alpha,
\end{align*} 
\tcr{so} condition \ref{max-quantile} is satisfied. 
Meanwhile, if we use the Romano-Wolf method, Propositions \ref{prop:quantile} and \ref{prop:factor-test} yield
\[
P\left(\max_{\ell\in\mathcal{L}_n(\theta_n)}\max_{k\in\mathcal{K}^\ell_n}\left|\tilde{\zeta}_n^k\right| >c_n^{\mathcal{L}_n(\theta_n)}(1-\alpha)\right)\to\alpha
\]
as $n\to\infty$, \tcr{so} condition \ref{max-quantile} is satisfied. Thus we complete the proof.\hfill\qed

{
\section*{Acknowledgements}
\addcontentsline{toc}{section}{Acknowledgements}

The author wishes to thank the associate editor and the referee for their careful reading and valuable comments that substantially improved the original version of this paper. 
This work was supported by JST, CREST and JSPS KAKENHI Grant Numbers JP16K17105, JP17H01100, JP18H00836.
}

{\small
\renewcommand*{\baselinestretch}{1}\selectfont
\addcontentsline{toc}{section}{References}
\bibliography{base}
}

\end{document}